\documentclass{article}

\usepackage{amsmath,amssymb,amsthm}
\usepackage[utf8]{inputenc}
\usepackage[T1]{fontenc}
\usepackage[a4paper]{geometry}
\usepackage{url}
\usepackage{datetime}
\usepackage{yhmath}
\usepackage{xcolor}

\usepackage{authblk}
\title{A Trefftz Continuous Galerkin method\\ for Helmholtz problems}
\author[1]{Nicola Galante\footnote{nicola.galante@inria.fr}}
\author[2]{Bruno Després\footnote{bruno.despres@inria.fr}}
\author[1]{Emile Parolin\footnote{emile.parolin@inria.fr}}

\affil[1]{Sorbonne Université, Université Paris Cité, CNRS, INRIA, Laboratoire Jacques-Louis Lions, LJLL, EPC ALPINES, Paris, F-75005, France}
\affil[2]{Sorbonne Université, Université Paris Cité, CNRS, INRIA, Laboratoire Jacques-Louis Lions, LJLL, Laboratoire de Probabilités, Statistique et Modélisation, LPSM, EPC MEGAVOLT, Paris, F-75005, France}

\usepackage{hyperref}
\hypersetup{
    bookmarksopen=true,
    colorlinks=true,
    linkcolor=blue,
    citecolor=blue,
    allcolors = {blue},
}

\usepackage{pgfplots}
\pgfplotsset{compat=newest}
\usepackage{caption,subcaption}

\newtheorem{theorem}{Theorem}[section]
\newtheorem{proposition}[theorem]{Proposition}
\newtheorem{lemma}[theorem]{Lemma}
\newtheorem{corollary}[theorem]{Corollary}
\newtheorem{definition}[theorem]{Definition}
\newtheorem{assumption}[theorem]{Assumption}
\newtheorem{remark}[theorem]{Remark}

\begin{document}

\maketitle

\begin{abstract}

This work introduces a novel Trefftz Continuous Galerkin (TCG) method for 2D Helm\-holtz problems based on evanescent plane waves (EPWs).
We construct a new globally-conforming discrete space, departing from standard discontinuous Trefftz formulations, and investigate its approximation properties, providing wavenumber-explicit best-approximation error estimates.
The mesh is defined by intersecting the domain with a Cartesian grid, and the basis functions are continuous in the whole computational domain, compactly supported, and can be expressed as simple linear combinations of EPWs within each element. This ensures they remain local solutions to the Helmholtz equation and allows the system matrix to be assembled in closed form for polygonal domains.
The discrete space provides stable approximations with bounded coefficients and spectral accuracy for analytic Helmholtz solutions.
The approximation error is proved to decay exponentially both at a fixed frequency, with respect to the discretization parameters, and along suitable sequences of increasing wavenumbers, with the number of degrees of freedom scaling linearly with the frequency.
Numerical results confirm these theoretical estimates for the full Galerkin error.

\end{abstract}



\clearpage

\tableofcontents

\section{Introduction}

Wave propagation problems pose significant numerical challenges, especially in the high-frequency regime. In the 2D acoustic setting, a prototypical model is the \emph{homogeneous Helmholtz equation}
\begin{equation} \label{eq:helmholtz_equation}
-\Delta u - \kappa^2 u = 0, \qquad \quad \text{in a bounded Lipschitz domain } \Omega \subset \mathbb{R}^2,
\end{equation}
with wavenumber $\kappa > 0$. As $\kappa$ increases, solutions become highly oscillatory and conventional discretizations require very fine resolution. When the wavelength $\lambda = 2\pi/\kappa$ is small compared to the size of $\Omega$, standard $h$-version finite element methods (FEM) demand an excessive number of degrees of freedom (DOFs) and suffer from dispersion errors \cite{Babuska2000}; for instance with linear elements the $H^1$-numerical error scales like $\sim \kappa^3h^2$ \cite{Ihlenburg1995}. Several strategies have been proposed to mitigate this difficulty, among which high-order polynomial methods such as $hp$-FEM \cite{Ihlenburg1997,Ainsworth2004,MelenkSauter2010} and Trefftz methods \cite{Hiptmair2016} are the most prominent; see \cite{Lieu2016} for a comparison.

\paragraph{Plane wave-based Trefftz methods.}

\emph{Trefftz methods} offer an attractive alternative by building discrete spaces from local solutions of the Helmholtz equation itself. This design naturally incorporates oscillatory behavior into the approximation, thereby reducing dispersion and achieving high accuracy with substantially fewer DOFs~\cite{Gittelson2014}.
Notable examples include the Ultra Weak Variational Formulation (UWVF)~\cite{Cessenat1998}, Trefftz Discontinuous Galerkin (TDG) methods~\cite{Gittelson2014,Gittelson2009}, the Variational Theory of Complex Rays (VTCR) \cite{Riou2011}, the Discontinuous Enrichment Method (DEM) \cite{Farhat2001}, the Fokas method \cite{Davis2014,Spence2015}, and the Wave-Based Method (WBM)~\cite{Deckers2014,Desmet1998}.

A common choice of basis functions in Trefftz methods are \emph{propagative plane waves} (PPWs), defined as $\mathbf{x} \mapsto e^{\imath  \kappa \mathbf{d} \cdot \mathbf{x}}$ with $\mathbf{d}\in \mathbb{S}^1$.
An important advantage of PPW-based formulations is that integration over flat subdomains can be carried out in closed form with $\kappa$-independent effort \cite[Sec.\ 4.1]{Hiptmair2016}, avoiding numerical quadrature and contributing to the overall efficiency and accuracy of the method.
While PPWs are highly effective in theory, offering local spectral convergence as the number of basis functions increases \cite{Melenk1995,Moiola2011}, their practical use is hampered by severe ill-conditioning of the discrete linear systems \cite{Barucq2021,Barucq2024,huybrechs2019,Luostari2013}. This instability stems from the nature of PPW expansions: representing high-frequency components necessarily requires large coefficients and cancellation, causing stagnation due to round-off errors~\cite[Th.\ 4.3]{Parolin2023}. Indeed, recent insights from approximation theory~\cite{Adcock2019,Adcock2020} have clarified that, in the presence of ill-conditioning, controlling the coefficient norm, not just the best-approximation error, is crucial for numerical stability.

A promising strategy to overcome these limitations is to enrich the approximation space with \emph{evanescent plane waves} (EPWs). EPWs share the exponential form of PPWs, $\mathbf{x} \mapsto e^{\imath \kappa \mathbf{d} \cdot \mathbf{x}}$, but use complex-valued direction vectors $\mathbf{d} \in \mathbb{C}^2$ satisfying $\mathbf{d} \cdot \mathbf{d} = 1$. They combine oscillatory behavior with exponential decay while preserving closed-form integration at a $\kappa$-independent cost. Recent studies provide numerical evidence that EPW spaces yield accurate, stable approximations with bounded coefficients~\cite{Parolin2023,Robert2024,Galante2024}. EPWs also feature in WBM, DEM, and Fokas methods. Nevertheless, a key challenge remains the construction of EPW sets that provably ensure both accuracy and stability, as, to the best of our knowledge, no rigorous convergence or coefficient-norm estimates are currently available in a Trefftz-like setting.

\paragraph{A new conforming EPW-based Trefftz space.}

We introduce a new Trefftz space for the numerical approximation of the Helmholtz equation \eqref{eq:helmholtz_equation}, built on a mesh obtained by intersecting the domain $\Omega$ with a Cartesian grid $\mathcal{T}_h$ of mesh size $h>0$.

The construction of the Trefftz basis functions starts from local modes defined on a reference rectangle. An $L^2$-Hilbert basis is chosen on a single edge, extended by zero to the remaining edges, and then lifted to the rectangle by solving the Helmholtz–Dirichlet boundary value problem, so that each resulting function has trace supported only on the chosen edge. This construction resembles that of the WBM method in \cite[Sec.\ 2.5.2]{Desmet1998}, where basis functions are obtained in the same fashion by considering a Helmholtz–Neumann problem on a rectangular geometry. These modes \cite[Eq.\ (2.85)]{Desmet1998} solve a homogeneous mixed problem with impedance/sound-hard  conditions on opposite faces. In both approaches, the functions arise from separation of variables and admit simple representations as linear combinations of EPWs. On the other hand, unlike the WBM basis, the single-edge Helmholtz modes used here enjoy several orthogonality properties.

From these local modes, two complementary families of global basis functions are constructed, associated with the edges and nodes of $\mathcal{T}_h$.
For each edge, a family of basis functions is defined, supported on the two adjacent cells and extended by zero elsewhere.
On these cells, each function coincides with a reference mode -- oriented so that its nonzero trace lies along the edge itself -- such that it is continuous across the edge, and hence $\Omega$.
Node functions are defined analogously, on a coarser mesh obtained by doubling the mesh size in only one coordinate direction, and only for edges aligned with this coarsening direction.
The resulting discrete space, spanned by $N_{\mathbf{e}}$ edge and $N_{\mathbf{n}}$ node basis functions, is denoted $V_{\mathbf{N}}(\mathcal{T}_h)$, with $\mathbf{N}=(N_{\mathbf{e}},N_{\mathbf{n}})$.

As a consequence, $V_{\mathbf{N}}(\mathcal{T}_h)$ defines a globally $H_\kappa^1(\Omega)$-conforming Trefftz space, in contrast to standard Trefftz methods, which typically employ discontinuous variational formulations. This structure allows for a continuous Galerkin discretization. Since $V_{\mathbf{N}}(\mathcal{T}_h)$ is redundant by construction, a least-squares Petrov–Galerkin formulation with a regularized SVD is adopted, in the spirit of \cite{Adcock2019,Adcock2020,Parolin2023,Galante2024}. This approach ensures stability and robustness despite its higher computational cost.
Numerical results for the full Galerkin error align with the derived best-approximation estimates.

\paragraph{Main approximation results.}

We establish wavenumber-explicit stability and accuracy estimates for the discrete space $V_{\mathbf{N}}(\mathcal{T}_h)$ on domains $\Omega$ tessellated by a Cartesian grid $\mathcal{T}_h$.
Specifically, we provide best-approximation error bounds in standard $\kappa$-dependent Sobolev norms -- see Definition~\ref{def: kappa-dependent Sobolev norms} -- in terms of the numbers of DOFs, $N_{\mathbf{e}}$ and $N_{\mathbf{n}}$, namely
\emph{
\begin{itemize}
    \item \emph{(Algebraic convergence in $N_{\mathbf{e}}$ -- analytic solutions).} For any $N_{\mathbf{n}}\geq 1$, there exists an explicit constant $C>0$, independent of $N_{\mathbf{e}}$, such that, for any Helmholtz solution $u$ analytic on $\overline{\Omega}$,
    \begin{equation*}
        \inf_{v \in V_{\mathbf{N}}(\mathcal{T}_h)}\|u-v\|_{H_{\kappa}^1(\Omega)}\leq C N_{\mathbf{e}}^{-(2N_{\mathbf{n}}-1/2)}\|u\|_{H_\kappa^{2N_{\mathbf{n}}+1}(\Omega)}, \qquad \quad N_{\mathbf{e}}\gtrsim\kappa h.
    \end{equation*}
    \item \emph{(Geometric convergence in $N_{\mathbf{n}}$ -- analytic solutions).} If $N_{\mathbf{e}}$ grows linearly with $N_{\mathbf{n}}$, with a sufficiently large factor, then there exist explicit constants $C>0$ and $\tau\in(0,1)$, independent of $N_{\mathbf{e}}$ and $N_{\mathbf{n}}$, such that, for any Helmholtz solution $u$ analytic on $\overline{\Omega}$,
    \begin{equation*}
        \inf_{v \in V_{\mathbf{N}}(\mathcal{T}_h)}\|u-v\|_{H_{\kappa}^1(\Omega)}\leq C N_{\mathbf{n}}^{7/2}\,\tau^{2N_{\mathbf{n}}-1/2}\|u\|_{H_\kappa^{2N_{\mathbf{n}}+1}(\Omega)}, \qquad  \quad N_{\mathbf{n}}\gtrsim\kappa h.
    \end{equation*}
    \item \emph{(Square root convergence in $N_{\mathbf{e}}$ -- solutions in $H^2_\kappa(\Omega)$).} There exists an explicit constant $C>0$, independent of $N_{\mathbf{e}}$, such that, for any Helmholtz solution $u \in H^2_\kappa(\Omega)$,
    \begin{equation*}
        \inf_{v \in V_{\mathbf{N}}(\mathcal{T}_h)}\|u-v\|_{H_{\kappa}^1(\Omega)}\leq C N_{\mathbf{e}}^{-1/2}\|u\|_{H_\kappa^{2}(\Omega)}, \qquad  \quad N_{\mathbf{e}}\gtrsim\kappa h.
    \end{equation*}
\end{itemize}
}
In addition, we establish several stability results. Particularly, we show that the coefficient norms in the discrete representation are uniformly bounded with respect to $N_{\mathbf{e}}$ and grow at most quadratically in $N_{\mathbf{n}}$; see Section \ref{sec: Stability with respect to edge and node parameters}.

The analysis also addresses the high-frequency regime, where the method shows its strengths:
\emph{
\begin{itemize}
\item \emph{(High-frequency regime -- analytic solutions).} Under suitable shape-regularity assumptions and assuming that the parameters $N_{\mathbf{e}}$ and $N_{\mathbf{n}}$ scale linearly with $\kappa h$ with sufficiently large factors, there exists an unbounded set $\mathcal{K} \subset (0,+\infty)$ and explicit constants $C, \eta > 0$, independent of $\kappa h$, such that, for any Helmholtz solution $u$ analytic on $\overline{\Omega}$,
\begin{equation*}
        \inf_{v \in V_{\mathbf{N}}(\mathcal{T}_h)}\|u-v\|_{H_{\kappa}^1(\Omega)}\leq C|\mathcal{T}_h|(\kappa h)^5 e^{-\eta \kappa h}\|u\|_{H_\kappa^{2N_{\mathbf{n}}+1}(\Omega)}, \qquad \quad 1\ll\kappa h \in \mathcal{K}.
\end{equation*}
\end{itemize}
}
Equivalently, for any fixed mesh $\mathcal{T}_h$, there exists an unbounded sequence of wavenumbers along which \textit{the best-approximation error decays exponentially}.

On the one hand, this indicates that achieving a $\kappa$-uniform approximation error requires only the number of DOFs to scale linearly with the wavenumber.
In comparison, in 2D, standard $hp$-FEM methods typically require $\text{\#DOF} \sim \kappa^2$ to maintain constant accuracy~\cite{MelenkSauter2010,MelenkSauter2011}.
An intermediate scaling is given by the Gaussian Coherent State (GCS) method~\cite{Chaumont-Frelet2024}, where $\text{\#DOF} \sim \kappa^{3/2}$.
On the other hand, the approximation does not merely maintain accuracy: the error actually decays exponentially for analytic Helmholtz solutions. Although the previous result is rigorously proved only for suitable diverging sequences that avoid resonance in the mesh cells, numerical evidence suggests that this behavior extends beyond such a theoretical restriction.

\paragraph{Outline of the paper.}
In Section \ref{sec: Local edge-based Trefftz space}, we introduce a local Trefftz space of Helmholtz solutions on a reference rectangular cell, focusing on functions with trace supported on a single edge. We prove that a relevant set of Helmholtz solutions lies in the closure of the span of these basis functions, which leads to a local exponential approximation result.

In Section \ref{sec: Global edge-based Trefftz space}, we extend this construction globally to a domain partitioned into rectangular cells. By defining $C^0(\overline{\Omega})$-continuous edge-based basis functions, we build a $H^1_\kappa(\Omega)$-conforming Trefftz space. We then prove global approximation properties, but show that this space alone is insufficient to approximate all Helmholtz solutions, motivating further enrichment.

To address these limitations, Section \ref{sec: Enrichment with node-based functions} introduces a complementary $H^1_\kappa(\Omega)$-conforming Tref\-ftz space associated with the mesh nodes. These additional functions are constructed to capture solution components that cannot be expressed by edge contributions alone. We also establish key properties of this nodal space that will be needed to derive best-approximation results.

In Section \ref{sec: Combined edge-and-node Trefftz space}, we analyze the accuracy and stability of the full discrete space obtained by combining edge- and node-based functions. We derive convergence estimates confirming its effectiveness in stably approximating Helmholtz solutions, thus validating the enrichment strategy.

In Section \ref{sec: Discrete scheme and numerical results}, we introduce the discrete scheme used to compute numerical solutions within the proposed Trefftz space. A regularized least-squares Petrov–Galerkin method is employed to address the redundancy of the approximation space and ensure stability.
The section conclude with some numerical experiments that validate the method and support the theoretical analysis.

Throughout the paper, we use the notation $\mathbb{N} := \{0, 1, 2, \dots\}$ for the set of non-negative integers, and $\mathbb{N}^* := \mathbb{N} \setminus \{0\}$ for the set of positive integers.

\section{Local edge-based Trefftz space} \label{sec: Local edge-based Trefftz space}

We introduce a local Trefftz space of Helmholtz solutions on a reference rectangular cell, with nonzero trace restricted to a single boundary edge. Using norm estimates for the basis functions, we show that their span is rich enough to approximate a wide class of Helmholtz solutions. This leads to a first local approximation result, establishing exponential convergence.

\subsection{Single-edge Helmholtz mode}

To ground the analysis, consider a reference rectangle $\widehat{K}$ and  its top edge $\hat{\mathbf{s}}$, defined by
\begin{equation} \label{eq: reference cell and edge}
    \widehat{K}:=(0,\widehat{h}_1)\times (0,\widehat{h}_2), \qquad \hat{\mathbf{s}} := [0,\widehat{h}_1] \times \{\widehat{h}_2\}, \qquad \widehat{h}_1,\widehat{h}_2>0.
\end{equation}
Moreover, we introduce the set of $\Delta$-Neumann eigenvalues on $\widehat{K}$, denoted by
\begin{equation} \label{eq: set Neumann reference}
    \sigma(\widehat{h}_1,\widehat{h}_2):=\left\{ \left( \frac{n\pi}{\widehat{h}_1} \right)^2  + \left( \frac{m\pi}{\widehat{h}_2} \right)^2 \right\}_{n, m \in \mathbb{N}}.
\end{equation}

We adopt the following standing assumption for Section \ref{sec: Local edge-based Trefftz space}:

\begin{assumption} \label{A0}
    $\kappa^2$ is assumed not to belong to the set $\sigma(\widehat{h}_1,\widehat{h}_2)$.
\end{assumption}

The set $\sigma(\widehat{h}_1,\widehat{h}_2)$ in \eqref{eq: set Neumann reference} includes, in particular, all $\Delta$-Dirichlet eigenvalues, corresponding to $n, m \in \mathbb{N}^*$. Hence, Assumption~\ref{A0} ensures the well-posedness of the Helmholtz–Dirichlet problem on $\widehat{K}$; see for instance \textup{\cite[Sec.~6.1]{Spence2015}}.

\begin{remark}
For any given wavenumber $\kappa$, Assumption \textup{\ref{A0}} can always be fulfilled by a suitable choice of the numerical parameters $(\widehat{h}_1,\widehat{h}_2)$.
\end{remark}

To the edge $\hat{\mathbf{s}}$ we associate the following family of solutions to the Helmholtz equation.

\begin{definition}[Single-edge Helmholtz mode]
For any $n \in \mathbb{N}^*$, we define
\begin{equation} \label{eq:edge_functions}
    \widehat{\varphi}_n(x,y; \widehat{h}_1,\widehat{h}_2) :=
    \sin\left( \kappa \widehat{\nu}_n x\right) \, 
    \frac{\sin\left(  y\kappa \sqrt{1 - \widehat{\nu}_n^2} \right)}{\sin\left(  \widehat{h}_2\kappa  \sqrt{1 - \widehat{\nu}_n^2} \right)}, \qquad (x,y)\in \widehat{K}, \qquad \text{where} \qquad \widehat{\nu}_n:=\frac{n\pi}{\kappa \widehat{h}_1},
\end{equation}
and $\sqrt{\,\cdot\,}$ denotes the principal square root.
\end{definition}

Assumption \ref{A0} guarantees that these functions are well defined, while the normalization ensures that the amplitude of each $\widehat{\varphi}_n$ on $\hat{\mathbf{s}}$ is uniformly controlled.
The real or imaginary nature of $\sqrt{1 - \widehat{\nu}_n^2}$ in \eqref{eq:edge_functions} governs the qualitative behavior of the functions $\widehat{\varphi}_n$:
\begin{itemize}
    \item If $\widehat{\nu}_n < 1$, the square root is real, and the function is \emph{propagative}, exhibiting oscillations in both the horizontal and vertical directions.
    \item If $\widehat{\nu}_n > 1$, the square root is imaginary, and the function is \emph{evanescent}, oscillating in the horizontal directions, and decaying exponentially in the vertical direction for decreasing $y$.
\end{itemize}
This behavior is illustrated in Figure \ref{fig: local basis}, where the left panel shows a propagative mode and the right panel shows an evanescent one.

\begin{figure}
\centering
\includegraphics[trim=120 220 120 180,clip,width=.26\textwidth]{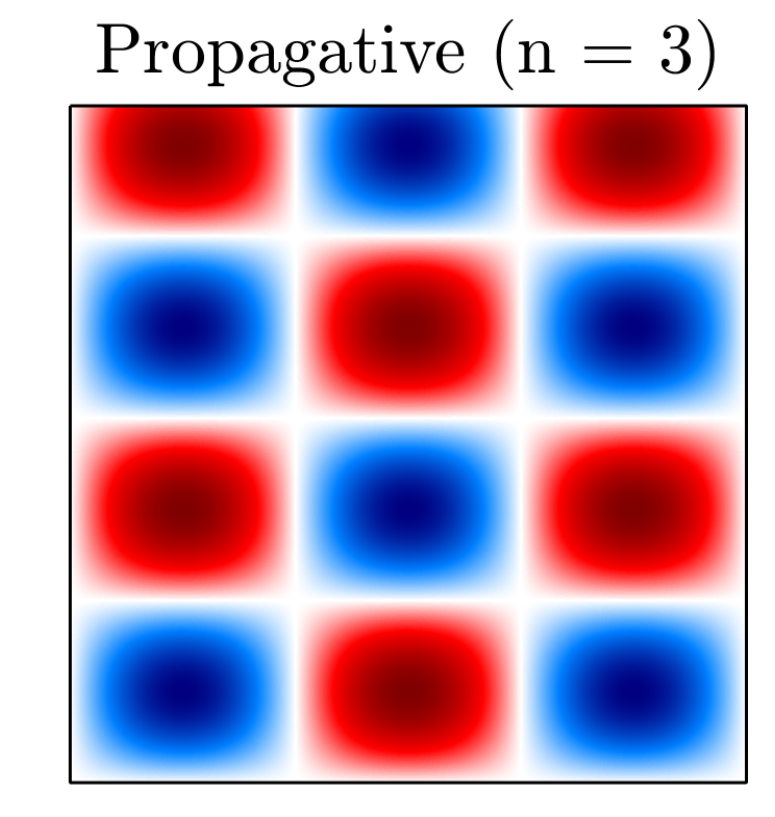}
\hspace{.2\textwidth}
\includegraphics[trim=120 220 120 180,clip,width=.26\textwidth]{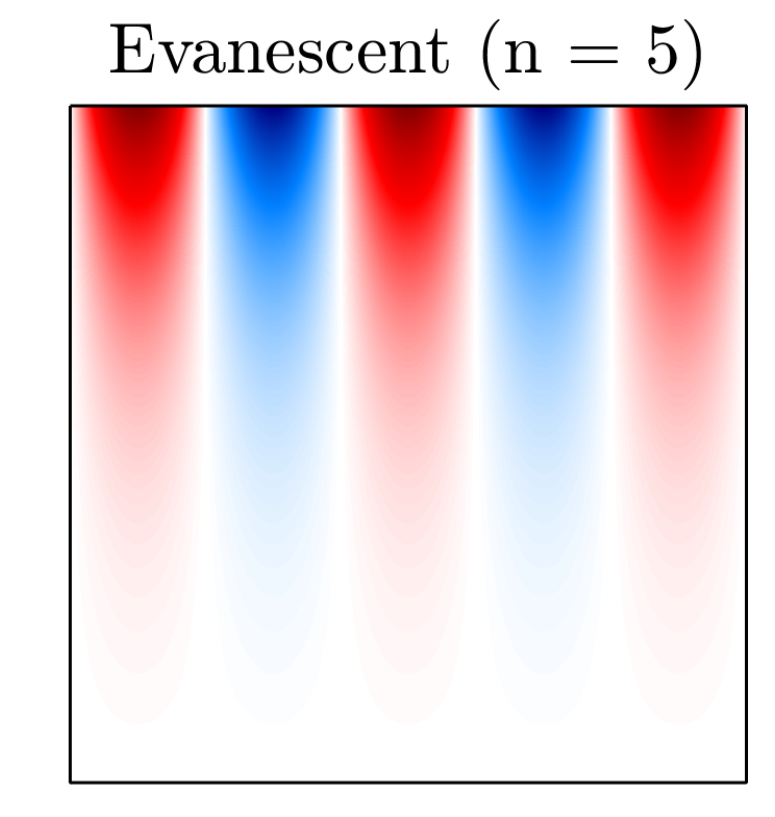}
\caption{Basis functions $\widehat{\varphi}_n$ associated with the top edge $\hat{\mathbf{s}}$ of $\widehat{K}$;
$\kappa=15$, $\widehat{h}_1=\widehat{h}_2=1$.}
\label{fig: local basis}
\end{figure}

\subsection{Fundamental properties}

Since we aim for wavenumber-explicit error bounds, we use $\kappa$-dependent Sobolev norms.

\begin{definition}[$\kappa$-dependent Sobolev norms] \label{def: kappa-dependent Sobolev norms}
Let $\Omega \subset \mathbb{R}^d$ be a bounded Lipschitz domain, with $d \in \{1,2\}$, and let $M \in \mathbb{N}$. We define the $\kappa$-weighted Sobolev inner product and norm by
\begin{equation*}
    \left(u,v\right)_{H^M_{\kappa}(\Omega)}:=\sum_{|\boldsymbol{\alpha}|\leq M}\kappa^{-2|\boldsymbol{\alpha}|}\left(\partial^{\boldsymbol{\alpha}}u,\partial^{\boldsymbol{\alpha}}v\right)_{L^2(\Omega)}, \qquad \|u\|^2_{H^M_{\kappa}(\Omega)}:=\left(u,u\right)_{H^M_{\kappa}(\Omega)},
\end{equation*}
where $\boldsymbol{\alpha} \in \mathbb{N}^d$ is a multi-index.  
The Sobolev space is $H^M_\kappa(\Omega) := \left\{ u \in L^2(\Omega) : \|u\|_{H^M_\kappa(\Omega)} < \infty \right\}$.
\end{definition}

We first derive orthogonality properties and norm bounds for the single-edge modes $\widehat{\varphi}_n$.

\begin{lemma} \label{lem: orthogonality L2}
The two families
    \begin{equation} \label{eq: dirichlet and neumann on s}
        \left\{ \sqrt{2/\widehat{h}_1} \sin\left( \kappa \widehat{\nu}_n t\right) \right\}_{n \in \mathbb{N}^*}
\qquad \text{and} \qquad
\left\{ 1/\sqrt{\widehat{h}_1}\right\}\cup \left\{\sqrt{2/\widehat{h}_1} \cos\left( \kappa \widehat{\nu}_n t\right) \right\}_{n \in \mathbb{N}^*}
    \end{equation}
    are Hilbert bases of $L^2(\hat{\mathbf{s}})$. Besides, the former coincides (up to normalization) with 
    $\{ \widehat{\varphi}_n|_{\hat{\mathbf{s}}}\}_{n \in \mathbb{N}^*}$.
\end{lemma}
\begin{proof}
The result easily follows as the two families in \eqref{eq: dirichlet and neumann on s} are the Dirichlet and Neumann eigenfunctions of $-\Delta$ on $(0,\widehat{h}_1)$, respectively; see \cite[Sec.\ 6.5.1]{Evans2010}.
\end{proof}

\begin{lemma} \label{lem:orthogonality lemma}
    Let $M \in \mathbb{N}$. The functions $\{\widehat{\varphi}_n\}_{n \in \mathbb{N}^*}$ are orthogonal in $H^{M}_{\kappa}(\widehat{K})$. Besides,
    \begin{align} 
        \|\widehat{\varphi}_n\|^2_{H^M_{\kappa}(\widehat{K})}&\leq\frac{\widehat{h}_1\pi^2(M+1)^2}{8\kappa^2 \widehat{h}_2\widehat{\mathrm{d}}^2}, && \qquad  \text{if}\quad \widehat{\nu}_n<1, \label{eq: bound phi_norm}\\
        \|\widehat{\varphi}_n\|^2_{H^M_{\kappa}(\widehat{K})}&\leq\frac{\widehat{h}_1(M+1)^2}{2\kappa \tanh(\kappa \widehat{h}_2)}\,\frac{\widehat{\nu}_n^{2M}}{\sqrt{\widehat{\nu}_n^2-1}}, && \qquad  \text{if}\quad \widehat{\nu}_n>1, \label{eq: bound phi_norm2}
    \end{align}
    where $\widehat{\nu}_n$ is defined in \eqref{eq:edge_functions}, for any $n \in \mathbb{N}^*$, and
    \begin{equation} \label{eq: definition d_m}
        \widehat{\mathrm{d}}:=\inf_{\substack{m \in \mathbb{N},n \in \mathbb{N}^*:\\\widehat{\nu}_n<1}}\left|\sqrt{1-\widehat{\nu}_n^2}-\frac{m\pi}{\kappa \widehat{h}_2}\right|>0.
    \end{equation}
\end{lemma}
\begin{proof}
    Since $\widehat{\varphi}_n$ are separable functions for any $n \in \mathbb{N}^*$, let us introduce the notations
    \begin{equation*}
        \widehat{\xi}_{n,j}(t):=\sin\left(\kappa \widehat{\nu}_n t +\frac{j\pi}{2}\right), \qquad \widehat{\chi}_{n,j}(s):=\frac{\sin\left(  \kappa s  \sqrt{1-\widehat{\nu}_n^2}+\frac{j\pi}{2} \right)}{\sin\left(\kappa  \widehat{h}_2  \sqrt{1 - \widehat{\nu}_n^2} \right)},
    \end{equation*}
    with $0\leq j\leq M$, $t \in [0,\widehat{h}_1]$ and $s \in [0,\widehat{h}_2]$. It easy to check that
    \begin{equation} \label{eq: step0}
        \left(\widehat{\xi}_{n,j},\widehat{\xi}_{m,j}\right)_{L^2(0,\widehat{h}_1)}=\frac{\widehat{h}_1}{2}\delta_{nm}, \qquad \|\widehat{\chi}_{n,j}\|^2_{L^2(0,\widehat{h}_2)}=\frac{\widehat{h}_2}{2}\left|\frac{\cot(\kappa \widehat{h}_2\sqrt{1-\widehat{\nu}_n^2})}{\kappa \widehat{h}_2\sqrt{1-\widehat{\nu}_n^2}}-\frac{\cos(j\pi)}{\sin^2(\kappa \widehat{h}_2\sqrt{1-\widehat{\nu}_n^2})}\right|,
    \end{equation}
    where the orthogonality of the functions $\{\widehat{\xi}_{n,j}\}_{n \in \mathbb{N}^*}$ follows directly from Lemma \ref{lem: orthogonality L2}.
    Then,
    \begin{equation*}
        \left(\widehat{\varphi}_n,\widehat{\varphi}_m\right)_{H^M_{\kappa}(\widehat{K})}=\!\!\!\sum_{|\boldsymbol{\alpha}|\leq M}\kappa^{-2|\boldsymbol{\alpha}|}\left(\partial^{\boldsymbol{\alpha}}\widehat{\varphi}_n,\partial^{\boldsymbol{\alpha}}\widehat{\varphi}_m\right)_{L^2(\widehat{K})}\!\!=\sum_{\ell=0}^M\sum_{j=0}^\ell\kappa^{-2\ell}\left(\partial_1^j\partial_2^{\ell-j}\widehat{\varphi}_n,\partial_1^j\partial_2^{\ell-j}\widehat{\varphi}_m\right)_{L^2(\widehat{K})}.
    \end{equation*}
    Since $\partial_1^{j}\partial_2^{\ell-j}\widehat{\varphi}_n\!=\!\widehat{\nu}_n^{2j}(1-\widehat{\nu}_n^2)^{\ell-j}\widehat{\xi}_{n,j}\widehat{\chi}_{n,\ell-j}$ for any $n \in \mathbb{N}^*$, $0\leq j\leq \ell \leq M$, and from \eqref{eq: step0}, one has
    \begin{align}
        \left(\widehat{\varphi}_n,\widehat{\varphi}_m\right)_{H^M_{\kappa}(\widehat{K})}&=\sum_{\ell=0}^M\sum_{j=0}^\ell\widehat{\nu}_n^{2j}|1-\widehat{\nu}_n^2|^{\ell-j}\left(\widehat{\xi}_{n,j},\widehat{\xi}_{m,j}\right)_{L^2(0,\widehat{h}_1)}\left(\widehat{\chi}_{n,\ell-j},\widehat{\chi}_{m,\ell-j}\right)_{L^2(0,\widehat{h}_2)} \nonumber\\
        &=\delta_{nm}\frac{\widehat{h}_1\widehat{h}_2}{4}\!\sum_{\ell=0}^M\sum_{j=0}^\ell\widehat{\nu}_n^{2j}|1-\widehat{\nu}_n^2|^{\ell-j}\!\left|\frac{\cot(\kappa \widehat{h}_2\sqrt{1-\widehat{\nu}_n^2})}{\kappa \widehat{h}_2\sqrt{1-\widehat{\nu}_n^2}}\!-\!\frac{\cos((\ell-j)\pi)}{\sin^2(\kappa \widehat{h}_2\sqrt{1-\widehat{\nu}_n^2})}\right|,\! \label{eq: 1step}
    \end{align}
    which proves orthogonality.
    
    We now prove the bound \eqref{eq: bound phi_norm} assuming $\widehat{\nu}_n<1$. From \eqref{eq: 1step}, one has that
    \begin{equation*}
        \|\widehat{\varphi}_n\|^2_{H^M_{\kappa}(\widehat{K})}=\frac{\widehat{h}_1\widehat{h}_2}{4}\sum_{\ell=0}^M\sum_{j=0}^\ell\widehat{\nu}_n^{2j}|1-\widehat{\nu}_n^2|^{\ell-j}\left|\frac{\cot(\kappa \widehat{h}_2\sqrt{1-\widehat{\nu}_n^2})}{\kappa \widehat{h}_2\sqrt{1-\widehat{\nu}_n^2}}-\frac{\cos((\ell-j)\pi)}{\sin^2(\kappa \widehat{h}_2\sqrt{1-\widehat{\nu}_n^2})}\right|.
    \end{equation*}
    Due to Assumption \ref{A0}, $\kappa \widehat{h}_2\sqrt{1-\widehat{\nu}_n^2}\notin \pi \mathbb{Z}$, and hence from Lemma \ref{lem: coth inequality} one obtains
    \begin{align*}
        \|\widehat{\varphi}_n\|^2_{H^M_{\kappa}(\widehat{K})}&\leq\frac{\widehat{h}_1\widehat{h}_2}{4}\sum_{\ell=0}^M\sum_{j=0}^\ell\widehat{\nu}_n^{2j}(1-\widehat{\nu}_n^2)^{\ell-j}\left(\left|\frac{\cot(\kappa \widehat{h}_2\sqrt{1-\widehat{\nu}_n^2})}{\kappa \widehat{h}_2\sqrt{1-\widehat{\nu}_n^2}}\right|+\frac{1}{\sin^2(\kappa \widehat{h}_2\sqrt{1-\widehat{\nu}_n^2})}\right)\\
        &\leq \frac{\widehat{h}_1\widehat{h}_2(M+1)^2}{2\sin^2(\kappa \widehat{h}_2\sqrt{1-\widehat{\nu}_n^2})}\leq \frac{\widehat{h}_1\widehat{h}_2\pi^2(M+1)^2}{8\min_{m \in \mathbb{N}}|\kappa \widehat{h}_2\sqrt{1-\widehat{\nu}_n^2}-m\pi|^2}.
    \end{align*}
    Finally, recalling the definition of $\widehat{\mathrm{d}}$ in \eqref{eq: definition d_m}, it follows
    \begin{equation*}
        \|\widehat{\varphi}_n\|^2_{H^M_{\kappa}(\widehat{K})}\leq \frac{\widehat{h}_1\pi^2(M+1)^2}{8\kappa^2 \widehat{h}_2\min_{m \in \mathbb{N},n \in \mathbb{N}^*: \widehat{\nu}_n<1}|\sqrt{1-\widehat{\nu}_n^2}-m\pi/\kappa\widehat{h}_2|^2}=\frac{\widehat{h}_1\pi^2(M+1)^2}{8\kappa^2 \widehat{h}_2\widehat{\mathrm{d}}^2}.
    \end{equation*}
    
    We now prove the inequality \eqref{eq: bound phi_norm2} assuming $\widehat{\nu}_n>1$.
    It follows that $\sqrt{1-\widehat{\nu}_n^2}=\imath\sqrt{\widehat{\nu}_n^2-1}$, and from \eqref{eq: 1step} one obtains:
    \begin{equation*}
        \|\widehat{\varphi}_n\|^2_{H^M_{\kappa}(\widehat{K})}=\frac{\widehat{h}_1\widehat{h}_2}{4}\sum_{\ell=0}^M\sum_{j=0}^\ell\widehat{\nu}_n^{2j}|\widehat{\nu}_n^2-1|^{\ell-j}\left|\frac{\coth(\kappa \widehat{h}_2\sqrt{\widehat{\nu}_n^2-1})}{\kappa \widehat{h}_2\sqrt{\widehat{\nu}_n^2-1}}-\frac{\cos((\ell-j)\pi)}{\sinh^2(\kappa \widehat{h}_2\sqrt{\widehat{\nu}_n^2-1})}\right|.
    \end{equation*}
    Hence,
    \begin{align*}
        \|\widehat{\varphi}_n\|^2_{H^M_{\kappa}(\widehat{K})}&\leq\frac{\widehat{h}_1\widehat{h}_2}{4}\sum_{\ell=0}^M(\ell+1)\widehat{\nu}_n^{2\ell}\left(\frac{\coth(\kappa \widehat{h}_2\sqrt{\widehat{\nu}_n^2-1})}{\kappa \widehat{h}_2\sqrt{\widehat{\nu}_n^2-1}}+\frac{1}{\sinh^2(\kappa \widehat{h}_2\sqrt{\widehat{\nu}_n^2-1})}\right)\\
        &\leq \frac{\widehat{h}_1\widehat{h}_2(M+1)^2\widehat{\nu}_n^{2M}}{4}\left(\frac{2\coth(\kappa \widehat{h}_2\sqrt{\widehat{\nu}_n^2-1})}{\kappa \widehat{h}_2\sqrt{\widehat{\nu}_n^2-1}}\right)\leq \frac{\widehat{h}_1(M+1)^2}{2\kappa \tanh(\kappa \widehat{h}_2)}\,\frac{\widehat{\nu}_n^{2M}}{\sqrt{\widehat{\nu}_n^2-1}},
    \end{align*}
    where the second inequality follows from the elementary bound $1/{\sinh^2(t)}\leq \coth(t)/{t}$, for $t>0$, and the last from the monotonicity of $\tanh$.
\end{proof}

\begin{remark}
    The quantity $\widehat{\mathrm{d}}$ defined in \eqref{eq: definition d_m} is strictly positive by Assumption \emph{\ref{A0}}, and measures the distance from $\kappa^2$ to the resonant frequencies with $n \in \mathbb{N}^*$ and $m \in \mathbb{N}$ in $\sigma(\widehat{h}_1,\widehat{h}_2)$ from \eqref{eq: set Neumann reference}.
\end{remark}

We now introduce a set of Helmholtz solutions, which will be well approximated by the single-edge Helmholtz modes $\{\widehat{\varphi}_n\}_{n \in \mathbb{N}^*}$.

\begin{definition} \label{def: local set}
    Let $N,M\in \mathbb{N}^*$ such that $N\leq \lceil M/2\rceil$.
    We define the set
    \begin{equation*}
\mathcal{S}^{N,M}_{\kappa}(\widehat{K},\hat{\mathbf{s}}) :=
\left\{ u \in H^1_\kappa(\widehat{K}) \;\middle|\;
\begin{aligned}
  & -\Delta u - \kappa^2 u = 0 && \text{in } \widehat{K}, \\
  &  u|_{\hat{\mathbf{s}}} \in H_\kappa^{M}(\hat{\mathbf{s}}), &&
     u|_{\partial \widehat{K} \setminus \hat{\mathbf{s}}} = 0, \\
  & \partial_{1}^{2n} ( u|_{\hat{\mathbf{s}}})({\partial\hat{\mathbf{s}}}) = 0, && 0 \leq n \leq N-1
\end{aligned}
\right\}.
\end{equation*}
\end{definition}

Let us first state the following technical lemma.

\begin{lemma} \label{lem: integration by parts}
    Let $N, M \in \mathbb{N}^*$ with $N \leq \lceil M/2 \rceil$, and $\mu \in \mathbb{N}$ with $\mu \leq \min(2N, M)$. Then, for any $u \in \mathcal{S}^{N,M}_{\kappa}(\widehat{K},\hat{\mathbf{s}})$, it holds
    \begin{equation} \label{eq: integration by parts}
        \left|( u,  \widehat{\varphi}_n)_{L^2(\hat{\mathbf{s}})}\right|=(\kappa \widehat{\nu}_n)^{-\mu}\left|\int_0^{\widehat{h}_1}\partial^{\mu}_1 u(t) \sin(\kappa \widehat{\nu}_n t+\mu\pi/2)\,\textup{d}t\right|, \qquad n\in \mathbb{N}^*.
    \end{equation}
\end{lemma}
\begin{proof}
    As $u \in \mathcal{S}^{N,M}_{\kappa}(\widehat{K},\hat{\mathbf{s}})$ and $\mu\leq M$, the quantities in \eqref{eq: integration by parts} are well-defined.
    Integrating by parts $\mu$ times, it follows
    \begin{align*}
        \left|( u,  \widehat{\varphi}_n)_{L^2(\hat{\mathbf{s}})}\right|&=\left|\int_0^{\widehat{h}_1}\!\!\! u(t) \sin(\kappa \widehat{\nu}_n t)\textup{d}t\right|=(\kappa \widehat{\nu}_n)^{-1}\!\left|\int_0^{\widehat{h}_1}\!\!\!\partial_1 u(t) \cos(\kappa \widehat{\nu}_n t)\textup{d}t\right|\\
        &=(\kappa \widehat{\nu}_n)^{-2}\!\left|\int_0^{\widehat{h}_1}\!\!\!\partial^2_1 u(t) \sin(\kappa \widehat{\nu}_n t)\textup{d}t\right|=\dots=(\kappa \widehat{\nu}_n)^{-\mu}\!\left|\int_0^{\widehat{h}_1}\!\!\!\partial^{\mu}_1 u(t) \sin(\kappa \widehat{\nu}_n t\!+\!\mu\pi/2)\textup{d}t\right|.
    \end{align*}
    Integration by parts is performed only up to order $\mu \leq 2N$, as higher orders introduce non-vanishing boundary terms and fail to produce additional powers of $\widehat{\nu}_n^{-1}$. \qedhere
    
\end{proof}

The next proposition shows that any function in $\mathcal{S}^{N,M}_{\kappa}(\widehat{K},\hat{\mathbf{s}})$ can be written as a linear combination of the single-edge modes, with coefficients depending only on their trace over $\widehat{\mathbf{s}}$.

\begin{proposition} \label{prop: completness}
    Let $N, M \in \mathbb{N}^*$ with $N \leq \lceil M/2 \rceil$, and $\mu \in \mathbb{N}$ with $\mu \leq \min(2N, M)$. Then,
    \begin{equation*}
        \mathcal{S}^{N,M}_{\kappa}(\widehat{K},\hat{\mathbf{s}})\subset \overline{\textup{span}\{\widehat{\varphi}_n\}_{n \in \mathbb{N}^*}}^{H^{\mu}_\kappa(\widehat{K})}.
    \end{equation*}
    Besides, for any $u \in \mathcal{S}^{N,M}_{\kappa}(\widehat{K},\hat{\mathbf{s}})$, its expansion coefficients depend only on its trace on $\hat{\mathbf{s}}$:
    \begin{equation} \label{eq:coefficient series}
        \frac{(u, \widehat{\varphi}_n)_{H_\kappa^{\mu}(\widehat{K})}}{\|\widehat{\varphi}_n\|_{H^{\mu}_\kappa(\widehat{K})}}=\frac{\|\widehat{\varphi}_n\|_{H^{\mu}_\kappa(\widehat{K})}}{\| \widehat{\varphi}_n\|^2_{L^2(\hat{\mathbf{s}})}}( u,  \widehat{\varphi}_n)_{L^2(\hat{\mathbf{s}})}, \qquad n \in \mathbb{N}^*.
    \end{equation}
\end{proposition}
\begin{proof}
    Consider $u \in \mathcal{S}^{N,M}_{\kappa}(\widehat{K},\hat{\mathbf{s}})$.
    We first prove that the series defined by the coefficients in \eqref{eq:coefficient series} converges.
    From the bound \eqref{eq: bound phi_norm2}, there exists a constant $C>0$
    independent of $n$ such that $\|\widehat{\varphi}_n\|^2_{H^\mu_\kappa(\widehat{K})}\leq C\widehat{\nu}_n^{2\mu-1}$, for any $n \in \mathbb{N}^*$. Besides, $\| \widehat{\varphi}_n\|^2_{L^2(\hat{\mathbf{s}})}=\widehat{h}_1/2$, and hence
    \begin{equation} \label{eq: s1}
        \sum_{n \in \mathbb{N}^*}\frac{\|\widehat{\varphi}_n\|^2_{H^\mu_\kappa(\widehat{K})}}{\| \widehat{\varphi}_n\|^4_{L^2(\hat{\mathbf{s}})}}\left|( u,  \widehat{\varphi}_n)_{L^2(\hat{\mathbf{s}})}\right|^2\leq \frac{4C}{\widehat{h}_1^2}\sum_{n \in \mathbb{N}^*}\widehat{\nu}_n^{2\mu-1}\left|( u,  \widehat{\varphi}_n)_{L^2(\hat{\mathbf{s}})}\right|^2.
    \end{equation}
    Therefore, due to Lemmas \ref{lem: orthogonality L2} and \ref{lem: integration by parts},
    \begin{equation} \label{eq: M norm phi_n}
        \sum_{n \in \mathbb{N}^*}\frac{\|\widehat{\varphi}_n\|^2_{H^\mu_\kappa(\widehat{K})}}{\| \widehat{\varphi}_n\|^4_{L^2(\hat{\mathbf{s}})}}\left|( u,  \widehat{\varphi}_n)_{L^2(\hat{\mathbf{s}})}\right|^2
        \!\leq\! \frac{4C}{\kappa^{2\mu}\widehat{h}_1^2}\sum_{n \in \mathbb{N}^*}\widehat{\nu}_n^{-1}\left|\int_0^{\widehat{h}_1}\!\!\!\partial^{\mu}_1 u(t) \sin \!\left(\kappa \widehat{\nu}_n t+\frac{\mu\pi}{2}\right)\textup{d}t\right|^2\!\!
        \!\!\leq\!\frac{2\kappa C}{\pi}| u|^2_{H_\kappa^\mu(\mathbf{\hat{s}})}.
    \end{equation}

    Now consider
    \begin{equation*}
        u_\mathrm{H}:=\sum_{n \in \mathbb{N}^*}\left(\frac{\|\widehat{\varphi}_n\|_{H^\mu_\kappa(\widehat{K})}}{\| \widehat{\varphi}_n\|^2_{L^2(\hat{\mathbf{s}})}}( u,  \widehat{\varphi}_n)_{L^2(\hat{\mathbf{s}})}\right)\frac{\widehat{\varphi}_n}{\|\widehat{\varphi}_n\|_{H^\mu_\kappa(\widehat{K})}},
    \end{equation*}
    which is well-defined, since, on one hand, the functions $\{\widehat{\varphi}_n\}_{n \in \mathbb{N}^*}$ are orthogonal in $H^{\mu}_\kappa(\widehat{K})$ by Lemma \ref{lem:orthogonality lemma} and, on the other hand, from \eqref{eq: M norm phi_n} it follows that
    \begin{equation*}
        u_\mathrm{H}\in \overline{\textup{span}\{\widehat{\varphi}_n\}_{n \in \mathbb{N}^*}}^{H^{\mu}_\kappa(\widehat{K})}, \qquad \text{and}\qquad \|u_{\mathrm{H}}\|^2_{H_\kappa^\mu(\widehat{K})}=\sum_{n \in \mathbb{N}^*}\frac{\|\widehat{\varphi}_n\|^2_{H^\mu_\kappa(\widehat{K})}}{\| \widehat{\varphi}_n\|^4_{L^2(\hat{\mathbf{s}})}}\left|( u,  \widehat{\varphi}_n)_{L^2(\hat{\mathbf{s}})}\right|^2<+\infty.
    \end{equation*}
    In particular, the uniqueness of the expansion coefficients implies:
    \begin{equation*}
        \frac{(u_\mathrm{H}, \widehat{\varphi}_n)_{H_\kappa^{\mu}(\widehat{K})}}{\|\widehat{\varphi}_n\|_{H^{\mu}_\kappa(\widehat{K})}}=\frac{\|\widehat{\varphi}_n\|_{H^{\mu}_\kappa(\widehat{K})}}{\| \widehat{\varphi}_n\|^2_{L^2(\hat{\mathbf{s}})}}( u,  \widehat{\varphi}_n)_{L^2(\hat{\mathbf{s}})}.
    \end{equation*}
    To conclude the proof, it remains to show that $u = u_\mathrm{H}$.
    By Lemma \ref{lem: orthogonality L2}, the set $\{ \widehat{\varphi}_n|_{\hat{\mathbf{s}}}\}_{n \in \mathbb{N}^*}$ forms a complete orthogonal system in $L^2(\hat{\mathbf{s}})$, and thus
    \begin{equation*}
         u|_{\hat{\mathbf{s}}}=\sum_{n \in \mathbb{N}^*}\frac{( u,  \widehat{\varphi}_n)_{L^2(\hat{\mathbf{s}})}}{\| \widehat{\varphi}_n\|^2_{L^2(\hat{\mathbf{s}})}} \widehat{\varphi}_n|_{\hat{\mathbf{s}}}.
    \end{equation*}
    Since $ u= \widehat{\varphi}_n=0$ on $\partial \widehat{K}\setminus\hat{\mathbf{s}}$ for any $n \in \mathbb{N}$, it follows:
    \begin{equation*}
         u_\mathrm{H}|_{\partial \widehat{K}}=\left(\sum_{n \in \mathbb{N}^*}\frac{( u,  \widehat{\varphi}_n)_{L^2(\hat{\mathbf{s}})}}{\| \widehat{\varphi}_n\|^2_{L^2(\hat{\mathbf{s}})}} \widehat{\varphi}_n\right)\Bigg\vert\rule{0pt}{5ex}_{\raisebox{-2.9ex}{$\scriptstyle \partial \widehat{K}$}}=\sum_{n \in \mathbb{N}^*}\frac{( u,  \widehat{\varphi}_n)_{L^2(\hat{\mathbf{s}})}}{\| \widehat{\varphi}_n\|^2_{L^2(\hat{\mathbf{s}})}} \widehat{\varphi}_n|_{\partial \widehat{K}}= u|_{\partial \widehat{K}}.
    \end{equation*}
    As $u_\mathrm{H}$ is a Helmholtz solution by construction and the Helmholtz--Dirichlet problem is well-posed under Assumption  \ref{A0}, we conclude that $u = u_\mathrm{H}$, as claimed.
\end{proof}

\subsection{Local approximation results}

We conclude this section by studying the approximation properties of the local Trefftz space
\begin{equation*}
    \widehat{V}_{N_{\mathbf{e}}}:=\textup{span}\{\widehat{\varphi}_n\}_{n=1}^{N_{\mathbf{e}}},
\end{equation*}
spanned by the first $N_{\mathbf{e}} \in \mathbb{N}^*$ single-edge Helmholtz modes introduced in \eqref{eq:edge_functions}.
Let us define the $H^1_\kappa(\widehat{K})$-orthogonal projection onto $\widehat{V}_{N_{\mathbf{e}}}$:
\begin{equation} \label{eq: local projection}
    \boldsymbol{\Pi}_{\hat{\mathbf{s}}} : H^1_\kappa(\widehat{K}) \rightarrow H^1_\kappa(\widehat{K}), \qquad 
    u \mapsto \sum_{n = 1}^{N_{\mathbf{e}}} \frac{(u, \widehat{\varphi}_n)_{H_\kappa^{1}(\widehat{K})}}{\|\widehat{\varphi}_n\|^2_{H^{1}_\kappa(\widehat{K})}}\widehat{\varphi}_n.
\end{equation}
\begin{proposition}[Local edge-based approximation result] \label{th: main local th_}
    Let $N, M \in \mathbb{N}^*$ such that $N \leq \lceil M/2 \rceil$, and $r, \mu \in \mathbb{N}$ with $r \leq \mu \leq \min(2N, M)$.
    Assume $N_\mathbf{e} \in \mathbb{N}^*$ such that $\widehat{\nu}_{N_\mathbf{e}}=N_\mathbf{e} \pi/(\kappa \widehat{h}_1)\geq \sqrt{2}$.
    Then, for any $u \in \mathcal{S}^{N,M}_{\kappa}(\widehat{K},\hat{\mathbf{s}})$, the best-approximation error in
    $\widehat{V}_{N_{\mathbf{e}}}$ satisfies:
    \begin{equation} \label{eq: local approximation result useful}
        \inf_{v \in \widehat{V}_{N_{\mathbf{e}}}}\|u-v\|_{H_{\kappa}^r(\widehat{K})}=\|u-\boldsymbol{\Pi}_{\hat{\mathbf{s}}}u\|_{H_{\kappa}^r(\widehat{K})}\leq  \frac{\sqrt{2}(r+1)}{\sqrt{\kappa\tanh(\kappa \widehat{h}_2)}} \widehat{\nu}^{r-(\mu+1/2)}_{N_{\mathbf{e}}}|  u|_{H_{\kappa}^{\mu}(\hat{\mathbf{s}})}.
    \end{equation}
    If moreover $u \in H_{\kappa}^{M+1}(\widehat{K})$, then:
    \begin{equation} \label{eq: local approximation result}
        \inf_{v \in \widehat{V}_{N_{\mathbf{e}}}}\|u-v\|_{H_{\kappa}^r(\widehat{K})}=\|u-\boldsymbol{\Pi}_{\hat{\mathbf{s}}}u\|_{H_{\kappa}^r(\widehat{K})}\leq \frac{2(r+1)\max(1,\kappa \widehat{h}_2)}{\sqrt{\kappa \widehat{h}_2\tanh(\kappa \widehat{h}_2)}} \widehat{\nu}^{r-(\mu+1/2)}_{N_{\mathbf{e}}}\| u\|_{H_{\kappa}^{\mu+1}(\widehat{K})}.
    \end{equation}       
\end{proposition}
\begin{proof} 
From Proposition \ref{prop: completness}, since $u \in \mathcal{S}^{N,M}_{\kappa}(\widehat{K},\hat{\mathbf{s}})$, the function $\boldsymbol{\Pi}_{\hat{\mathbf{s}}}u$ also coincides with the $H_\kappa^r(\widehat{K})$-orthogonal projection of $u$ onto $\widehat{V}_{N_{\mathbf{e}}}$, for any $r\leq \min(2N,M)$. As a result, one has that
    \begin{equation*}
        \inf_{v \in \widehat{V}_{N_{\mathbf{e}}}}\!\!\|u-v\|^2_{H_{\kappa}^r(\widehat{K})}\!=\|u-\boldsymbol{\Pi}_{\hat{\mathbf{s}}}u\|^2_{H_{\kappa}^r(\widehat{K})}\!=\!\!\!\sum_{n >N_{\mathbf{e}}}\frac{\left|(u, \widehat{\varphi}_n)_{H_\kappa^r(\widehat{K})}\right|^2}{\| \widehat{\varphi}_n\|^2_{H_\kappa^r(\widehat{K})}}
        \!=\!\!\!\sum_{n > N_{\mathbf{e}}}\frac{\|\widehat{\varphi}_n\|^2_{H^r_\kappa(\widehat{K})}}{\| \widehat{\varphi}_n\|^4_{L^2(\hat{\mathbf{s}})}}\left|( u,  \widehat{\varphi}_n)_{L^2(\hat{\mathbf{s}})}\right|^2\!\!.
    \end{equation*}
Moreover, by Lemma \ref{lem: integration by parts} and the bound \eqref{eq: bound phi_norm2} -- which applies for any $n > N_{\mathbf{e}}$, as $\widehat{\nu}_n > \widehat{\nu}_{N_{\mathbf{e}}} \geq \sqrt{2}$ -- together with the fact that $\| \widehat{\varphi}_n\|^2_{L^2(\hat{\mathbf{s}})}=\widehat{h}_1/2$, it follows
\begin{align*}
    \frac{\kappa \widehat{h}_1 \tanh(\kappa \widehat{h}_2)}{2(r+1)^2}\|u-\boldsymbol{\Pi}_{\hat{\mathbf{s}}}u\|^2_{H_{\kappa}^r(\widehat{K})}&\leq \frac{1}{\kappa^{2\mu}}\sum_{n>N_{\mathbf{e}}}\frac{\widehat{\nu}_n^{2(r-\mu)}}{\sqrt{\widehat{\nu}_n^2-1}}\left|\int_0^{\widehat{h}_1}\partial^{\mu}_1 u(t) \sin(\kappa \widehat{\nu}_n t+\mu\pi/2)\,\textup{d}t\right|^2\\
    &\leq \frac{2}{\kappa^{2\mu}}\sum_{n>N_{\mathbf{e}}}\widehat{\nu}_n^{2(r-\mu)-1}\left|\int_0^{\widehat{h}_1}\partial^{\mu}_1 u(t) \sin(\kappa \widehat{\nu}_n t+\mu\pi/2)\,\textup{d}t\right|^2\\
    &\leq \frac{2\widehat{\nu}_{N_{\mathbf{e}}}^{2(r-\mu)-1}}{\kappa^{2\mu}}\sum_{n>N_{\mathbf{e}}}\left|\int_0^{\widehat{h}_1}\partial^{\mu}_1 u(t) \sin(\kappa \widehat{\nu}_n t+\mu\pi/2)\,\textup{d}t\right|^2\\
    &\leq\widehat{h}_1\widehat{\nu}_{N_{\mathbf{e}}}^{2(r-\mu)-1}| u|^2_{H_\kappa^{\mu}(\hat{\mathbf{s}})},
\end{align*}
with the last inequality derived using Lemma \ref{lem: orthogonality L2}.
Since $u \in \mathcal{S}^{N,M}_{\kappa}(\widehat{K}, \hat{\mathbf{s}})$ with $\mu \leq M$, we have $ u|_{\hat{\mathbf{s}}} \in H_\kappa^{\mu}(\hat{\mathbf{s}})$, and integration by parts is stopped at $\mu \leq 2N$ due to non-vanishing boundary terms and lack of gain in decay rates; see Lemma \ref{lem: integration by parts}.
Hence, one obtains the first bound \eqref{eq: local approximation result useful}, and the second one \eqref{eq: local approximation result} follows from Lemma \ref{lem: trace inequality}.
\end{proof}

We conclude this section with a few remarks concerning particular cases.

\paragraph{Pre-asymptotic regime.}
The convergence predicted by estimate~\eqref{eq: local approximation result} becomes effective if the number of basis functions $N_{\mathbf{e}}$ is sufficiently large. This behavior is reflected in the condition $\widehat{\nu}_{N_{\mathbf{e}}}\geq \sqrt{2}$, which marks the onset of the asymptotic convergence regime.

\paragraph{Degeneracy as $\widehat{h}_2 \to 0$.}
As $\widehat{h}_2$ tends to zero, the domain $\widehat{K}$ degenerates into an edge. This geometric collapse is captured in the bound by the blow-up of the factor $(\kappa \widehat{h}_2\tanh(\kappa \widehat{h}_2))^{-1/2} \sim (\kappa \widehat{h}_2)^{-1}$ as $\widehat{h}_2 \to 0$. Consequently, the constant in the estimate becomes arbitrarily large, leading to deterioration in the approximation quality.

\paragraph{High-frequency regime: $\kappa \widehat{h}_1 \to +\infty$.}
When $\kappa \widehat{h}_1 \gg 1$, the number of basis functions $N_{\mathbf{e}}$ required to achieve a target relative error $\epsilon > 0$ scales linearly with $\kappa \widehat{h}_1$. More precisely, estimate~\eqref{eq: local approximation result} guarantees that
\begin{equation*}
N_{\mathbf{e}} \geq \frac{\kappa \widehat{h}_1}{\pi} \left(\frac{2(r+1)\max(1,\kappa \widehat{h}_2)}{\epsilon \sqrt{\kappa \widehat{h}_2 \tanh(\kappa \widehat{h}_2)}}\right)^{\frac{2}{\mu+1-r}} 
\quad \Longrightarrow \qquad
\frac{ \| u - \boldsymbol{\Pi}_{\hat{\mathbf{s}}} u \|_{H_{\kappa}^r(\widehat{K})} }{ \| u \|_{H_{\kappa}^{\mu+1}(\widehat{K})} } \leq \epsilon.
\end{equation*}
This illustrates that the relative error is independent of $\kappa \widehat{h}_1$ provided that the number of basis functions $N_{\mathbf{e}}$ scales linearly with $\kappa \widehat{h}_1$.

\paragraph{Low-frequency regime: $\kappa \widehat{h}_1 \to 0$.}
Note that in this regime all the basis functions $\widehat{\varphi}_n$ correspond to evanescent modes. Assuming $\kappa \widehat{h}_1 \ll 1$, it follows that $\widehat{\nu}_{N_{\mathbf{e}}} \geq \sqrt{2}$, so estimate \eqref{eq: local approximation result} applies.
Hence, choosing for instance $N_{\mathbf{e}}=1$, the approximation bound in \eqref{eq: local approximation result} simplifies to
\begin{equation*}
\inf_{v \in \widehat{V}_{N_{\mathbf{e}}}}\|u-v\|_{H_{\kappa}^r(\widehat{K})}=\| u - \boldsymbol{\Pi}_{\hat{\mathbf{s}}} u \|_{H_{\kappa}^r(\widehat{K})}
\leq \frac{2(r+1)\max(1,\kappa \widehat{h}_2)}{\sqrt{\kappa \widehat{h}_2\tanh(\kappa \widehat{h}_2)}}\, \left(\frac{\kappa \widehat{h}_1}{\pi}\right)^{\mu+1/2-r} \!\!\!\!\!\!\!\!\!\!\| u \|_{H_{\kappa}^{\mu+1}(\widehat{K})},
\end{equation*}
which shows spectral convergence in $\kappa \widehat{h}_1$.

\section{Global edge-based Trefftz space} \label{sec: Global edge-based Trefftz space}

We extend the previous construction to a domain tessellated by rectangular cells.
For each mesh edge, we introduce basis functions that coincide in the neighboring cells with the reference modes from \eqref{eq:edge_functions}, oriented so that their nonzero trace lies on the edge and vanishing elsewhere. This construction produces $C^0(\overline{\Omega})$-continuous functions and defines a $H^1_\kappa(\Omega)$-conforming Trefftz space.
The set $\mathcal{S}_\kappa^{N,M}(\widehat{K},\hat{\mathbf{s}})$ is then generalized to the global setting, and the local result in Proposition \ref{th: main local th_} is used to derive global approximation properties. However, this edge-based space alone cannot approximate all Helmholtz solutions of interest, motivating the enrichment introduced in Section \ref{sec: Enrichment with node-based functions}.

\subsection{Construction of the edge-based Trefftz space}

Consider a domain $\Omega \subset \mathbb{R}^2$ that can be tessellated using a mesh composed of uniform rectangular open cells of fixed dimensions $h_1,h_2>0$.
We define the mesh size and shape parameter as
\begin{equation} \label{eq: shape regularity}
h := \max(h_1, h_2), \qquad \rho := \frac{h}{\min(h_1, h_2)} \geq 1.
\end{equation}
Here, $h$ denotes the length of the longest side of each cell, while $\rho$ quantifies the aspect ratio. In particular, $\rho = 1$ for square cells.
Next, we define the mesh, its skeleton, and its set of nodes as
\begin{equation*}
\mathcal{T}_h := \bigcup {K}, \qquad
\Sigma_h := \bigcup {\mathbf{s}}, \qquad
\mathcal{N}_h := \bigcup {\mathbf{p}},
\end{equation*}
where $K$ ranges over the mesh cells, $\mathbf{s}$ over the edges, and $\mathbf{p}$ over the mesh nodes, both interior and boundary.
Throughout Section \ref{sec: Global edge-based Trefftz space}, we adopt the following assumption:

\begin{assumption} \label{A1}
    Recall the definition of the set $\sigma$ in \eqref{eq: set Neumann reference}. We assume $\kappa^2 \not\in\sigma(h_1,h_2)$.
\end{assumption}

\begin{remark} \label{rem: resonant frequencies edge}
If $\kappa$ does not satisfy Assumption \textup{\ref{A1}}, one possible strategy is to adjust the mesh in order to avoid the resonant frequencies.
In practice, this can be achieved by choosing suitable subdivisions of $h_1$ and $h_2$ that exclude the corresponding integer solutions, or by introducing slight perturbations of the mesh sizes, which prevent exact matches with the resonant frequencies.
\end{remark}

We now introduce some basis functions, denoted by $\{ \phi_{\mathbf{s}, n} \}_{\mathbf{s} \in \Sigma_h,  n \in \mathbb{N}^*}$, associated with the edges of the mesh $\Sigma_h$.

\begin{definition}[Edge basis function] \label{def: edge basis functions}
    Let $\mathbf{s} \in \Sigma_h$ and $(x_{\mathbf{s}},y_{\mathbf{s}})\in \mathbb{R}^2$.
    \begin{itemize}
\item If $\mathbf{s} = [x_{\mathbf{s}}, x_{\mathbf{s}} + h_1] \times \{y_{\mathbf{s}}\}$ is horizontal, its adjacent cells are
  \begin{equation} \label{eq: adj cell hor edge functions}
      K_{\mathbf{s}}^+ = (x_{\mathbf{s}}, x_{\mathbf{s}} + h_1) \times (y_{\mathbf{s}}, y_{\mathbf{s}} + h_2), \qquad
      K_{\mathbf{s}}^- = (x_{\mathbf{s}}, x_{\mathbf{s}} + h_1) \times (y_{\mathbf{s}} - h_2, y_{\mathbf{s}}),
  \end{equation}
  and, for any $n \in \mathbb{N}^*$, the corresponding function $ \phi_{\mathbf{s},n} : \Omega \to \mathbb{R} $ is defined as
  \begin{equation} \label{eq: hor edge functions}
  \phi_{\mathbf{s},n}(x,y) := 
  \begin{cases}
  \widehat{\varphi}_n(x - x_{\mathbf{s}}, h_2 \pm (y_{\mathbf{s}} - y); h_1, h_2), & (x,y) \in \overline{K_{\mathbf{s}}^\pm}, \\
  0, & \text{otherwise}.
  \end{cases}
  \end{equation}
\item If $\mathbf{s} = \{x_{\mathbf{s}}\} \times [y_{\mathbf{s}}, y_{\mathbf{s}} + h_2]$ is vertical, its adjacent cells are
  \begin{equation*}
    K_{\mathbf{s}}^+ = (x_{\mathbf{s}}, x_{\mathbf{s}} + h_1) \times (y_{\mathbf{s}}, y_{\mathbf{s}} + h_2), \qquad
    K_{\mathbf{s}}^- = (x_{\mathbf{s}} - h_1, x_{\mathbf{s}}) \times (y_{\mathbf{s}}, y_{\mathbf{s}} + h_2),
  \end{equation*}
  and, for any $n \in \mathbb{N}^*$, the corresponding function $ \phi_{\mathbf{s},n} : \Omega \to \mathbb{R} $ is defined as
  \begin{equation} \label{eq: ver edge functions}
  \phi_{\mathbf{s},n}(x,y) := 
  \begin{cases}
  \widehat{\varphi}_n(y - y_{\mathbf{s}}, h_1 \pm (x_{\mathbf{s}} - x); h_2, h_1), & (x,y) \in \overline{K_{\mathbf{s}}^\pm}, \\
  0, & \text{otherwise}.
  \end{cases}
  \end{equation}
    \end{itemize}
\end{definition}

Thanks to Assumption \ref{A1}, these functions are well-defined. Besides, $\phi_{\mathbf{s},n} \in C^0(\overline{\Omega})$ for any $\mathbf{s} \in \Sigma_h$ and $n \in \mathbb{N}^*$.
Moreover, they are compactly supported in $\Omega$, and are local solutions to the homogeneous Helmholtz equation \eqref{eq:helmholtz_equation} in each $K \in \mathcal{T}_h$, but not globally, as their normal trace jumps across $K^{\pm}_{\mathbf{s}}$ interfaces.

The functions $\phi_{\mathbf{s},n}$ inherit the same qualitative behavior as the local basis functions $\widehat{\varphi}_n$: they exhibit either oscillatory (propagative) or exponentially decaying (evanescent) behavior in the direction orthogonal to the edge $\mathbf{s}$, depending on the value of $n$. This is illustrated in Figure \ref{fig: basis edge functions}, where the left plot shows a propagative mode and the right plot an evanescent one.

\begin{figure}
\centering
\includegraphics[trim=120 220 120 180,clip,width=.3\textwidth]{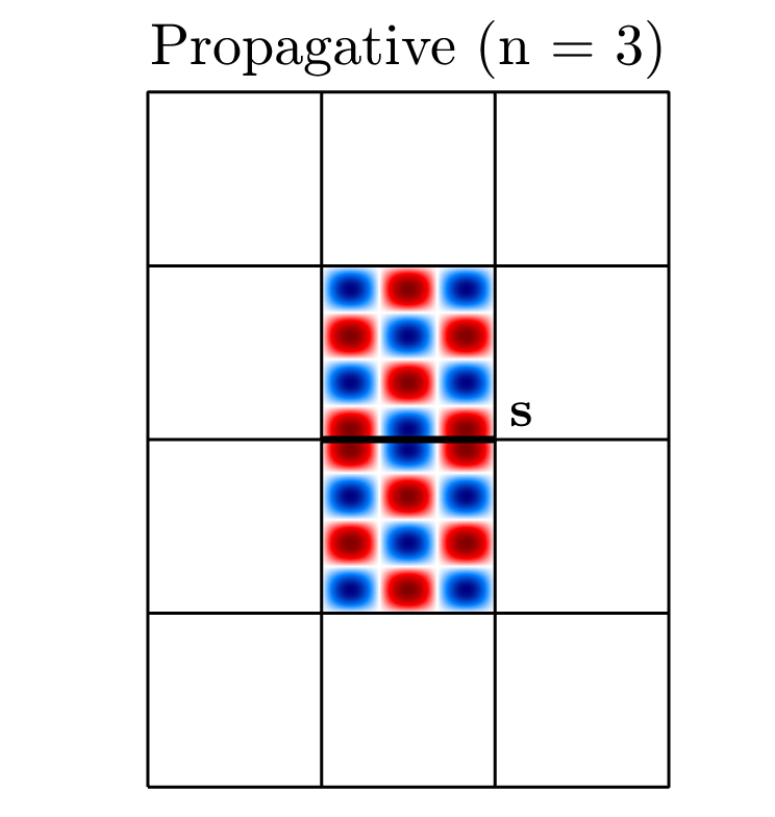}
\hspace{.15\textwidth}
\includegraphics[trim=120 220 120 180,clip,width=.3\textwidth]{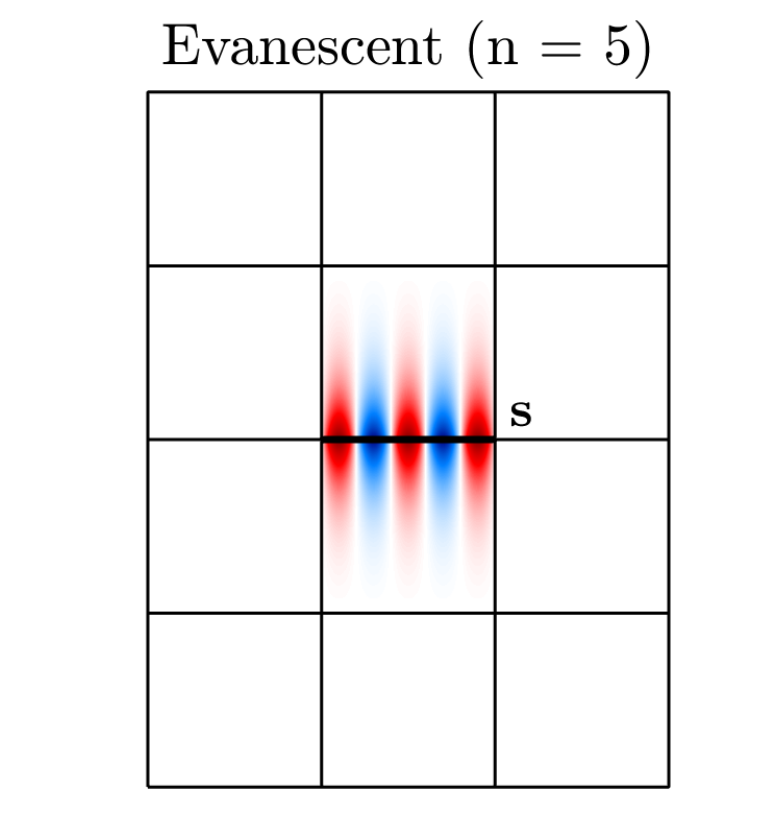}
\caption{Basis functions $\phi_{\mathbf{s},n}$ associated with an horizontal edge $\mathbf{s} \in \Sigma_h$;
$\kappa=15$, $h_1=h_2=1$.}
\label{fig: basis edge functions}
\end{figure}

We now introduce the following Trefftz space on $\Omega$.

\begin{definition}[Edge-based Trefftz space]
For any $N_{\mathbf{e}}\in \mathbb{N}^*$, define
\begin{equation} \label{eq: edge-based Trefftz space}
V_{N_{\mathbf{e}}}(\Sigma_h) := \textup{span} \left\{ \phi_{\mathbf{s}, n} \right\}_{\mathbf{s} \in \Sigma_h, \; 1 \leq n \leq N_{\mathbf{e}}}.
\end{equation}
where the functions $\phi_{\mathbf{s}, n}$ are introduced in \eqref{eq: hor edge functions} and \eqref{eq: ver edge functions}.
\end{definition}

It is clear that the dimension of $V_{N_\mathbf{e}}(\Sigma_h)$ is $N_\mathbf{e} \times |\Sigma_h|$, as the space is generated by the functions associated with all edges in $\Sigma_h$.

The orthogonality properties of the basis functions on the reference cell $\widehat{K}$, stated in Lemma \ref{lem:orthogonality lemma}, are not preserved in this construction; in fact, functions associated with different edges are generally not orthogonal, leading to redundancy in the approximation space.
Moreover, since any function in $V_{N_\mathbf{e}}(\Sigma_h)$ is solution to the Helmholtz equation \eqref{eq:helmholtz_equation} in any $K \in \mathcal{T}_h$ and is globally continuous in $\Omega$, $V_{N_\mathbf{e}}(\Sigma_h)$ is a $H_\kappa^1(\Omega)$-conforming Trefftz space, namely
\begin{equation} \label{eq: edge-based Trefftz space conformity}
    V_{N_\mathbf{e}}(\Sigma_h)\subset \{u \in H_\kappa^1(\Omega): -\Delta u -\kappa^2 u=0 \quad\!\! \text{in any}\quad\!\!\!\! K \in \mathcal{T}_h\}.
\end{equation}

\subsection{Global approximation results} 

We introduce broken Sobolev spaces on the mesh $\mathcal{T}_h$ and its skeleton $\Sigma_h$, equipped with $\kappa$-dependent norms.

\begin{definition}[$\kappa$-dependent broken Sobolev norms]
For any $M \in \mathbb{N}$, we define the broken Hilbert spaces equipped with the $\kappa$-dependent norms:
\begin{align*}
    &\left(H_{\kappa}^M(\mathcal{T}_h),\|\cdot\|_{H_{\kappa}^M(\mathcal{T}_h)}\right): &  &H_{\kappa}^M(\mathcal{T}_h):=\bigoplus_{K \in \mathcal{T}_h}H_{\kappa}^M(K), &   &\|\cdot\|_{H_{\kappa}^M(\mathcal{T}_h)}^2:=\sum_{K \in \mathcal{T}_h}\|\cdot\|_{H_{\kappa}^M(K)}^2,\\
    &\left(H_{\kappa}^M(\Sigma_h),\|\cdot\|_{H_{\kappa}^M(\Sigma_h)}\right):  & &H_{\kappa}^M(\Sigma_h):=\bigoplus_{\mathbf{s} \in \Sigma_h}H_{\kappa}^M(\mathbf{s}), &   &\|\cdot\|_{H_{\kappa}^M(\Sigma_h)}^2:=\sum_{\mathbf{s} \in \Sigma_h}\|\cdot\|_{H_{\kappa}^M(\mathbf{s})}^2.
\end{align*}
\end{definition}

To analyze the approximation properties of the edge-based Trefftz space $V_{N_\mathbf{e}}(\Sigma_h)$, we present a global counterpart to the local set described in Definition \ref{def: local set}.

\begin{definition} \label{def: global set}
    Let $N,M\in \mathbb{N}^*$ such that $N\leq \lceil M/2\rceil$.
    We define the set
    $$
\mathcal{S}^{N,M}_{\kappa}(\mathcal{T}_h) :=
\left\{ u \in H^1_\kappa(\Omega) \;\middle|\;
\begin{aligned}
  & -\Delta u - \kappa^2 u = 0 && \text{in every } K \in \mathcal{T}_h, \\
  & \,\,\,\,\, u|_{\mathbf{s}} \in H_\kappa^{M}(\mathbf{s}) && \text{for all } \mathbf{s} \in \Sigma_h, \\
  & \partial_{1}^{2n} ( u|_{\mathbf{s}})(\partial \mathbf{s}) = 0 && \text{for all } \mathbf{s} \in \Sigma_h \text{ horizontal},\; 0 \leq n \leq N-1,\\
  & \partial_{2}^{2n} ( u|_{\mathbf{s}})(\partial \mathbf{s}) = 0 && \text{for all } \mathbf{s} \in \Sigma_h \text{ vertical},\; 0 \leq n \leq N-1
\end{aligned}
\right\}.
$$
\end{definition}

For any $N_{\mathbf{e}} \in \mathbb{N}^*$, we introduce the global edge-based interpolant into $V_{N_{\mathbf{e}}}(\Sigma_h)$ as
\begin{equation} \label{eq: global interpolant}
    \boldsymbol{\mathcal{E}} : H_\kappa^1(\Omega) \rightarrow H_\kappa^1(\Omega), \qquad 
    u \mapsto \sum_{\mathbf{s}\in \Sigma_h}\sum_{n=1}^{N_{\mathbf{e}}}\frac{( u,\phi_{\mathbf{s},n})_{L^2(\mathbf{s})}}{\| \phi_{\mathbf{s},n}\|^2_{L^2(\mathbf{s})}}\phi_{\mathbf{s},n}.
\end{equation}

\begin{proposition}[Global edge-based approximation result] \label{th: main theorem_2}
    Let $N, M \in \mathbb{N}^*$ such that $N \leq \lceil M/2 \rceil$, and let $r,\mu \in \mathbb{N}$ with $r \leq \mu \leq \min(2N, M)$.
    Assume $N_\mathbf{e} \in \mathbb{N}^*$ such that
    \begin{equation} \label{eq: nuN}
        \nu_{N_\mathbf{e}}:=\frac{N_\mathbf{e}\pi}{\kappa h}\geq \sqrt{2}.
    \end{equation}
    Then, for any $u \in \mathcal{S}^{N,M}_{\kappa}(\mathcal{T}_h)$, the best-approximation error in $V_{N_{\mathbf{e}}}(\Sigma_h)$ satisfies:
    \begin{equation} \label{eq: global approximation result useful}
        \inf_{v \in V_{N_{\mathbf{e}}}(\Sigma_h)}\|u-v\|_{H_{\kappa}^r(\mathcal{T}_h)}\leq\|u-\boldsymbol{\mathcal{E}}u\|_{H_{\kappa}^r(\mathcal{T}_h)}\leq \frac{2\sqrt{2}(r+1)}{\sqrt{\kappa\tanh(\kappa h/\rho)}} \nu^{r-(\mu+1/2)}_{N_{\mathbf{e}}}|  u|_{H_{\kappa}^{\mu}(\Sigma_h)}.
    \end{equation}
    If moreover $u \in H_{\kappa}^{M+1}(\mathcal{T}_h)$, then:
    \begin{equation} \label{eq: global approximation result}
        \inf_{v \in V_{N_{\mathbf{e}}}(\Sigma_h)}\|u-v\|_{H_{\kappa}^r(\mathcal{T}_h)}\leq\|u-\boldsymbol{\mathcal{E}}u\|_{H_{\kappa}^r(\mathcal{T}_h)}\leq C_{1}\left(\frac{r+1}{2}\right) \nu^{r-(\mu+1/2)}_{N_{\mathbf{e}}}\| u\|_{H_{\kappa}^{\mu+1}(\mathcal{T}_h)},
    \end{equation}
    where
    \begin{equation} \label{eq: constants main theorem 1}
        C_{1}:=\frac{8\sqrt{\rho}\max(1,\kappa h)}{\sqrt{\kappa h\tanh(\kappa h/\rho)}}.
    \end{equation}
\end{proposition}
\begin{proof}
    Consider $u \in \mathcal{S}^{N,M}_{\kappa}(\mathcal{T}_h)$ and fix an element $K \in \mathcal{T}_h$. For any edge $\mathbf{s} \subset \partial K$, the trace satisfies $ u|_{\mathbf{s}} \in H^M_{\kappa}(\mathbf{s})$ with zero boundary conditions $ u|_{\mathbf{s}}(\partial \mathbf{s}) = 0$.
    Therefore, the zero extension to the whole boundary $\widetilde{ u|_{\mathbf{s}}}\in H^1_\kappa(\partial K)$. By the well-posedness of the Helmholtz-Dirichlet problem under Assumption \ref{A1}, there exists a Helmholtz solution $u_{\mathbf{s}} \in H^1_\kappa(K)$ such that $ u_{\mathbf{s}}|_{\partial K} = \widetilde{ u|_{\mathbf{s}}}$. In particular, $u_{\mathbf{s}}\in \mathcal{S}^{N,M}_{\kappa}(K,\mathbf{s})$ for any $\mathbf{s}\subset \partial K$, and moreover
    \begin{equation} \label{eq: u decomposition u_s}
        u|_K=\sum_{\mathbf{s}\subset \partial K}u_{\mathbf{s}}.
    \end{equation}
    The operator $\boldsymbol{\Pi}_{\hat{\mathbf{s}}}$ from \eqref{eq: local projection} can be extended to any edge $\mathbf{s} \subset \partial K$ by using the functions $\phi_{\mathbf{s},n}$, whose restriction to $K$ coincides -- up to a rigid transformation -- with the reference modes $\widehat{\varphi}_n$ on $\widehat{K}$ taking $(\widehat{h}_1,\widehat{h}_2)=(h_1,h_2)$ or $(\widehat{h}_1,\widehat{h}_2)=(h_2,h_1)$; see Definition~\ref{def: edge basis functions}.
    Specifically, for any $\mathbf{s} \subset \partial K$, we introduce the corresponding $H^1_\kappa(K)$-orthogonal projection
    \begin{equation*}
    \boldsymbol{\Pi}_{\mathbf{s}} : H^1_\kappa(K) \rightarrow H^1_\kappa(K), \qquad 
    u \mapsto \sum_{n = 1}^{N_{\mathbf{e}}} \frac{(u, \phi_{\mathbf{s},n})_{H_\kappa^{1}(K)}}{\|\phi_{\mathbf{s},n}\|^2_{H^{1}_\kappa(K)}}\phi_{\mathbf{s},n}|_K.
    \end{equation*} 
    Since $u_{\mathbf{s}} \in \mathcal{S}^{N,M}_{\kappa}(K,\mathbf{s})$ for every $\mathbf{s} \subset \partial K$, Proposition \ref{prop: completness} implies that:
    \begin{align}
        \boldsymbol{\mathcal{E}}u|_K&=\sum_{\mathbf{s}\subset \partial K}\sum_{n=1}^{N_\mathbf{e}}\frac{( u, \phi_{\mathbf{s},n})_{L^2(\mathbf{s})}}{\| \phi_{\mathbf{s},n}\|^2_{L^2(\mathbf{s})}}\phi_{\mathbf{s},n}|_K=\sum_{\mathbf{s}\subset \partial K}\sum_{n=1}^{N_\mathbf{e}}\frac{( u_{\mathbf{s}}, \phi_{\mathbf{s},n})_{L^2(\mathbf{s})}}{\| \phi_{\mathbf{s},n}\|^2_{L^2(\mathbf{s})}}\phi_{\mathbf{s},n}|_K \nonumber\\
        &=\sum_{\mathbf{s}\subset \partial K}\sum_{n=1}^{N_\mathbf{e}}\frac{(u_{\mathbf{s}},\phi_{\mathbf{s},n})_{H_\kappa^1(K)}}{\| \phi_{\mathbf{s},n}\|^2_{H_\kappa^1(K)}}\phi_{\mathbf{s},n}|_K=\sum_{\mathbf{s}\subset \partial K}\boldsymbol{\Pi}_{\mathbf{s}}u_{\mathbf{s}}. \label{eq: global local projection property}
    \end{align}    
    From \eqref{eq: u decomposition u_s} and \eqref{eq: global local projection property}, it follows
    \begin{equation*}
        \|u-\boldsymbol{\mathcal{E}}u\|^2_{H^r_\kappa(\mathcal{T}_h)}=\sum_{K \in \mathcal{T}_h}\|u-\boldsymbol{\mathcal{E}}u\|^2_{H^r_\kappa(K)}\leq
        \sum_{K \in \mathcal{T}_h} \sum_{\mathbf{s}\subset \partial K}\|u_{\mathbf{s}}-\boldsymbol{\Pi}_{\mathbf{s}}u_{\mathbf{s}}\|^2_{H^r_\kappa(K)}.
    \end{equation*}
    Since in \eqref{eq: nuN} we are assuming $\frac{N_\mathbf{e}\pi}{\kappa \min(h_1,h_2)}\geq\nu_{N_\mathbf{e}}\geq \sqrt{2}$,
    Proposition \ref{th: main local th_} can be applied to control the local error $\|u_{\mathbf{s}}-\boldsymbol{\Pi}_{\mathbf{s}}u_{\mathbf{s}}\|_{H^r_\kappa(K)}$.
    So, combining the local estimate \eqref{eq: local approximation result useful} with the bounds on $h_1$ and $h_2$ in terms of the mesh size $h$ and the shape parameter $\rho$ from \eqref{eq: shape regularity}, and exploiting the fact that $ u_{\mathbf{s}}|_{\mathbf{s}} =  u|_{\mathbf{s}}$, one obtains
    \begin{equation} \label{eq: midstep edges proof}
        \|u-\boldsymbol{\mathcal{E}}u\|^2_{H^r_\kappa(\mathcal{T}_h)}\leq \sum_{K \in \mathcal{T}_h} \sum_{\mathbf{s}\subset \partial K} \frac{4(r+1)^2}{\kappa\tanh(\kappa h/\rho)} \nu^{2(r-\mu)-1}_{N_{\mathbf{e}}}|  u|^2_{H_{\kappa}^{\mu}(\mathbf{s})}.
    \end{equation}
    Hence, from \eqref{eq: midstep edges proof} we derive the estimate \eqref{eq: global approximation result useful}, and by additionally applying Lemma~\ref{lem: trace inequality} to each $H^\mu_\kappa(\mathbf{s})$-seminorm in \eqref{eq: midstep edges proof}, we obtain the approximation bound \eqref{eq: global approximation result}.
\end{proof}

\subsection{Approximation limitations} \label{sec: approximation limitations}

Having established global approximation results in Proposition \ref{th: main theorem_2} under the assumption $u \in \mathcal{S}^{N,M}_{\kappa}(\mathcal{T}_h)$, we now test its sharpness with a numerical experiment.
Consider the orthogonal projection of a propagative plane wave onto the discrete space $V_{N_\mathbf{e}}(\Sigma_h)$, using the $H_\kappa^1(\Omega)$-inner product. This test case is particularly meaningful, as such waves should be well approximated by any effective Helmholtz-adapted discrete space.

The domain is set to $\Omega = (0,1)^2$, with wavenumber $\kappa = 30$, and a mesh $\mathcal{T}_h$ consisting of four square cells. The target function is
\begin{equation} \label{eq: PPW numerical experiment}
u(\mathbf{x}) = e^{i \kappa \mathbf{d} \cdot \mathbf{x}}, \qquad \mathbf{x} \in \Omega, \qquad \text{where} \qquad \mathbf{d} = \left(\tfrac{1}{\sqrt{2}}, \tfrac{1}{\sqrt{2}}\right) \in \mathbb{S}^1.
\end{equation}
This function is an analytic solution of the homogeneous Helmholtz equation, but $u \not \in \mathcal{S}^{N,M}_{\kappa}(\mathcal{T}_h)$, as it does not satisfy the even-order vanishing condition required in Definition \ref{def: global set}, namely
\begin{equation} \label{eq: zeroing assumption}
\begin{aligned}
  & \partial_{1}^{2n} ( u|_{\mathbf{s}})(\partial \mathbf{s}) \neq 0 && \text{for all} && \mathbf{s} \in \Sigma_h \text{ horizontal}, && n \in \mathbb{N},\\
  & \partial_{2}^{2n} ( u|_{\mathbf{s}})(\partial \mathbf{s}) \neq 0 && \text{for all} && \mathbf{s} \in \Sigma_h \text{ vertical}, && n \in \mathbb{N}.
\end{aligned}
\end{equation}

\begin{figure}
\centering
\includegraphics[trim=130 240 60 190,clip,width=.3\textwidth]{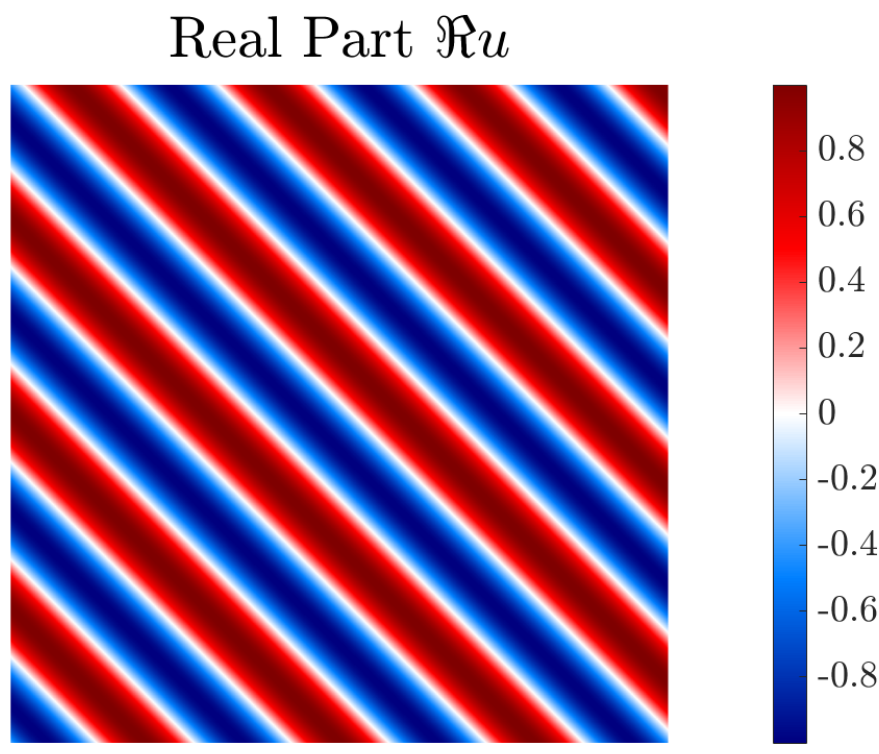}
\hspace{.01\textwidth}
\includegraphics[trim=130 240 60 190,clip,width=.3\textwidth]{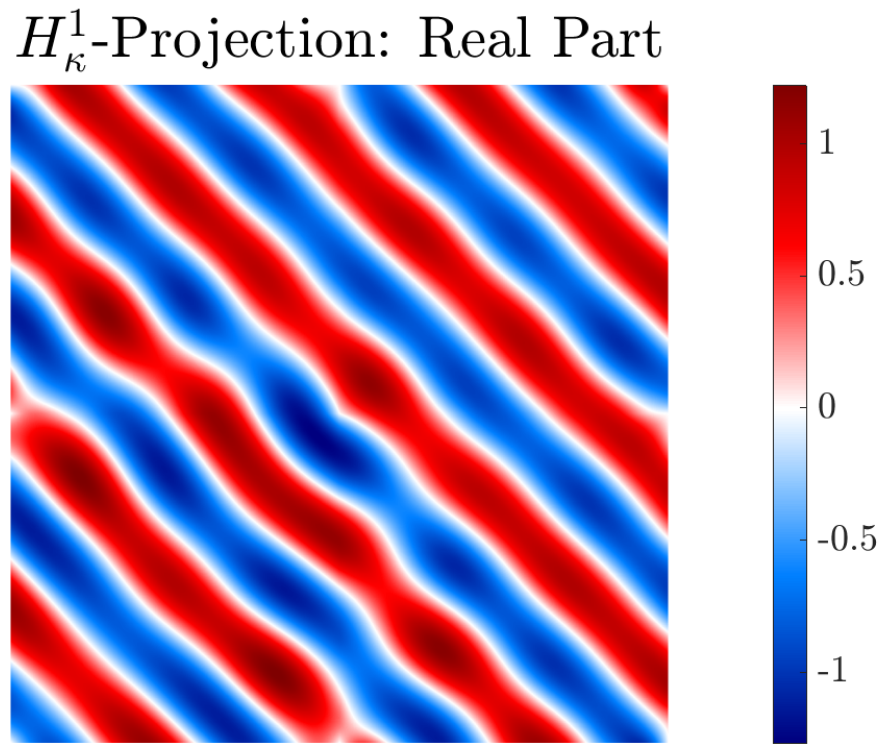}
\hspace{.01\textwidth}
\includegraphics[trim=130 240 60 190,clip,width=.3\textwidth]{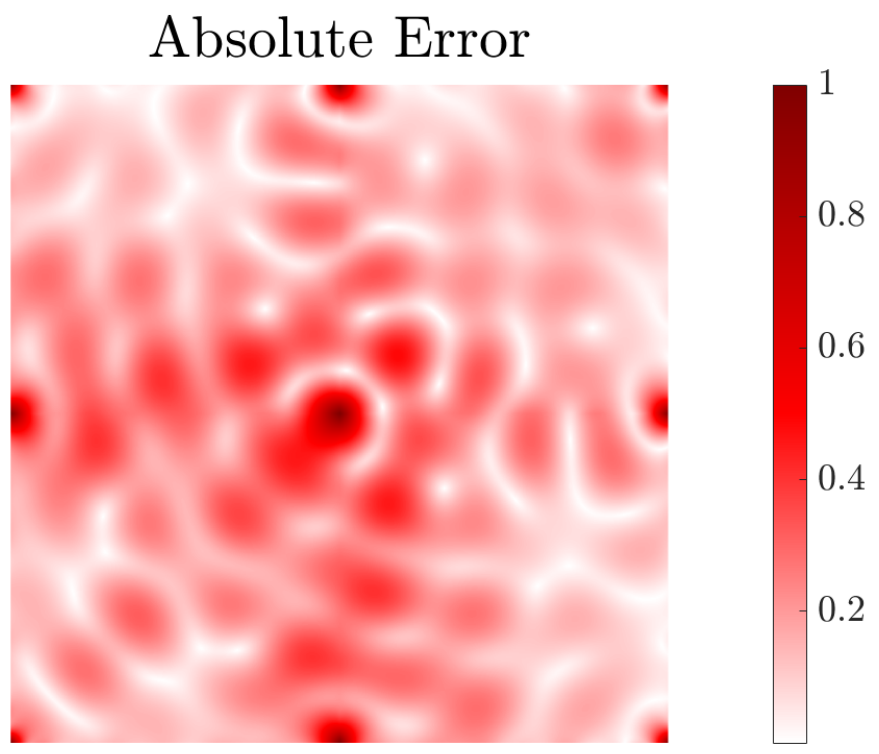}
\caption{$H_\kappa^1(\Omega)$-orthogonal projection of the plane wave $u(\mathbf{x}) = e^{\imath \kappa \mathbf{d}\cdot \mathbf{x}}$ with $\mathbf{d} = (1/\sqrt{2}, 1/\sqrt{2})$; $\Omega=(0,1)^2$, wavenumber $\kappa=30$. Left: plane wave real part $\Re u$. Center and right: real part of the plane wave $H_\kappa^1(\Omega)$-projection and absolute error for $N_{\mathbf{e}} = 12$.}
\label{fig: plane wave H1}
\end{figure}

The numerical results are illustrated in Figure \ref{fig: plane wave H1}.
Notably, the maximum absolute error attains the value 1 precisely at the mesh nodes $\mathcal{N}_h$, since all the edge basis functions $\phi_{\mathbf{s},n}$ vanish at every node $\mathbf{p} \in \mathcal{N}_h$.
This highlights an intrinsic limitation of the discrete space $V_{N_\mathbf{e}}(\Sigma_h)$, which, by construction, lacks degrees of freedom at the nodes and thus cannot capture critical aspects of the exact solution near them. Consequently, even analytic Helmholtz solutions such as plane waves cannot be accurately approximated, if the property \eqref{eq: zeroing assumption} is not satisfied.

This result motivates the need to enrich the discrete space $V_{N_\mathbf{e}}(\Sigma_h)$ with additional functions that can convey information at the mesh nodes $\mathcal{N}_h$.

\section{Enrichment with node-based functions} \label{sec: Enrichment with node-based functions}

In this section, we address the limitations of the discrete edge-based Trefftz space $V_{N_{\mathbf{e}}}(\Sigma_h)$ in \eqref{eq: edge-based Trefftz space}, which, as shown in Section \ref{sec: approximation limitations}, is not rich enough to approximate all relevant Helmholtz solutions. To remedy this, we introduce a complementary space spanned by functions associated with the nodes of the mesh $\mathcal{N}_h$.
We begin by defining some basis functions, localized around the mesh nodes, satisfying the homogeneous Helmholtz equation \eqref{eq:helmholtz_equation} in each cell, and constructed to ensure $H_\kappa^1(\Omega)$-conformity. We then present a series of results that characterize the properties of this node-based space and provide the tools required to establish the approximation results that will be presented in Section \ref{sec: Combined edge-and-node Trefftz space}.
The constructions presented in this section are intended for theoretical analysis only and are not part of a practical numerical method.

\subsection{Construction of the node-based Trefftz space}

The idea is to construct basis functions associated with the mesh nodes $\mathbf{p} \in \mathcal{N}_h$ that are still globally continuous and compactly supported. These functions are defined in the same spirit as the edge basis functions introduced in Definition \ref{def: edge basis functions}, but with odd indices $n \in \mathbb{N}^*$ and linked to the horizontal edges of a coarser mesh with doubled cell base $2h_1$.
This is formalized below.

\begin{definition}[Node basis function] \label{def: ndoe basis functions}
  Let $\mathbf{p}=(x_{\mathbf{p}},y_{\mathbf{p}}) \in \mathcal{N}_h$.
  Its adjacent coarse cells are
  \begin{equation} \label{eq: adjacent cells nodal}
      K_{\mathbf{p}}^+ = (x_{\mathbf{p}}-h_1, x_{\mathbf{p}} + h_1) \times (y_{\mathbf{p}}, y_{\mathbf{p}} + h_2), \qquad
      K_{\mathbf{p}}^- = (x_{\mathbf{p}}-h_1, x_{\mathbf{p}} + h_1) \times (y_{\mathbf{p}} - h_2, y_{\mathbf{p}}),
  \end{equation}
  and, for any $n \in \mathbb{N}^*$, the corresponding function $ \psi_{\mathbf{p},n} : \Omega \to \mathbb{R} $ is defined as
  \begin{equation} \label{eq: definition basis nodes}
  \psi_{\mathbf{p},n}(x,y) := 
  \begin{cases}
  \widehat{\varphi}_{2n-1}(h_1-(x_{\mathbf{p}} - x), h_2 \pm (y_{\mathbf{p}} - y); 2h_1, h_2), & (x,y) \in \overline{K_{\mathbf{p}}^\pm}, \\
  0, & \text{otherwise}.
  \end{cases}
  \end{equation}
\end{definition}

To ensure the functions in \eqref{eq: definition basis nodes} are well-defined, for the remainder of the paper we assume:

\begin{assumption} \label{A1+A2}
    Recall the definition of the set $\sigma$ in \eqref{eq: set Neumann reference}. We assume $\kappa^2 \not\in\sigma(2h_1,h_2)$.
\end{assumption}

\begin{remark} \label{rem: asimmetry}
The family $\{\psi_{\mathbf{p},n}\}_{\mathbf{p}\in \mathcal{N}_h,n\in \mathbb{N}^*}$ in \eqref{eq: definition basis nodes} consists of functions of the same type as those in \eqref{eq: hor edge functions}, now associated with horizontal edges of length $2h_1$. This explains the asymmetry between $2h_1$ and $h_2$ in Assumption~\textup{\ref{A1+A2}}. One could instead define the functions $\psi_{\mathbf{p},n}$ to be associated with the vertical edges of the mesh obtained by doubling $h_2$, or even include both horizontal and vertical families; in either case Assumption~\textup{\ref{A1+A2}} would need to be adjusted accordingly. Our analysis shows that the horizontal family \eqref{eq: definition basis nodes} alone is sufficient to derive good approximation estimates.
\end{remark}

For any $\mathbf{p}\in \mathcal{N}_h$, the nodal basis functions $\psi_{\mathbf{p},n}\in C^0(\overline{\Omega})$ and are compactly supported on the four cells surrounding the node $\mathbf{p}$. Like their edge-based counterparts $\phi_{\mathbf{s},n}$ from Definition \ref{def: edge basis functions}, they satisfy the homogeneous Helmholtz equation \eqref{eq:helmholtz_equation} exactly on each cell $K \in \mathcal{T}_h$. Qualitatively, $\psi_{\mathbf{p},n}$ share the same analytical features, exhibiting either propagative (oscillatory) or evanescent (exponentially decaying) profiles along the vertical direction depending on the index $n$. This behavior is illustrated in Figure \ref{fig: basis node functions}, where both modes are depicted.

\begin{figure}
\centering
\includegraphics[trim=120 220 120 180,clip,width=.31\textwidth]{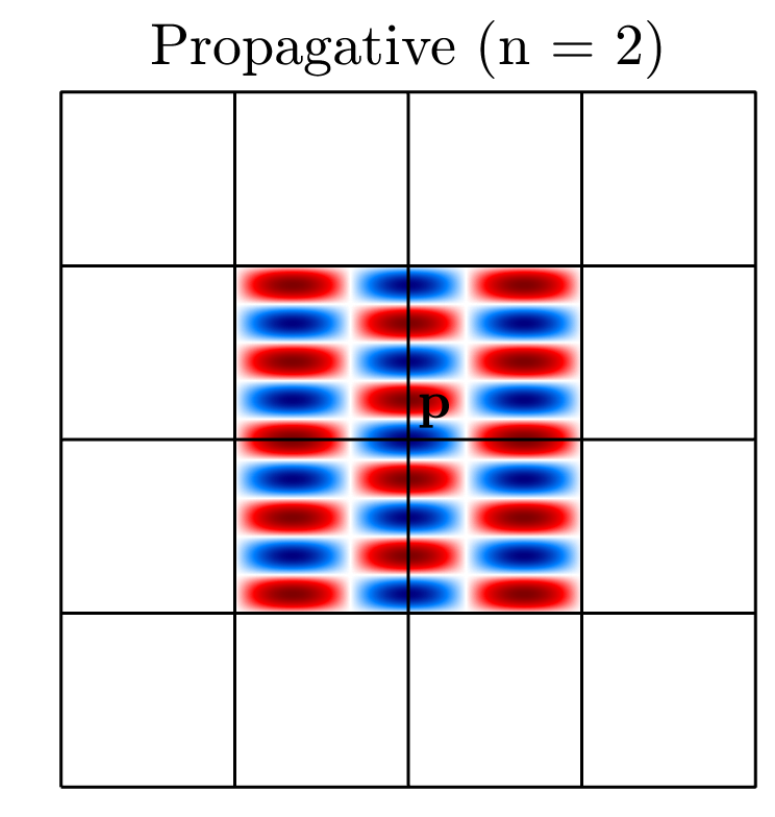}
\hspace{.2\textwidth}
\includegraphics[trim=120 220 120 180,clip,width=.31\textwidth]{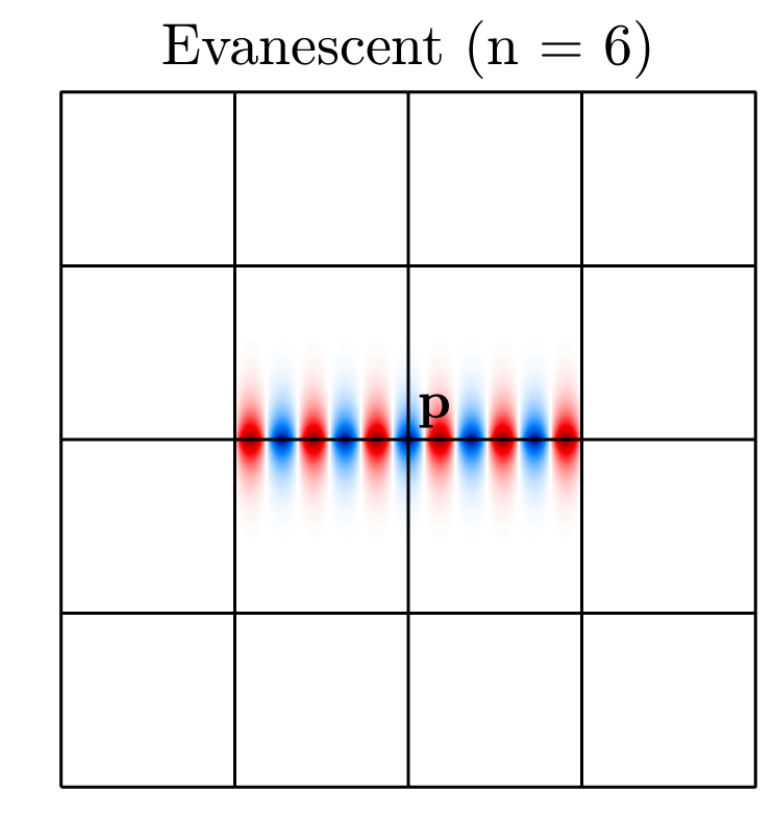}
\caption{Basis functions $\psi_{\mathbf{p},n}$ associated with a node $\mathbf{p} \in \mathcal{N}_h$;
$\kappa=15$, $h_1=h_2=1$.}
\label{fig: basis node functions}
\end{figure}

We now introduce the nodal Trefftz space on $\Omega$.

\begin{definition}[Node-based Trefftz space]
For any $N_{\mathbf{n}}\in \mathbb{N}^*$, define
\begin{equation} \label{eq: node-based Trefftz space}
V_{N_{\mathbf{n}}}(\mathcal{N}_h) := \textup{span} \left\{ \psi_{\mathbf{p}, n} \right\}_{\mathbf{p} \in \mathcal{N}_h, \; 1 \leq n \leq N_{\mathbf{n}}},
\end{equation}
where the functions $\psi_{\mathbf{p}, n}$ are introduced in \eqref{eq: definition basis nodes}.
\end{definition}

The dimension of the space $V_{N_\mathbf{n}}(\mathcal{N}_h)$ is given by $N_\mathbf{n} \times |\mathcal{N}_h|$, since it is spanned by $N_\mathbf{n}$ basis functions for each node in $\mathcal{N}_h$. As in the edge-based case, the orthogonality of the reference basis is lost: in general, functions associated with different nodes are not mutually orthogonal. By construction, every function in $V_{N_\mathbf{n}}(\mathcal{N}_h)$ solves the Helmholtz equation \eqref{eq:helmholtz_equation} locally on each cell $K \in \mathcal{T}_h$ and is globally continuous over $\Omega$. Hence the space is $H_\kappa^1(\Omega)$-conforming and consists entirely of local Helmholtz solutions, namely
\begin{equation} \label{eq: node-based Trefftz space conformity}
    V_{N_\mathbf{n}}(\mathcal{N}_h)\subset \{u \in H_\kappa^1(\Omega): -\Delta u -\kappa^2 u=0 \quad\!\! \text{in any}\quad\!\!\!\! K \in \mathcal{T}_h\}.
\end{equation}

\subsection{Fundamental properties}

We now investigate the properties of the node-based Trefftz space \eqref{eq: node-based Trefftz space}.
The next lemma shows that, for any $\mathbf{p}\in \mathcal{N}_h$ and $m \in \mathbb{N}^*$, the basis functions $\psi_{\mathbf{p},m}$ admit continuous even-order partial derivatives, and that their even $\partial_1$-derivatives vanish at all nodes in $\mathcal{N}_h$ except $\mathbf{p}$.

\begin{lemma} \label{lem: psi_n properties}
For any $n\in \mathbb{N}$, $m \in \mathbb{N}^*$, $\mathbf{p},\mathbf{q}\in \mathcal{N}_h$ and $i=1,2$, $\partial_i^{2n}\psi_{\mathbf{p},m}\in C^0(\overline{\Omega})$. Moreover,
\begin{equation} \label{eq: psi_n properties}
    (\kappa^{-1}\partial_1)^{2n}\psi_{\mathbf{p},m}(\mathbf{q})=-\delta_{\mathbf{pq}} (-1)^{n+m} \widetilde{\nu}_{m}^{2n}, \qquad \text{where} \qquad \widetilde{\nu}_{m}:=\frac{(2m-1)\pi}{2\kappa h_1}.
\end{equation}
\end{lemma}
\begin{proof}
Fix $m \in \mathbb{N}^*$ and $\mathbf{p} \in \mathcal{N}_h$.
From \eqref{eq:edge_functions}, for all $n \in \mathbb{N}$, one has
\begin{equation*}
    \partial_1^{2n}\widehat{\varphi}_{2m-1}(x,y;2h_1,h_2)=(-1)^n(\kappa \widetilde{\nu}_{m})^{2n}\widehat{\varphi}_{2m-1}(x,y;2h_1,h_2), \qquad (x,y)\in \widehat{K}.
\end{equation*}
By \eqref{eq: definition basis nodes}, the function $\psi_{\mathbf{p},m}$ is piecewise analytic. Therefore, on each region where it is analytic, the same identity holds:
\begin{equation} \label{eq: cont1}
    \partial_1^{2n}\psi_{\mathbf{p},m}=(-1)^n(\kappa \widetilde{\nu}_{m})^{2n}\psi_{\mathbf{p},m}.
\end{equation}
As $\psi_{\mathbf{p},m}\in C^0(\overline{\Omega})$, this implies $\partial_1^{2n}\psi_{\mathbf{p},m}\in C^0(\overline{\Omega})$ as well.
A similar argument can be used to show that $\partial_2^{2n}\psi_{\mathbf{p},m}\in C^0(\overline{\Omega})$.
Moreover, by construction, $\psi_{\mathbf{p},m}(\mathbf{q})=0$ for any $\mathbf{q} \neq \mathbf{p}$, and
\begin{equation*}
    \psi_{\mathbf{p},m}(\mathbf{p})=\widehat{\varphi}_{2m-1}(h_1,h_2;2h_1,h_2)=\sin(\kappa \widetilde{\nu}_{m}h_1)
    =(-1)^{m-1},
\end{equation*}
so the identity in \eqref{eq: psi_n properties} follows directly from \eqref{eq: cont1}.
\end{proof}

Building on this localization property, the following proposition constructs a nodal collocation system ensuring that, for any sufficiently regular function $u$, there exists a unique function in the Trefftz space $V_{N}(\mathcal{N}_h)$ that matches the first $N$ even $\partial_1$-derivatives of $u$ at all mesh nodes.

\begin{proposition}[Nodal collocation system] \label{th: nodal collocation}
    Let $N \in \mathbb{N}^*$, and $u\in H_\kappa^{2N}(\Omega)$. There exists a unique coefficient vector $\boldsymbol{\beta}\in \mathbb{C}^{N |\mathcal{N}_h|}$ solution to the collocation system
    \begin{equation} \label{eq: local collocation operator condition}
        \sum_{\mathbf{p}\in \mathcal{N}_h}\sum_{m=1}^N \beta_{\mathbf{p},m}(\kappa^{-1}\partial_1)^{2n}\psi_{\mathbf{p},m}(\mathbf{q})=(\kappa^{-1}\partial_1)^{2n}u(\mathbf{q}), \qquad 0\leq n\leq N-1, \quad \mathbf{q}\in \mathcal{N}_h,
    \end{equation}
    where $\boldsymbol{\beta}=(\beta_{\mathbf{p},m})_{\mathbf{p}\in \mathcal{N}_h,1\leq m \leq N}$.
    Besides, the following bound holds
    \begin{equation} \label{eq:beta_pn0_p}
        \|\boldsymbol{\beta}\|_{\infty}\!\leq\! \frac{2\rho \max(1,\kappa h)^2}{h}\chi_{N}\|u\|_{H_\kappa^{2N}(\Omega)}, \quad\;\; \text{where} \quad\;\; \chi_{N}\!:=\!\min\!\left(\!\frac{\rho e^{\kappa h}}{\kappa h},e\max(1,\kappa h)^{2(N-1)}\!\right).
    \end{equation}
\end{proposition}
\begin{proof}
For any $0 \le n \le N-1$, Lemma~\ref{lem: psi_n properties} ensures that the values $\partial_1^{2n} \psi_{\mathbf{p},m}(\mathbf{q})$ are well-defined. Moreover, since $u \in H_\kappa^{2N}(\Omega)$, the derivatives $\partial_1^{2n} u(\mathbf{q})$ are also well-defined.
Using \eqref{eq: psi_n properties}, the linear system in \eqref{eq: local collocation operator condition} decouples into $|\mathcal{N}_h|$ independent collocation systems, one for each node. Specifically, for any $\mathbf{q} \in \mathcal{N}_h$, one obtains
\begin{equation} \label{eq: collocation sys}
    \mathbf{A} \boldsymbol{\beta}_{\mathbf{q}} = \mathbf{b}_{\mathbf{q}}, \qquad \text{with} \qquad \boldsymbol{\beta}_{\mathbf{q}} = (\beta_{\mathbf{q},m})_{m=1}^N,
\end{equation}
where, shifting the index $n$ from $0 \leq n \leq N-1$ to $1 \leq n \leq N$ for notational convenience,
\begin{equation} \label{eq: matrix A}
    \mathbf{A}_{nm} := (-1)^{n+m} \widetilde{\nu}_{m}^{2(n-1)}, \qquad (\mathbf{b}_{\mathbf{q}})_n := (\kappa^{-1} \partial_1)^{2(n-1)} u(\mathbf{q}), \qquad 1 \leq n, m \leq N.
\end{equation}
with $\widetilde{\nu}_{m}$ introduced in \eqref{eq: psi_n properties}.
The matrix $\mathbf{A}$ is invertible, as it is unitarily similar to a Vandermonde matrix with $N$ distinct nodes $\widetilde{\nu}_{m}^2$, for $1 \leq m \leq N$, defined in \eqref{eq: psi_n properties}. As a result, $\boldsymbol{\beta} \in \mathbb{C}^{N|\mathcal{N}_h|}$ is uniquely determined.

    From the definition of the right-hand side in \eqref{eq: matrix A} and Lemma \ref{lem: Linfty trace inequality}, it follows that
    \begin{align}
        \left\|\mathbf{b}_{\mathbf{q}}\right\|^2_{\infty}&=\max_{1\leq n\leq N}\left|(\kappa^{-1} \partial_1)^{2(n-1)} u(\mathbf{q})\right|^2\leq \max_{1\leq n \leq N}\sum_{K \in \mathcal{T}_h}\left\|(\kappa^{-1}\partial_1)^{2(n-1)} u\right\|^2_{L^{\infty}(\partial K)} \nonumber\\
        &\leq \left(\frac{2 \max(1,\kappa h_1,\kappa h_2)^2}{\min( h_1, h_2)}\right)^2\max_{1\leq n \leq N}\sum_{K \in \mathcal{T}_h}\left\|(\kappa^{-1}\partial_1)^{2(n-1)}u\right\|^2_{H_\kappa^{2}(K)} \nonumber\\
        &\leq \left(\frac{2\rho \max(1,\kappa h)^2}{h} \right)^2\left\|u\right\|^2_{H_\kappa^{2N}(\Omega)}, \label{eq: bound b}
    \end{align}
    where we bounded $h_1$ and $h_2$ in terms of the mesh size $h$ and the shape parameter $\rho$ from \eqref{eq: shape regularity}.
    
    The next step is to bound the norm $\left\|\mathbf{A}^{-1}\right\|_{\infty}$.
    In particular, since $\mathbf{A}$ has the structure given in \eqref{eq: matrix A}, the explicit formula from \cite[Eq.\ (3.1)]{Gautschi1962} applies with nodes $\widetilde{\nu}_{m}^2$, that is
    \begin{equation}\label{eq:norm vandermonde_}
        \left\|\mathbf{A}^{-1}\right\|_{\infty}=\max_{1\leq n\leq N}\prod_{\substack{m=1 \\ m \ne n}}^N
        \frac{1+\widetilde{\nu}_{m}^2}{|\widetilde{\nu}_{n}^2-\widetilde{\nu}_{m}^2|}=\max_{1\leq n\leq N}\prod_{\substack{m=1 \\ m \ne n}}^N\frac{1/4+\left(\kappa h_1/\pi\right)^2+m(m-1)}{|n-m|(n+m-1)}.
    \end{equation}
    Using standard factorial identities, it easy to verify that, for any $1\leq n \leq N$,
    \begin{equation*}
        \prod_{\substack{m=1 \\ m \ne n}}^N |n-m|=(n-1)!(N-n)!, \qquad \text{and} \qquad  \prod_{\substack{m=1 \\ m \ne n}}^N (n+m-1)=\frac{(N+n-1)!}{(2n-1)(n-1)!},
    \end{equation*}
    thus the denominator in \eqref{eq:norm vandermonde_} can be rewritten as
    \begin{equation}\label{eq:norm proof 1_}
        \prod_{\substack{m=1 \\ m \ne n}}^N |n-m| (n+m-1)=\frac{(N+n-1)!(N-n)!}{2n-1}.
    \end{equation}
    As for the numerator, one can obtain the identity
\begin{equation} \label{eq:norm proof A_}
    \prod_{\substack{m=1 \\ m \ne n}}^N \left(\frac{1}{4}\!+\!\left(\frac{\kappa h_1}{\pi}\right)^2\!\!\!+m(m-1)\right)=\frac{\prod_{m=1}^N\left(m-1/2+\imath\kappa h_1/\pi\right)\left(m-1/2-\imath \kappa h_1/\pi\right)}{1/4+\left(\kappa h_1/\pi\right)^2+n(n-1)},
\end{equation}
extracting the $n$-th term and factorizing each remaining term as the product of two complex conjugates.
Moreover, iteratively applying the relation $\Gamma(z+1)=z \Gamma(z)$ for $z \in \mathbb{C}$ yields
\begin{equation} \label{eq:gamma relation_}
    \prod_{m=1}^N(m+z)=\frac{\Gamma(N+z+1)}{\Gamma(z+1)}, \qquad z \in \mathbb{C}.
\end{equation}
Combining \eqref{eq:norm proof A_} with \eqref{eq:gamma relation_}, and using the identity $\Gamma(z)\Gamma(\overline{z})=|\Gamma(z)|^2$ for $z \in \mathbb{C}$, one deduces that
\begin{equation*} 
    \prod_{\substack{m=1 \\ m \ne n}}^N \left(\frac{1}{4}\!+\!\left(\frac{\kappa h_1}{\pi}\right)^2\!\!\!+m(m-1)\right)=\left(\frac{1}{4}\!+\!\left(\frac{\kappa h_1}{\pi}\right)^2\!\!\!+n(n-1)\right)^{\!\!-1}\left|\frac{\Gamma(N+1/2+\imath \kappa h_1/\pi)}{\Gamma(1/2+\imath \kappa h_1/\pi)}\right|^2.
\end{equation*}
Therefore, substituting into \eqref{eq:norm vandermonde_} and using \eqref{eq:norm proof 1_}, the norm can be expressed as
\begin{equation} \label{eq:norm proof 2_}
    \left\|\mathbf{A}^{-1}\right\|_{\infty}=\left|\frac{\Gamma(N+1/2+\imath \kappa h_1/\pi)}{\Gamma(1/2+\imath \kappa h_1/\pi)}\right|^2 \max_{1\leq n\leq N}\frac{2n-1}{1/4+(\kappa h_1/\pi)^2+n(n-1)}\;\frac{1}{(N+n-1)!(N-n)!}.
\end{equation}
It is straightforward to check that
\begin{equation*}
    \sup_{x\in \mathbb{R}} \frac{2x-1}{1/4+(\kappa h_1/\pi)^2+x(x-1)}=\frac{\pi}{\kappa h_1}, \qquad \text{and} \qquad \inf_{n \in \mathbb{N}^*} (N+n-1)!(N-n)!=N!(N-1)!,
\end{equation*}
and hence, since $\Gamma(N)=(N-1)!$, the norm in \eqref{eq:norm proof 2_} can be bounded as
\begin{equation} \label{eq:bound norm}
    \left\|\mathbf{A}^{-1}\right\|_{\infty}\leq \frac{\pi}{N\kappa h_1}\left|\frac{\Gamma(N+1/2+\imath \kappa h_1/\pi)}{\Gamma(N)\Gamma(1/2+\imath \kappa h_1/\pi)}\right|^2.
\end{equation}
Then, using \cite[Eqs.\ (5.4.4), (5.6.4) and (5.6.6)]{DLMF}, namely
\begin{equation*}
    \left|\Gamma\left(\frac{1}{2}+\imath x\right)\right|^2=\frac{\pi}{\cosh(\pi x)}, \qquad \frac{\Gamma(x+1/2)}{\Gamma(x+1)}<\frac{1}{x^{1/2}}, \qquad |\Gamma(x+\imath y)|\leq |\Gamma(x)|, \quad \qquad x,y \in \mathbb{R},
\end{equation*}
and since $\Gamma(N+1)=N\Gamma(N)$, we estimate
\begin{equation*}
    \frac{1}{\cosh(\kappa h_1)}\left|\frac{\Gamma(N+1/2+\imath \kappa h_1/\pi)}{\Gamma(N)\Gamma(1/2+\imath \kappa h_1/\pi)}\right|^2\leq\frac{1}{\pi} \left|\frac{\Gamma(N+1/2)}{\Gamma(N)}\right|^2\leq \frac{N}{\pi},
\end{equation*}
which implies the following bound for \eqref{eq:bound norm}:
\begin{equation} \label{eq: bound A}
    \left\|\mathbf{A}^{-1}\right\|_{\infty}\leq \frac{\pi}{N\kappa h_1}\,\frac{N\cosh(\kappa h_1)}{\pi}\leq\frac{e^{\kappa h_1}}{\kappa h_1}.
\end{equation}
Hence, combining this estimate with the one in \eqref{eq: bound b}, we obtain:
\begin{equation} \label{eq: bound1}
    \left\|\boldsymbol{\beta}\right\|_{\infty}=\max_{\mathbf{q}\in \mathcal{N}_h}\left\|\boldsymbol{\beta}_{\mathbf{q}}\right\|_{\infty}\leq \left\|\mathbf{A}^{-1}\right\|_{\infty}\max_{\mathbf{q}\in \mathcal{N}_h}\left\|\mathbf{b}_{\mathbf{q}}\right\|_{\infty}\leq \frac{2\rho \max(1,\kappa h)^2}{h}\frac{\rho e^{\kappa h}}{\kappa h}\|u\|_{H_\kappa^{2N}(\Omega)},
\end{equation}
where in the last inequality $h_1$ is bounded in terms of $h$ and $\rho$.

Besides, if we multiply the system in \eqref{eq: collocation sys} on the left by the diagonal matrix $\mathbf{K}$ defined as
\begin{equation*}
    \mathbf{K}_{nn} = (\kappa h_1)^{2(n-1)}, \qquad 1\leq n\leq N,
\end{equation*}
then the matrix $\mathbf{KA}$ has the same structure as $\mathbf{A}$ in \eqref{eq: matrix A}, with the particular choice $\kappa h_1\!=\!1$, i.e.\
\begin{equation*}
    \left(\mathbf{KA}\right)_{nm}=(-1)^{n+m}\left[\frac{(2n-1)\pi}{2}\right]^{2(n-1)}.
\end{equation*}
Hence, the estimate in \eqref{eq: bound A} applies directly to $\|(\mathbf{KA})^{-1}\|_{\infty}$, yielding
\begin{equation} \label{eq: bound2}
    \left\|\boldsymbol{\beta}\right\|_{\infty}\leq \left\|(\mathbf{KA})^{-1}\right\|_{\infty}\|\mathbf{K}\|_{\infty}\max_{\mathbf{q}\in \mathcal{N}_h}\left\|\mathbf{b}_{\mathbf{q}}\right\|_{\infty}\leq\frac{2\rho \max(1,\kappa h)^2}{h}e\max(1,\kappa h)^{2(N-1)}\|u\|_{H_\kappa^{2N}(\Omega)},
\end{equation}
where in the last inequality $h_1$ has been bounded in terms of $h$ and $\rho$.
Finally, taking the minimum between the bounds in \eqref{eq: bound1} and \eqref{eq: bound2}, we obtain the estimate in \eqref{eq:beta_pn0_p}. \qedhere
\end{proof}

For any $N \in \mathbb{N}^*$, we construct the global nodal collocation operator onto $V_N(\mathcal{N}_h)$ as
\begin{equation} \label{eq: collocation operator}
    \boldsymbol{\mathcal{V}}:H_\kappa^{2N}(\Omega)\rightarrow H_\kappa^{1}(\Omega), \qquad u \mapsto \sum_{\mathbf{p}\in \mathcal{N}_h}\sum_{m=1}^N\beta_{\mathbf{p},m}\psi_{\mathbf{p},m},
\end{equation}
where $\boldsymbol{\beta}=(\beta_{\mathbf{p},m})_{\mathbf{p},m}$ is the solution of the collocation system \eqref{eq: local collocation operator condition}. Proposition~\ref{th: nodal collocation} ensures that, by construction, $\partial_1^{2n}\boldsymbol{\mathcal{V}}u(\mathbf{p})=\partial_1^{2n} u(\mathbf{p})$ for any $u \in H^{2N}_\kappa(\Omega)$, $0\leq n\leq N-1$, and $\mathbf{p}\in \mathcal{N}_h$.

\begin{lemma} \label{lem: V continuity}
    Let $N \in \mathbb{N}^*$ and $M \in \mathbb{N}$. The operator $\boldsymbol{\mathcal{V}}:H_\kappa^{2N}(\Omega)\rightarrow H_\kappa^{M}(\mathcal{T}_h)$ is bounded and
    \begin{equation} \label{eq: continuity collocation operator}
        \|\boldsymbol{\mathcal{V}}u\|_{H_\kappa^{M}(\mathcal{T}_h)}\leq C_{2}\chi_{N}(M+1)\left[N+C_{3}N^{\frac{3}{2}}(\rho\nu_{N})^{M-1/2} \right]\|u\|_{H_\kappa^{2N}(\Omega)},
    \end{equation}
    where $\nu_N$ and $\chi_{N}$ are defined in \eqref{eq: nuN} and \eqref{eq:beta_pn0_p}, respectively, and
    \begin{equation} \label{eq: constants main theorem 2 and 3}
        C_{2}:=\frac{2\pi|\mathcal{N}_h|\rho^{3/2} \max(1,\kappa h)^2}{\sqrt{2}\kappa h\widetilde{\mathrm{d}}}, \qquad C_{3}:=\frac{2\widetilde{\mathrm{d}}}{\sqrt{\pi \rho \tanh(\kappa h/\rho)\widetilde{\mathrm{d}}_0}},
    \end{equation}
    with
    \begin{equation} \label{eq: definition d_m tilde}
        \widetilde{\mathrm{d}}:=\inf_{\substack{m \in \mathbb{N},n \in \mathbb{N}^*:\\\widetilde{\nu}_n<1}}\left|\sqrt{1-\widetilde{\nu}_n^2}-\frac{m\pi}{\kappa h_2}\right|, \qquad \widetilde{\mathrm{d}}_0:=\inf_{\substack{n \in \mathbb{N}^*:\\\widetilde{\nu}_n>1}}\sqrt{\widetilde{\nu}_n^2-1},
    \end{equation}
    and $\widetilde{\nu}_n$ introduced in \eqref{eq: psi_n properties}.
\end{lemma}
\begin{proof}
    Observe that $\widetilde{\mathrm{d}}>0$ and $\widetilde{\mathrm{d}}_0>0$ thanks to Assumption \ref{A1+A2}.
    From the definitions of the node basis functions \eqref{eq: definition basis nodes} and the operator \eqref{eq: collocation operator}, we have
    \begin{align}
        \|\boldsymbol{\mathcal{V}}u\|^2_{H_\kappa^{M}(\mathcal{T}_h)}&=\sum_{K \in \mathcal{T}_h}\sum_{|\boldsymbol{\alpha}|\leq M}\kappa^{-2|\boldsymbol{\alpha}|}\left\|\sum_{\mathbf{p}\in \mathcal{N}_h}\sum_{m=1}^N\beta_{\mathbf{p},m}\partial^{\boldsymbol{\alpha}}\psi_{\mathbf{p},m}\right\|^2_{L^2(K)} \nonumber\\
        &\leq N|\mathcal{N}_h|\|\boldsymbol{\beta}\|^2_{\infty}\sum_{\mathbf{p}\in \mathcal{N}_h}\sum_{m=1}^N\sum_{K \in \mathcal{T}_h}\|\psi_{\mathbf{p},m}\|^2_{H_\kappa^M(K)} \nonumber\\
        &= N|\mathcal{N}_h|\|\boldsymbol{\beta}\|^2_{\infty}\sum_{\mathbf{p}\in \mathcal{N}_h}\sum_{m=1}^N\left(\|\psi_{\mathbf{p},m}\|^2_{H_\kappa^M(K^+_{\mathbf{p}}\cap\, \Omega)}+\|\psi_{\mathbf{p},m}\|^2_{H_\kappa^M(K^-_{\mathbf{p}}\cap\, \Omega)}\right) \nonumber\\
        &\leq 2N|\mathcal{N}_h|^2\|\boldsymbol{\beta}\|^2_{\infty}\sum_{m=1}^N\|\widehat{\varphi}_{2m-1}\|^2_{H_\kappa^M(\widehat{K})}, \label{eq: STST1}
    \end{align}
    where, in the last inequality, we exploit the fact that, for any $\mathbf{p}\in \mathcal{N}_h$, the restriction of $\psi_{\mathbf{p},m}$ to its adjacent cells $K^{\pm}_{\mathbf{p}}$ in \eqref{eq: adjacent cells nodal} coincides, up to a rigid transformation, with  the reference function $\widehat{\varphi}_{2m-1}$ on $\widehat{K}$ taking $(\widehat{h}_1,\widehat{h}_2)=(2h_1,h_2)$.
    From the upper bounds in Lemma \ref{lem:orthogonality lemma}, one obtains
    \begin{align*}
        \sum_{m=1}^N\|\widehat{\varphi}_{2m-1}\|^2_{H_\kappa^M(\widehat{K})}&\leq \sum_{\substack{m=1  \\ \widetilde{\nu}_m<1}}^N\|\widehat{\varphi}_{2m-1}\|^2_{H_\kappa^M(\widehat{K})}+\sum_{\substack{m=1  \\ \widetilde{\nu}_m>1}}^N\|\widehat{\varphi}_{2m-1}\|^2_{H_\kappa^M(\widehat{K})}\\
        &\leq\frac{Nh_1\pi^2(M+1)^2}{4\kappa^2h_2\widetilde{\mathrm{d}}^2}+\frac{h_1(M+1)^2}{\kappa \tanh(\kappa h_2)}\sum_{\substack{m=1  \\ \widetilde{\nu}_m>1}}^N\frac{\widetilde{\nu}^{2M}_m}{\sqrt{\widetilde{\nu}^{2}_m-1}}\\
        &\leq\frac{N\pi^2\rho(M+1)^2}{4\kappa^2\widetilde{\mathrm{d}}^2}+\frac{N^2\pi(M+1)^2(\rho \nu_N)^{2M-1}}{\kappa^2 \tanh(\kappa h/\rho)\widetilde{\mathrm{d}}_0},
    \end{align*}
    where $\widetilde{\mathrm{d}}$ and $\widetilde{\mathrm{d}}_0$ are defined in \eqref{eq: definition d_m tilde}, and the last inequality follows from the bound $\widetilde{\nu}_m^{2M}\leq \tfrac{N\pi}{\kappa h_1}(\rho\nu_N)^{2M-1}$ for any $1\leq m\leq N$, with $\rho$ the shape-regularity constant from \eqref{eq: shape regularity}.
    Hence,
    \begin{equation} \label{eq: STST2}
        \frac{\kappa^2}{\pi(M+1)^2}\sum_{m=1}^N\|\widehat{\varphi}_{2m-1}\|^2_{H_\kappa^M(\widehat{K})}\!\leq\! \frac{N\pi \rho}{4\widetilde{\mathrm{d}}^2}\!+\!\frac{N^2(\rho \nu_N)^{2M-1}}{\tanh(\kappa h/\rho)\widetilde{\mathrm{d}}_0}\!\leq\! \Bigg(\frac{\sqrt{N\pi \rho}}{2\widetilde{\mathrm{d}}}\!+\!\frac{N(\rho \nu_N)^{M-1/2}}{\sqrt{\tanh(\kappa h/\rho)\widetilde{\mathrm{d}}_0}}\Bigg)^{\!\!2},
    \end{equation}
    and the bound in \eqref{eq: continuity collocation operator} is derived by combining \eqref{eq: STST1} and \eqref{eq: STST2} with the coefficient bound \eqref{eq:beta_pn0_p}.
\end{proof}

\begin{remark}
Similarly to $\widehat{\mathrm{d}}$ in \eqref{eq: definition d_m}, the quantities $\widetilde{\mathrm{d}}$ and $\widetilde{\mathrm{d}}_0$ introduced in \eqref{eq: definition d_m tilde} provide a measure of the gap between the wavenumber $\kappa$ and the node-associated resonant frequencies
\begin{equation} \label{eq: resonant freq nodes}
\left\{ \left( \frac{(2n-1)\pi}{2h_1} \right)^2  + \left( \frac{m\pi}{h_2} \right)^2 \right\}_{n \in \mathbb{N}^*, m \in \mathbb{N}},
\end{equation}
which are excluded under Assumption \emph{\ref{A1+A2}}.
If, on the other hand, $\kappa^2$ coincides with one of these frequencies, one may adjust the mesh to avoid such resonances, as discussed in Remark \emph{\ref{rem: resonant frequencies edge}}.
An alternative strategy is to remove from the nodal approximation set $\{\psi_{\mathbf{p},n}\}_{\mathbf{p}\in \mathcal{N}_h,1\leq n \leq N}$ those basis functions for which there exists $m\in\mathbb{N}$ such that $\kappa^2 = ((2n-1)\pi/(2h_1))^2 + \big(m\pi/h_2)^2$, since in this case the corresponding functions are not well defined. These functions can then be replaced by an equal number of additional node basis functions with indices $n>N$ for which no such resonance occurs. The resulting set remains linearly independent, and thus the collocation system \eqref{eq: local collocation operator condition} is still solvable.
However, a detailed analysis of the resulting estimates -- as in Proposition \emph{\ref{th: nodal collocation}} and Lemma \emph{\ref{lem: V continuity}} -- becomes more involved in this case and it is not addressed here.
\end{remark}

\subsection{Relations between node-based and edge-based spaces}

Let $N \in \mathbb{N}^*$ and $u \in H^{2N}_\kappa(\Omega)$. By Proposition~\ref{th: nodal collocation}, we have that $\partial_1^{2n}\boldsymbol{\mathcal{V}}u(\mathbf{p}) = \partial_1^{2n} u(\mathbf{p})$ for all $0 \le n \le N-1$ and all nodes $\mathbf{p} \in \mathcal{N}_h$. The next lemma shows that, if $u$ is furthermore a homogeneous Helmholtz solution, the same holds for the even $\partial_2$-derivatives, namely $\partial_2^{2n}\boldsymbol{\mathcal{V}}u(\mathbf{p}) = \partial_2^{2n} u(\mathbf{p})$.

\begin{lemma} \label{lem: zero_nodes_lemma}
    Let $N \in \mathbb{N}^*$, $K$ be a rectangle, and $u \in H_\kappa^{2N}(K)$ be an homogeneous Helmholtz solution. Then, for any $\mathbf{p}\in \overline{K}$ and any $0 \leq n \leq N-1$, the following equivalence holds:
    \begin{equation} \label{eq: zero_nodes_}
        \partial_1^{2n}u(\mathbf{p})=0 \quad  \Longleftrightarrow \quad  \partial_2^{2n}u(\mathbf{p})=0.
    \end{equation}
\end{lemma}
\begin{proof}
    Since $u \in H_\kappa^{2N}(K)$, the derivatives $\partial_i^{2n} u(\mathbf{p})$ are well-defined for any $\mathbf{p}\in \overline{K}$, $0 \leq n \leq N-1$, and $i=1,2$.
    Moreover, $u$ solves the homogeneous Helmholtz equation in $K$, so it follows that $(\kappa^{-1}\partial_1)^{2n}u = \left[-\left(\kappa^{-2}\partial_2^2 + \textup{Id}\right)\right]^{n}u$, for any $n \in \mathbb{N}$.
    Hence, one obtains the binomial identity
    \begin{equation} \label{eq:zero1_}
        \mathbf{b}^{(1)}_\mathbf{p} = \mathbf{L} \mathbf{b}^{(2)}_{\mathbf{p}}, \qquad \mathbf{p} \in \overline{K},
    \end{equation}
    where the components, indexed by $0\leq n,m \leq N-1$, are defined as
    \begin{equation*} 
        \mathbf{L}_{nm} := \begin{cases}
            (-1)^{n} \binom{n}{m} & \text{if } n \geq m \\
            0 & \text{if } n < m
        \end{cases},
        \qquad \qquad \left(\mathbf{b}^{(i)}_{\mathbf{p}}\right)_n:=\left(\kappa^{-1}\partial_i\right)^{2n}u(\mathbf{p}), \qquad i=1,2.
    \end{equation*}
    Since $\mathbf{L}$ is invertible, \eqref{eq:zero1_} yields that $\mathbf{b}^{(1)}_\mathbf{p}=\mathbf{0}$ if and only if $\mathbf{b}^{(2)}_\mathbf{p}=\mathbf{0}$.
\end{proof}

We now prove that for every sufficiently regular Helmholtz solution $u$, the function $u-\boldsymbol{\mathcal{V}}u$ belongs to the global set $\mathcal{S}_\kappa^{N,M}(\mathcal{T}_h)$. This result is fundamental in connecting the two approximation spaces and sets the stage for the subsequent convergence results.

\begin{proposition} \label{prop: u-Vu}
    Let $N,M\in \mathbb{N}^*$ such that $N\leq \lceil M/2\rceil$, and $u \in H_\kappa^{M+1}(\Omega)$ an homogeneous Helmholtz solution. Then
    \begin{equation*}
        u-\boldsymbol{\mathcal{V}}u \in \mathcal{S}_\kappa^{N,M}(\mathcal{T}_h)\cap H_\kappa^{M+1}(\mathcal{T}_h).
    \end{equation*}
\end{proposition}
\begin{proof}
     Since $u \in H_\kappa^{M+1}(\Omega)$ and $\boldsymbol{\mathcal{V}}u$ is piecewise analytic, we have that $u - \boldsymbol{\mathcal{V}}u \in H_\kappa^{M+1}(\mathcal{T}_h)$.
     
     We now verify that $u - \boldsymbol{\mathcal{V}}u \in \mathcal{S}_\kappa^{N,M}(\mathcal{T}_h)$ by checking the conditions in Definition \ref{def: global set}.   
     As $u \in H_\kappa^{M+1}(\Omega)$ is a homogeneous Helmholtz solution in $\Omega$ and $\boldsymbol{\mathcal{V}}u$ belongs to the $H^1_\kappa(\Omega)$-conforming Trefftz space $V_N(\mathcal{N}_h)$, it follows that $u - \boldsymbol{\mathcal{V}}u\in H^1_\kappa(\Omega)$, and satisfies the homogeneous Helmholtz equation locally in each element $K \in \mathcal{T}_h$.     
     Moreover, Lemma \ref{lem: trace inequality} ensures that $(u-\boldsymbol{\mathcal{V}}u)|_{\mathbf{s}}\in H_\kappa^M(\mathbf{s})$ for any $\mathbf{s}\in \Sigma_h$.    
     Thanks to Lemma \ref{lem: psi_n properties} and the fact that $u \in H_\kappa^{M+1}(\Omega) \subset H_\kappa^{2N}(\Omega)$, we have $\partial_i^{2n}(u - \boldsymbol{\mathcal{V}}u) \in C^0(\overline{\Omega})$ for all $0 \leq n \leq N-1$ and $i = 1,2$.
     Moreover, by Proposition \ref{th: nodal collocation}, we know that $\partial_1^{2n}(u - \boldsymbol{\mathcal{V}}u)(\mathbf{p}) = 0$ for every $\mathbf{p} \in \mathcal{N}_h$ and $0 \leq n \leq N-1$.
     Since $u - \boldsymbol{\mathcal{V}}u \in H_\kappa^{2N}(\mathcal{T}_h)$ satisfies the homogeneous Helmholtz equation on each element $K \in \mathcal{T}_h$, Lemma \ref{lem: zero_nodes_lemma} implies that also $\partial_2^{2n}(u - \boldsymbol{\mathcal{V}}u)(\mathbf{p}) = 0$ for all $\mathbf{p} \in \mathcal{N}_h$ and $0 \leq n \leq N-1$.
     Therefore, the vanishing derivative conditions in Definition \ref{def: global set} are satisfied, and $u - \boldsymbol{\mathcal{V}}u \in \mathcal{S}_\kappa^{N,M}(\mathcal{T}_h)$.
\end{proof}

\section{Combined edge-and-node Trefftz space} \label{sec: Combined edge-and-node Trefftz space}

This section introduces the combined edge-and-node Trefftz space and its associated interpolation operator, providing error and stability estimates. These include bounds for fixed $\kappa h$ as the edge and node parameters, $N_\mathbf{e}$ and $N_\mathbf{n}$, vary, as well as results for the high-frequency regime.

\subsection{Construction of the combined Trefftz space}

Let us introduce the combined Trefftz space, obtained by summing the edge-based space $V_{N_{\mathbf{e}}}(\Sigma_h)$ and the node-based space $V_{N_{\mathbf{n}}}(\mathcal{N}_h)$, presented in \eqref{eq: edge-based Trefftz space} and \eqref{eq: node-based Trefftz space}, respectively.

\begin{definition}[Combined edge-and-node Trefftz space] \label{def: edge-and-node Trefftz space}
For any $N_{\mathbf{e}},N_{\mathbf{n}}\in \mathbb{N}^*$, define
\begin{equation} \label{eq: combined Trefftz space}
    V_{\mathbf{N}}(\mathcal{T}_h) :=  V_{N_{\mathbf{e}}}(\Sigma_h) \oplus V_{N_\mathbf{n}}(\mathcal{N}_h), \qquad \text{with} \qquad \mathbf{N} := (N_{\mathbf{e}}, N_\mathbf{n}).
\end{equation}
\end{definition}

Owing to the properties of these two approximation spaces -- see \eqref{eq: edge-based Trefftz space conformity} and \eqref{eq: node-based Trefftz space conformity} -- the combined space $V_{\mathbf{N}}(\mathcal{T}_h)$ is itself a Trefftz space and inherits global conformity. In particular, it consists of globally $H_\kappa^1(\Omega)$-regular functions that satisfy the homogeneous Helmholtz equation element-wise:
\begin{equation*}
    V_{\mathbf{N}}(\mathcal{T}_h)\subset \{u \in H_\kappa^1(\Omega): -\Delta u -\kappa^2 u=0 \quad\!\! \text{in any}\quad\!\!\!\! K \in \mathcal{T}_h\}.
\end{equation*}

The combined interpolation operator onto $V_{\mathbf{N}}(\mathcal{T}_h)$ is constructed by combining edge and node contributions.
Let $M \in \mathbb{N}^*$ and $u \in H^{M+1}_\kappa(\Omega)$ be an homogeneous Helmholtz solution.
The operator is designed so that $u - \boldsymbol{\mathcal{V}} u$ can be effectively approximated within the edge-based space $V_{N_{\mathbf{e}}}(\Sigma_h)$.
Since at most $N_{\mathbf{n}}$ node basis functions are available, and $u$ has at most Sobolev-regularity $M+1$, the collocation system \eqref{eq: local collocation operator condition} -- and hence $\boldsymbol{\mathcal{V}}u$ -- is well-defined for at most
\begin{equation*}
    N=\min(N_{\mathbf{n}},\lceil M/2\rceil).
\end{equation*}
This allows us to define the combined interpolant as
\begin{equation*}
    \boldsymbol{\mathcal{I}}:H_\kappa^{M+1}(\Omega)\rightarrow H_\kappa^{1}(\Omega), \qquad \boldsymbol{\mathcal{I}}:=\boldsymbol{\mathcal{E}}\left(\textrm{Id}-\boldsymbol{\mathcal{V}}\right)+\boldsymbol{\mathcal{V}},
\end{equation*}
with $\boldsymbol{\mathcal{E}}$ the edge-based interpolant \eqref{eq: global interpolant} and $\boldsymbol{\mathcal{V}}$ the nodal operator \eqref{eq: collocation operator}. More explicitly, we have
\begin{align}
    &\boldsymbol{\mathcal{I}}u=\boldsymbol{\mathcal{E}}(\textrm{Id}-\boldsymbol{\mathcal{V}})u+\boldsymbol{\mathcal{V}}u=\sum_{\mathbf{s}\in \Sigma_h}\sum_{n=1}^{N_{\mathbf{e}}}\alpha_{\mathbf{s},n}\phi_{\mathbf{s},n}+\sum_{\mathbf{p}\in \mathcal{N}_h}\sum_{m=1}^{N}\beta_{\mathbf{p},m}\psi_{\mathbf{p},m}, \qquad u \in H^{M+1}_\kappa(\Omega), \nonumber \\
    &\text{where}\qquad \qquad \alpha_{\mathbf{s},n}:=\frac{( (u-\boldsymbol{\mathcal{V}}u),\phi_{\mathbf{s},n})_{L^2(\mathbf{s})}}{\| \phi_{\mathbf{s},n}\|^2_{L^2(\mathbf{s})}}, \qquad \qquad\beta_{\mathbf{p},m} \text{ solutions to \eqref{eq: local collocation operator condition}}. \qquad \qquad \label{eq: coefficient final alpha}
\end{align}

\subsection{Accuracy with respect to edge and node parameters} \label{sec: Accuracy with respect to edge and node parameters}

We begin by presenting best-approximation error estimates, focusing on their dependence on the parameters $N_{\mathbf{e}}, N_{\mathbf{n}} \in \mathbb{N}^*$. The results are stated in full generality and then specialized to two main cases: a Helmholtz solution analytic on $\overline{\Omega}$ and one of low Sobolev-regularity in $H^2_\kappa(\Omega)$.

\begin{theorem}[Edge-and-node based approximation result] \label{th: best approximation}
    Let $N_\mathbf{e},N_\mathbf{n},M\in \mathbb{N}^*\!$, $r\in\{0,1\}$, and assume that \eqref{eq: nuN} holds.
    Then, for any homogeneous Helmholtz solution $u\in H_\kappa^{M+1}(\Omega)$, the best-approximation error in the combined edge-and-node based Trefftz space $V_{\mathbf{N}}(\mathcal{T}_h)$ satisfies: 
        \begin{align*}
        &\inf_{v \in V_{\mathbf{N}}(\mathcal{T}_h)}\|u-v\|_{H_{\kappa}^r(\Omega)}\leq\|u-\boldsymbol{\mathcal{I}}u\|_{H_\kappa^r(\Omega)}\\
        & \qquad\leq  C_{1}\left[1+C_{2}\chi_{\lceil \mu/2 \rceil}(\mu+2)\left(\lceil \tfrac{\mu}{2} \rceil+C_{3}\lceil \tfrac{\mu}{2} \rceil^{\frac{3}{2}}\left(\frac{\pi \rho\lceil \tfrac{\mu}{2} \rceil}{\kappa h}\right)^{\mu+\frac{1}{2}} \right)\right]\nu^{r-(\mu+1/2)}_{N_{\mathbf{e}}}\|u\|_{H_\kappa^{\mu+1}(\Omega)},
        \end{align*}
    where $\mu=\min(2N_\mathbf{n},M)$, and $\chi_{\lceil \mu/2 \rceil}$ and $C_{i}>0$ for $i=1,2,3$ are defined in \eqref{eq:beta_pn0_p}, \eqref{eq: constants main theorem 1} and \eqref{eq: constants main theorem 2 and 3}.
\end{theorem}
\begin{proof}
    Since $N=\min(N_{\mathbf{n}},\lceil M/2 \rceil)\leq N_{\mathbf{n}}$, it follows that $\boldsymbol{\mathcal{V}}u\in V_{N}(\mathcal{N}_h) \subset V_{N_{\mathbf{n}}}(\mathcal{N}_h)$, and hence $\boldsymbol{\mathcal{I}}u\in V_\mathbf{N}(\mathcal{T}_h)$, with $\mathbf{N}=(N_{\mathbf{e}},N_{\mathbf{n}})$. Therefore:
    \begin{equation} \label{eq: main proof step 1}
        \inf_{v \in V_{\mathbf{N}}(\mathcal{T}_h)}\|u-v\|_{H_{\kappa}^r(\Omega)}\leq\|u-\boldsymbol{\mathcal{I}}u\|_{H_{\kappa}^r(\Omega)}=\|(u-\boldsymbol{\mathcal{V}}u)-\boldsymbol{\mathcal{E}}(u-\boldsymbol{\mathcal{V}}u)\|_{H_{\kappa}^r(\mathcal{T}_h)}.
    \end{equation}
    As $u \in H_\kappa^{M+1}(\Omega)$ is a Helmholtz solution and $N \leq \lceil M/2 \rceil$, Proposition \ref{prop: u-Vu} implies that $u-\boldsymbol{\mathcal{V}}u \in \mathcal{S}_\kappa^{N,M}(\mathcal{T}_h)\cap H_\kappa^{M+1}(\mathcal{T}_h)$.
    Moreover $r\leq \mu= \min(2N_{\mathbf{n}},M)=\min(2N,M)$, so, applying Proposition~\ref{th: main theorem_2} to $u-\boldsymbol{\mathcal{V}}u$, from \eqref{eq: main proof step 1} it follows
    \begin{equation*}
        \|u-\boldsymbol{\mathcal{I}}u\|_{H_{\kappa}^r(\Omega)}
        \leq C_{1} \nu^{r-(\mu+1/2)}_{N_{\mathbf{e}}}\left(\| u\|_{H_{\kappa}^{\mu+1}(\Omega)}+\|\boldsymbol{\mathcal{V}}u\|_{H_{\kappa}^{\mu+1}(\mathcal{T}_h)}\right),
    \end{equation*}
    where $C_{1}>0$ is defined in \eqref{eq: constants main theorem 1}.
    Hence, applying Lemma \ref{lem: V continuity} to bound $\|\boldsymbol{\mathcal{V}}u\|_{H_\kappa^{\mu+1}(\mathcal{T}_h)}$, we obtain
    \begin{align*}
        \frac{\|u-\boldsymbol{\mathcal{I}}u\|_{H_{\kappa}^r(\Omega)}}{C_{1}\nu^{r-(\mu+1/2)}_{N_{\mathbf{e}}}}
        & \leq \| u\|_{H_{\kappa}^{\mu+1}(\Omega)}+C_{2}\chi_{N}(\mu+2)\left(N+C_{3}N^{\frac{3}{2}}(\rho\nu_{N})^{\mu+\frac{1}{2}}\right) \|u\|_{H_\kappa^{2N}(\Omega)}\\
        &\leq \left[1+C_{2}\chi_{N}(\mu+2)\left(N+C_{3}N^{\frac{3}{2}}\left(\frac{\pi \rho N}{\kappa h}\right)^{\mu+\frac{1}{2}} \right)\right]\|u\|_{H_\kappa^{\max(\mu+1,2N)}(\Omega)}\\
        &\leq \left[1+C_{2}\chi_{\lceil \mu/2\rceil}(\mu+2)\left(\lceil \tfrac{\mu}{2}\rceil+C_{3}\lceil \tfrac{\mu}{2}\rceil^{\frac{3}{2}}\left(\frac{\pi \rho\lceil \tfrac{\mu}{2}\rceil}{\kappa h}\right)^{\mu+\frac{1}{2}} \right)\right]\|u\|_{H_\kappa^{\mu+1}(\Omega)},
    \end{align*}
    where in the last inequality we exploit the fact that $N = \lceil \mu/2\rceil$ and $2N \leq \mu+1$.
\end{proof}

\paragraph{Analytic Helmholtz solution.}

As a consequence, when the exact solution $u$ is analytic on $\overline{\Omega}$, the approximation error achieves \textit{spectral convergence rates}.
This behavior is formalized in the following corollary.

\begin{corollary}[Analytic Helmholtz solution] \label{cor: corollary Cinfty}
Let $N_{\mathbf{e}}, N_{\mathbf{n}}\in \mathbb{N}^*$, $r\in\{0,1\}$, and let $u$ be a homogeneous Helmholtz solution analytic on $\overline{\Omega}$. Assume that the condition \eqref{eq: nuN} holds.
\begin{itemize}
    \item \emph{(Algebraic convergence in $N_{\mathbf{e}}$).} There exists an explicit $C=C(\rho,\kappa h,r,N_{\mathbf{n}})>0$ such that
    \begin{equation*}
        \inf_{v \in V_{\mathbf{N}}(\mathcal{T}_h)}\|u-v\|_{H_{\kappa}^r(\Omega)}\leq\|u-\boldsymbol{\mathcal{I}}u\|_{H_\kappa^r(\Omega)}\leq C N_{\mathbf{e}}^{-(2N_{\mathbf{n}}+\frac{1}{2}-r)}\|u\|_{H_\kappa^{2N_{\mathbf{n}}+1}(\Omega)}.
    \end{equation*}
    \item \emph{(Geometric convergence in $N_{\mathbf{n}}$).} If the number of basis functions $N_{\mathbf{e}}$ and $N_{\mathbf{n}}$ are such that
    \begin{equation*}
        N_{\mathbf{e}} = \tau_{\mathbf{e}} N_{\mathbf{n}}, \qquad \text{where} \qquad \tau_{\mathbf{e}} > \rho,
    \end{equation*}
    then there exists an explicit constant $C=C(\rho,\kappa h,r)>0$ such that
    \begin{equation*}
        \inf_{v \in V_{\mathbf{N}}(\mathcal{T}_h)}\|u-v\|_{H_{\kappa}^r(\Omega)}\leq\|u-\boldsymbol{\mathcal{I}}u\|_{H_\kappa^r(\Omega)}\leq C N_{\mathbf{n}}^{\frac{5}{2}+r}\left(\frac{\rho}{\tau_{\mathbf{e}}}\right)^{2N_{\mathbf{n}}+\frac{1}{2}-r}\|u\|_{H_\kappa^{2N_{\mathbf{n}}+1}(\Omega)}.
    \end{equation*}
\end{itemize}

\end{corollary}
\begin{proof}
    The first bound follows directly from Theorem \ref{th: best approximation} with $\mu = 2N_{\mathbf{n}}$, namely
    \begin{equation} \label{eq: AA}
        \frac{\|u-\boldsymbol{\mathcal{I}}u\|_{H_\kappa^r(\Omega)}}{\|u\|_{H_\kappa^{2N_{\mathbf{n}}+1}(\Omega)}}\leq C_{1}\left[1+4C_{2}\chi_{N_{\mathbf{n}}}\left(N_{\mathbf{n}}^{2}+C_{3}N_{\mathbf{n}}^{\frac{5}{2}}\left(\frac{\rho N_{\mathbf{n}}\pi}{\kappa h}\right)^{2N_{\mathbf{n}}+\frac{1}{2}} \right)\right]\left(\frac{\kappa h}{\pi N_{\mathbf{e}}}\right)^{2N_{\mathbf{n}}+\frac{1}{2}-r},
    \end{equation}
    where $\chi_{N_{\mathbf{n}}}$ and the constants $C_{i}>0$ for $i=1,2,3$ are defined in \eqref{eq:beta_pn0_p}, \eqref{eq: constants main theorem 1} and \eqref{eq: constants main theorem 2 and 3}, respectively.
    
    As for the second bound, setting $N_{\mathbf{e}} = \tau_{\mathbf{e}} N_{\mathbf{n}}$ in \eqref{eq: AA}, one obtains
    \begin{align} 
        &\frac{\|u-\boldsymbol{\mathcal{I}}u\|_{H_\kappa^r(\Omega)}}{\|u\|_{H_\kappa^{2N_{\mathbf{n}}+1}(\Omega)}}\leq C_{1}\left[1+4C_{2}\chi_{N_{\mathbf{n}}}\left(N_{\mathbf{n}}^{2}+C_{3}N_{\mathbf{n}}^{\frac{5}{2}}\left(\frac{\rho N_{\mathbf{n}}\pi}{\kappa h}\right)^{2N_{\mathbf{n}}+\frac{1}{2}} \right)\right]\left(\frac{\kappa h}{\pi \tau_{\mathbf{e}}N_{\mathbf{n}}}\right)^{2N_{\mathbf{n}}+\frac{1}{2}-r} \nonumber\\
        &\quad \leq C_{1}\left[1+4C_{2}\chi_{N_{\mathbf{n}}}\left(N_{\mathbf{n}}^{2}+C_{3}N_{\mathbf{n}}^{\frac{5}{2}}\left(\frac{N_{\mathbf{n}}\pi}{\kappa h}\right)^{2N_{\mathbf{n}}+\frac{1}{2}} \right)\right]\left(\frac{\kappa h}{N_{\mathbf{n}}\pi}\right)^{2N_{\mathbf{n}}+\frac{1}{2}}\left(\frac{\rho N_{\mathbf{n}}\pi}{\kappa h}\right)^{r}\left(\frac{\rho}{\tau_{\mathbf{e}}}\right)^{2N_{\mathbf{n}}+\frac{1}{2}-r} \nonumber\\
        &\quad= C_{1}\left[\left(1+4C_{2}\chi_{N_{\mathbf{n}}}N_{\mathbf{n}}^{2}\right)\left(\frac{\kappa h}{\rho N_{\mathbf{n}}\pi}\right)^{2N_{\mathbf{n}}+\frac{1}{2}}+4C_{2}C_{3}\chi_{N_{\mathbf{n}}}N_{\mathbf{n}}^{\frac{5}{2}} \right]\left(\frac{\rho N_{\mathbf{n}}\pi}{\kappa h}\right)^{r}\left(\frac{\rho}{\tau_{\mathbf{e}}}\right)^{2N_{\mathbf{n}}+\frac{1}{2}-r}. \label{eq: BB}
    \end{align}
    Moreover, for any $N_{\mathbf{n}}\in \mathbb{N}^*$,
    \begin{equation} \label{eq: bound chi Cinfty}
    \left(\frac{\kappa h}{\rho N_{\mathbf{n}}\pi}\right)^{2N_{\mathbf{n}}}\leq \exp\left(\frac{2\kappa h}{\rho e \pi}\right), \qquad \chi_{N_{\mathbf{n}}} = \min\Bigg(\frac{\rho e^{\kappa h}}{\kappa h},\, e\,\max(1,\kappa h)^{2(N_{\mathbf{n}}-1)}\Bigg)
    \le \frac{\rho e^{\kappa h}}{\kappa h},
    \end{equation}
    so from \eqref{eq: BB} it follows:
    \begin{align*}
        \frac{\|u-\boldsymbol{\mathcal{I}}u\|_{H_\kappa^r(\Omega)}}{\|u\|_{H_\kappa^{2N_{\mathbf{n}}+1}(\Omega)}}\!
        &\leq \! C_{1}\!\left[\left(\!1\!+\!\frac{4C_{2}\rho }{\kappa h e^{-\kappa h}}N_{\mathbf{n}}^{2}\right)\!\!\left(\frac{\kappa h}{\rho N_{\mathbf{n}}\pi}\right)^{\!\frac{1}{2}}\!\!\!e^{\frac{2\kappa h}{\rho e \pi}}\!+\!\frac{4C_{2}C_{3}\rho}{\kappa h e^{-\kappa h}}N_{\mathbf{n}}^{\frac{5}{2}} \right]\!\!\left(\frac{\rho N_{\mathbf{n}}\pi}{\kappa h}\right)^{\!r}\!\!\!\left(\frac{\rho}{\tau_{\mathbf{e}}}\right)^{\!2N_{\mathbf{n}}+\frac{1}{2}-r}\\
        &\leq \! C_{1}\!\left[\left(\!1\!+\!\frac{4C_{2}\rho }{\kappa h e^{-\kappa h}}\right)\!\!\left(\frac{\kappa h}{\rho\pi}\right)^{\!\frac{1}{2}}\!\!\!e^{\frac{2\kappa h}{\rho e \pi}}\!+\!\frac{4C_{2}C_{3}\rho }{\kappa h e^{-\kappa h}} \right]\!\left(\frac{\rho \pi}{\kappa h}\right)^{\!r}N_{\mathbf{n}}^{\frac{5}{2}+r}\left(\frac{\rho}{\tau_{\mathbf{e}}}\right)^{2N_{\mathbf{n}}+\frac{1}{2}-r}. \qedhere
    \end{align*}
\end{proof}

In Corollary \ref{cor: corollary Cinfty} two distinct regimes are identified:
\begin{itemize}
    \item \textit{Algebraic convergence in $N_\mathbf{e}$ with $N_\mathbf{n}$-dependent rate}: For any fixed nodal degree $N_\mathbf{n}$, the interpolation error decays algebraically in $N_\mathbf{e}$, with a rate that improves as $N_\mathbf{n}$ increases.
    \item \textit{Geometric convergence in $N_\mathbf{n}$ with $N_\mathbf{e}$-dependent rate}:
    If the edge degree $N_\mathbf{e}$ is sufficiently large relative to the nodal degree $N_\mathbf{n}$, for instance $N_\mathbf{e} = \tau_{\mathbf{e}} N_\mathbf{n}$ with $\tau_{\mathbf{e}} > \rho$, then the error decays geometrically in $N_\mathbf{n}$, with a rate that improves as $N_\mathbf{e}$ increases.    
\end{itemize}

\paragraph{Helmholtz solution in $H_\kappa^2(\Omega)$.}

If $u \in H_\kappa^2(\Omega)$, Theorem \ref{th: best approximation} guarantees \textit{algebraic convergence in $N_{\mathbf{e}}$, independently of $N_{\mathbf{n}}$}, as the rate is limited by the regularity of $u$.
\begin{corollary}[Helmholtz solution in $H_\kappa^2(\Omega)$] \label{cor: H2}
    Let $N_\mathbf{e},N_\mathbf{n}\in \mathbb{N}^*$, $r\in\{0,1\}$, and let $u\in H^2_\kappa(\Omega)$ be a homogeneous Helmholtz solution. Assume that the condition \eqref{eq: nuN} holds.
    Then, there exists an explicit constant $C=C(\rho, \kappa h,r)>0$ such that
    \begin{equation*}
        \inf_{v \in V_{\mathbf{N}}(\mathcal{T}_h)}\|u-v\|_{H_{\kappa}^r(\Omega)}\leq\|u-\boldsymbol{\mathcal{I}}u\|_{H_\kappa^r(\Omega)}\leq C N_{\mathbf{e}}^{-(\frac{3}{2}-r)}\|u\|_{H_\kappa^{2}(\Omega)}.
    \end{equation*}
\end{corollary}
\begin{proof}
Choosing $\mu = M=1$ in Theorem \ref{th: best approximation}, and since $\chi_{1} = \min(\rho e^{\kappa h}/(\kappa h), e) \leq  e$, one has
\begin{equation} \label{eq: CC}
        \frac{\|u-\boldsymbol{\mathcal{I}}u\|_{H_\kappa^r(\Omega)}}{\|u\|_{H_\kappa^{2}(\Omega)}}\leq C_{1}\left[1+3C_{2}e\left(1
        +C_{3}\left(\frac{\pi \rho}{\kappa h}\right)^{\frac{3}{2}} \right)\right]\left(\frac{\kappa h}{\pi N_{\mathbf{e}}}\right)^{\frac{3}{2}-r},
\end{equation}
where $C_{i}>0$ are the constants defined in \eqref{eq: constants main theorem 1} and \eqref{eq: constants main theorem 2 and 3} for $i=1,2,3$.
\end{proof}

\subsection{Accuracy in the high-frequency regime} \label{sec: Accuracy in high-frequency and low-frequency regimes}

In what follows, we derive convergence estimates in the high-frequency regime, distinguishing as before between analytic and low-regularity homogeneous Helmholtz solutions.

The $\kappa$-explicit analysis is complicated by the presence of the resonant frequencies identified in \eqref{eq: resonant freq nodes}.
For any shape-regularity constant $\rho\ge 1$ in \eqref{eq: shape regularity} and $\kappa h \in (0,+\infty)$, we introduce
\begin{equation} \label{eq: D tilde big}
    \widetilde{D}(\rho,\kappa h):=1+\frac{1}{\widetilde{\textrm{\textup{d}}}(\rho,\kappa h)}+\frac{1}{\sqrt{\widetilde{\textrm{\textup{d}}}_0(\rho,\kappa h)}},
\end{equation}
where $\widetilde{\mathrm{d}} = \widetilde{\mathrm{d}}(\rho,\kappa h)$ and $\widetilde{\mathrm{d}}_0 = \widetilde{\mathrm{d}}_0(\rho,\kappa h)$ are defined in \eqref{eq: definition d_m tilde}.
The quantity $\widetilde{D}(\rho,\kappa h)$ measures how close $\kappa$ is to these resonant frequencies, and it blows up as $\kappa$ approaches any of them.

The next lemma, whose proof is given in Appendix~\ref{proof: lem: d tilde lemma}, will be instrumental in the subsequent derivation of accuracy and stability estimates in the high-frequency regime.

\begin{lemma}[High-frequency regime: $\kappa h \to +\infty$] \label{lem: d tilde lemma}
    For any shape-regularity constant $\rho\geq 1$ such that $\rho^2\in \mathbb{Q}$, there exists an unbounded set $\mathcal{K}\subset (0,+\infty)$ such that
    \begin{equation} \label{lem: d tilde lemma equation}
        \sup_{\kappa h \in \mathcal{K}}\frac{\widetilde{D}(\rho,\kappa h)}{(\kappa h)^2}<+\infty.
    \end{equation}
\end{lemma}

\begin{remark}
For any $\rho \ge 1$ such that $\rho^2 \in \mathbb{Q}$, the previous lemma identifies an unbounded set $\mathcal{K}$ of positive values $\kappa h$ for which $\widetilde{D}(\rho,\kappa h)$ grows at most quadratically with $\kappa h$, with a constant $C(\rho)>0$ depending only on $\rho$. The assumption $\rho^2 \in \mathbb{Q}$ is crucial for controlling $C(\rho)$: as shown in the proof in Appendix~\textup{\ref{proof: lem: d tilde lemma}}, if $\rho^2 = p/q$ with $\gcd(p,q)=1$, the constant $C(\rho)$ increases with $q$. By Dirichlet’s theorem \textup{\cite[Cor.\ 1B]{Schmidt1980}}, approximating an irrational $\rho^2$ requires arbitrarily large denominators, so $C(\rho)$ becomes unbounded as $\rho^2$ approaches an irrational value.
\end{remark}

\begin{remark}
A key difficulty in proving the existence of an unbounded set $\mathcal{K}$ for which \eqref{lem: d tilde lemma equation} holds lies in showing that $\widetilde{\mathrm{d}}(\rho,\kappa h)\,(\kappa h)^2$ remains uniformly bounded from below in $\kappa h$ for arbitrarily large values of $\kappa h \gg 1$.
Assume for instance that $h_1 \geq h_2$, so that $h_1=h$ and $h_2=h/\rho$.
In this case, $\widetilde{\mathrm{d}}=0$ if and only if there exist $ n \in \mathbb{N}^*,m \in \mathbb{N}$ such that
\begin{equation*}
\left(\frac{(2n-1)\pi}{2}\right)^2 + (m\pi \rho)^2 = (\kappa h)^2,
\end{equation*}
situation which we excluded due to Assumption \textup{\ref{A1+A2}}.
Thus, $\widetilde{\mathrm{d}}$ can be interpreted as a distance-like quantity between the circle of radius $\kappa h$ and the lattice
\begin{equation} \label{eq: lattice}
\Bigg\{\left(\frac{(2n-1)\pi}{2},\, m\pi \rho\right)\Bigg\}_{n \in \mathbb{N}^*, m \in \mathbb{N}}.
\end{equation}
This is essentially the same kind of question as in the \emph{Gauss circle problem \cite{Gauss1834}}, which studies the remainder $E(r)$ in $\#\{(m,n) \in \mathbb{Z}^2 : m^2 + n^2 \le r^2\} = \pi r^2 + E(r)$, measuring the discrepancy between lattice points and the circle’s area. The core difficulty in both settings lies in understanding how lattice points are distributed near circles, making the analysis of $\mathcal{K}$ nontrivial.
\end{remark}

\paragraph{Analytic Helmholtz solution.}
Let us consider a Helmholtz solution $u$ analytic on $\overline{\Omega}$.
If $\rho^2 \in \mathbb{Q}$ and the numbers of nodal and edge basis functions, $N_{\mathbf{n}}$ and $N_{\mathbf{e}}$, scale linearly with respect to $\kappa h$, then the approximation error decays at \textit{a geometric rate} as $\kappa h \to +\infty$ within an unbounded set $\mathcal{K}$ satisfying \eqref{lem: d tilde lemma equation}.
Equivalently, for any fixed mesh $\mathcal{T}_h$, there exists a sequence $\kappa_j \to +\infty$ along which the best-approximation error converges to zero at a geometric rate.
This is made precise in the following result.

\begin{corollary}[Analytic Helmholtz solution -- High-frequency regime: $\kappa h \to +\infty$] \label{cor: high freq Cinfty}
Assume $\rho^2 \in \mathbb{Q}$ and $r \in \{0,1\}$, let $\mathcal{K}$ be an unbounded set satisfying \eqref{lem: d tilde lemma equation}, and let $u$ be a homogeneous Helmholtz solution analytic on $\overline{\Omega}$. Additionally, let $\tau_{\mathbf{e}},\tau_{\mathbf{n}} \geq 1$ be chosen so that
\begin{equation} \label{eq Ne Nn Cinfty}
N_{\mathbf{e}} = \frac{\tau_{\mathbf{e}}}{\pi} \kappa h\in \mathbb{N}^*, \qquad N_{\mathbf{n}} = \frac{\tau_{\mathbf{n}}}{\pi} \kappa h\in \mathbb{N}^*,\qquad \text{and} \qquad \eta:=\frac{\tau_{\mathbf{e}}}{\rho \tau_{\mathbf{n}}}\exp\left(-\frac{\pi}{2\tau_{\mathbf{n}}}\right)>1,
\end{equation}
which can always be ensured by an appropriate choice of $\tau_{\mathbf{e}}$ and $\tau_{\mathbf{n}}$.
Then, there exists an explicit constant $C=C(\rho,r,\tau_{\mathbf{e}},\tau_{\mathbf{n}})>0$, such that, for any $1 \ll \kappa h\in \mathcal{K}$,
\begin{equation*}
        \inf_{v \in V_{\mathbf{N}}(\mathcal{T}_h)}\|u-v\|_{H_{\kappa}^r(\Omega)}\leq\|u-\boldsymbol{\mathcal{I}}u\|_{H_\kappa^r(\Omega)}\leq C|\mathcal{N}_h|(\kappa h)^5\eta^{-\frac{2}{\pi}\tau_{\mathbf{n}}\kappa h}\|u\|_{H_\kappa^{2N_{\mathbf{n}}+1}(\Omega)}.
\end{equation*}
\end{corollary}
\begin{proof}
    This result is a direct consequence of the estimate \eqref{eq: AA} in the proof of Corollary~\ref{cor: corollary Cinfty}, which applies with the parameter choices in \eqref{eq Ne Nn Cinfty}, as the condition \eqref{eq: nuN} is satisfied: $\nu_{N_{\mathbf{e}}}=N_{\mathbf{e}}\pi/(\kappa h)  = \tau _{\mathbf{e}}> \tau_{\mathbf{n}} \exp(\frac{\pi}{2\tau_{\mathbf{n}}})> \sqrt{2}$.
    Particularly, from the definitions of the constants $C_{i}>0$ in \eqref{eq: constants main theorem 1} and \eqref{eq: constants main theorem 2 and 3} for $i=1,2,3$, one can derive
    \begin{equation} \label{eq: constant bound infty}
        C_{1}\leq 8\sqrt{\frac{\rho \kappa h}{\tanh(1/\rho)}}, \qquad C_{2}\leq \frac{2\pi\rho^{\frac{3}{2}}|\mathcal{N}_h|\kappa h}{\sqrt{2}\widetilde{\mathrm{d}}}, \qquad C_{3}\leq \frac{2\widetilde{\mathrm{d}}}{\sqrt{\pi \rho \tanh(1/\rho)\widetilde{\mathrm{d}}_0}}.
    \end{equation}
    Hence, thanks moreover to the bound on $\chi_{N_{\mathbf{n}}}$ in \eqref{eq: bound chi Cinfty} and Lemma \ref{lem: d tilde lemma}, there exist explicit constants $C(\rho)>0$, only dependent on the shape parameter $\rho\geq1$, such that
    \begin{align*}
        \frac{\|u-\boldsymbol{\mathcal{I}}u\|_{H_{\kappa}^r(\Omega)}}{\|u\|_{H_\kappa^{2N_{\mathbf{n}}+1}(\Omega)}}&\leq C(\rho)C_{1}\left[1+C_{2}\tau^{5/2}_{\mathbf{n}}(\kappa h)^{3/2}e^{\kappa h}\left(1+C_{3}(\rho \tau_{\mathbf{n}})^{\frac{2}{\pi}\tau_{\mathbf{n}}\kappa h}\right)\right]\left(\frac{1}{\tau_{\mathbf{e}}}\right)^{\frac{2}{\pi}\tau_{\mathbf{n}}\kappa h+\frac{1}{2}-r}\\    
        &\leq C(\rho)\tau_{\mathbf{e}}^{r-\frac{1}{2}}\tau_{\mathbf{n}}^{5/2}(\kappa h)^3\left[
        1+\frac{|\mathcal{N}_h|e^{\kappa h}}{\widetilde{\textrm{d}}(\rho,\kappa h)}+\frac{|\mathcal{N}_h|e^{\kappa h}}{\sqrt{\widetilde{\textrm{d}}_0(\rho,\kappa h)}}(\rho \tau_{\mathbf{n}})^{\frac{2}{\pi}\tau_{\mathbf{n}}\kappa h}
        \right]\left(\frac{1}{\tau_{\mathbf{e}}}\right)^{\frac{2}{\pi}\tau_{\mathbf{n}}\kappa h}\\
        &\leq C(\rho)\tau_{\mathbf{e}}^{r-1/2}\tau_{\mathbf{n}}^{5/2}(\kappa h)^3|\mathcal{N}_h|\widetilde{D}(\rho,\kappa h)\left(\frac{\rho \tau_{\mathbf{n}}e^{\frac{\pi}{2\tau_{\mathbf{n}}}}}{\tau_{\mathbf{e}}}\right)^{\frac{2}{\pi}\tau_{\mathbf{n}}\kappa h}\\
        &=C(\rho)\tau_{\mathbf{e}}^{r-1/2}\tau_{\mathbf{n}}^{5/2}|\mathcal{N}_h|(\kappa h)^5 \eta^{-\frac{2}{\pi}\tau_{\mathbf{n}}\kappa h},
    \end{align*}
    where $\widetilde{D}(\rho,\kappa h)$ is defined in \eqref{eq: D tilde big}.
\end{proof}

\paragraph{Helmholtz solution in $H_\kappa^2(\Omega)$.}

Let us now consider a low Sobolev-regularity Helmholtz solution $u \in H^2_\kappa(\Omega)$.
When $\rho^2 \in \mathbb{Q}$ and $1 \ll \kappa h \in \mathcal{K}$, where $\mathcal{K}$ is an unbounded set satisfying \eqref{lem: d tilde lemma equation}, if the number of edge basis functions $N_{\mathbf{e}}$ scales according to a \textit{$r$-dependent polynomial degree} in $\kappa h$, for $r\in\{0,1\}$, the relative error $\|u-\boldsymbol{\mathcal{I}}u\|_{H^r_\kappa(\Omega)}/\|u\|_{H_\kappa^2(\Omega)}$ is bounded uniformly in $\kappa$, with its dependence on 
the mesh size $h$ appearing only via $|\mathcal{N}_h|$.
This is clarified in the next result.

\begin{corollary}[Helmholtz solution in $H_\kappa^2(\Omega)$ -- High-frequency regime: $\kappa h \to +\infty$] \label{cor: rough high freq}
Assume $\rho^2 \in \mathbb{Q}$, $N_{\mathbf{e}},N_{\mathbf{n}}\in \mathbb{N}^*$ and $r \in \{0,1\}$, let $\mathcal{K}$ be an unbounded set satisfying \eqref{lem: d tilde lemma equation}, and let $u\in H^2_\kappa(\Omega)$ be a homogeneous Helmholtz solution. Additionally, assume that condition \eqref{eq: nuN} holds.
Then, there exists an explicit constant $C=C(\rho)>0$ such that, for any $1 \ll \kappa h \in \mathcal{K}$,
\begin{equation} \label{eq: condition H2}
        \inf_{v \in V_{\mathbf{N}}(\mathcal{T}_h)}\|u-v\|_{H_{\kappa}^r(\Omega)}\leq\|u-\boldsymbol{\mathcal{I}}u\|_{H_\kappa^r(\Omega)}\leq C|\mathcal{N}_h|(\kappa h)^{5-r} N_{\mathbf{e}}^{-(\frac{3}{2}-r)}\|u\|_{H_\kappa^{2}(\Omega)}.
    \end{equation}
\end{corollary}
\begin{proof}
    The result follows directly from the estimate \eqref{eq: CC} in the proof of Corollary \ref{cor: H2}. Thanks to the bounds in \eqref{eq: constant bound infty} and Lemma \ref{lem: d tilde lemma}, there exist explicit constants $C(\rho)>0$ such that
    \begin{align*}
        \frac{\|u-\boldsymbol{\mathcal{I}}u\|_{H_{\kappa}^r(\Omega)}}{\|u\|_{H_\kappa^{2}(\Omega)}}&
        \leq C(\rho)\sqrt{\kappa h}\left[1+\frac{|\mathcal{N}_h|\kappa h}{\widetilde{\mathrm{d}}(\rho, \kappa h)}+\frac{|\mathcal{N}_h|\kappa h}{\sqrt{\widetilde{\mathrm{d}}_0(\rho, \kappa h)}}\right]\left(\frac{\kappa h}{\pi N_{\mathbf{e}}}\right)^{\frac{3}{2}-r}\\
        & \leq C(\rho)|\mathcal{N}_h|(\kappa h)^{\frac{3}{2}}\widetilde{D}(\rho,\kappa h)\left(\frac{\kappa h}{\pi N_{\mathbf{e}}}\right)^{\frac{3}{2}-r}\leq C(\rho)|\mathcal{N}_h|(\kappa h)^{5-r} N_{\mathbf{e}}^{-(\frac{3}{2}-r)},
    \end{align*}
    with $\widetilde{D}(\rho,\kappa h)$ given by \eqref{eq: D tilde big}.
\end{proof}

\begin{remark}
    Corollary \emph{\ref{cor: rough high freq}} shows that, for Helmholtz solutions of regularity $H^2_\kappa(\Omega)$, achieving a best-approximation error independent of $\kappa h$ requires the edge parameter $N_{\mathbf{e}}\in \mathbb{N}^*$ to scale according to a sufficiently high polynomial power of $\kappa h$; see \eqref{eq: condition H2}. Nevertheless, the numerical results in Figure \emph{\ref{fig: corner sing 2}} of Section \emph{\ref{sec: Numerical experiments}} suggest that a merely linear scaling in $\kappa h$ already suffices, even for solutions of lower regularity than $H^2_\kappa(\Omega)$. 
\end{remark}

\subsection{Stability with respect to edge and node parameters} \label{sec: Stability with respect to edge and node parameters}

Alongside the approximation results in Section \ref{sec: Accuracy with respect to edge and node parameters}, we establish corresponding stability bounds, stated in full generality and then specialized to the two main cases considered previously: a Helmholtz solution analytic on $\overline{\Omega}$ and a low Sobolev-regularity solution in $H^2_\kappa(\Omega)$.

Let $\boldsymbol{\alpha}\!=\!(\alpha_{\mathbf{s},n})_{\mathbf{s},n}$ and $\boldsymbol{\beta}\!=\!(\beta_{\mathbf{p},m})_{\mathbf{p},m}$ denote the edge-based and node-based coefficient vectors associated with the interpolant $\boldsymbol{\mathcal{I}}u$, as introduced in \eqref{eq: coefficient final alpha} and \eqref{eq: local collocation operator condition}. The following result shows that the norms of $\boldsymbol{\alpha}$ and $\boldsymbol{\beta}$ grow at most quadratically with $\mu = \min(2N_{\mathbf{n}},M)$, while remaining independent of the edge degree $N_{\mathbf{e}}$. This ensures that the expansion coefficients remain moderate and can be computed, thus contributing to the numerical stability of the scheme; see \cite{Adcock2019,Adcock2020}.

\begin{theorem}[Edge-and-node based stability result] \label{th: general coefficient}
    Let $M\in \mathbb{N}^*$ and $u \in H_\kappa^{M+1}(\Omega)$ be a homogeneous Helmholtz solution.
    Then, for any $N_\mathbf{e},N_\mathbf{n} \in \mathbb{N}^*$, the following estimates hold: 
    \begin{align}
        \|\boldsymbol{\alpha}\|_{\ell^2}&\leq \frac{8\rho \max(1,\kappa h)}{h}\left[\frac{1}{2}+C_{2}\chi_{\lceil \mu/2 \rceil}  \left(\lceil \tfrac{\mu}{2} \rceil+C_{3}\lceil \tfrac{\mu}{2} \rceil^2\sqrt{\frac{\pi \rho }{\kappa h}}\right) \right]\|u\|_{H_\kappa^{\mu+1}(\Omega)}, \label{eq: coefficient alpha beta1}\\
        \|\boldsymbol{\beta}\|_{\ell^2}&\leq \frac{3\rho \max(1,\kappa h)^2}{h}\chi_{\lceil \mu/2 \rceil}\sqrt{|\mathcal{N}_h|\lceil \tfrac{\mu}{2} \rceil} \|u\|_{H_\kappa^{\mu+1}(\Omega)}, \label{eq: coefficient alpha beta2}
    \end{align}
    where $\mu=\min(2N_\mathbf{n},M)$, and
    $\chi_{\lceil \mu/2 \rceil}$ and the constants $C_{2},C_{3}>0$ are defined in \eqref{eq:beta_pn0_p} and \eqref{eq: constants main theorem 2 and 3}.
\end{theorem}
\begin{proof}
    We begin by considering the bound \eqref{eq: coefficient alpha beta1}.
    For any $\mathbf{s}\in \Sigma_h$, the traces $\phi_{\mathbf{s},n}|_{\mathbf{s}}$ form a Hilbert basis of $L^2(\mathbf{s})$, as they coincides, up to a rigid transformation, with the functions $\widehat{\varphi}_{n}|_{\hat{\mathbf{s}}}$ -- taking $(\widehat{h}_1,\widehat{h}_2)=(h_1,h_2)$ or $(\widehat{h}_1,\widehat{h}_2)=(h_2,h_1)$ -- which form a Hilbert basis of $L^2(\hat{\mathbf{s}})$ by Lemma \ref{lem: orthogonality L2}.
    Hence, as $\| \phi_{\mathbf{s},n}\|^2_{L^2(\mathbf{s})}\geq\frac{h}{2\rho}$ for any $\mathbf{s}\in \Sigma_h$ and $n \in \mathbb{N}^*$, we have
    \begin{equation*}
        \|\boldsymbol{\alpha}\|^2_{\ell^2}=\sum_{\mathbf{s}\in \Sigma_h}\sum_{n=1}^{N_{\mathbf{e}}}\frac{|( (u-\boldsymbol{\mathcal{V}}u),\phi_{\mathbf{s},n})_{L^2(\mathbf{s})}|^2}{\| \phi_{\mathbf{s},n}\|^4_{L^2(\mathbf{s})}}\leq \frac{2\rho}{h}\sum_{\mathbf{s}\in \Sigma_h}\| (u-\boldsymbol{\mathcal{V}}u)\|^2_{L^2(\mathbf{s})},
    \end{equation*}
    and, thanks to Lemma \ref{lem: trace inequality}, it follows
    \begin{equation*}
        \left(\frac{h}{4\rho\max(1,\kappa h)}\right)^{2}\|\boldsymbol{\alpha}\|^2_{\ell^2}\leq\sum_{K \in \mathcal{T}_h}\|u-\boldsymbol{\mathcal{V}}u\|^2_{H^1_\kappa(K)}\leq \left(\|u\|^2_{H^1_\kappa(\Omega)}+\|\boldsymbol{\mathcal{V}}u\|^2_{H^1_\kappa(\mathcal{T}_h)}\right).
    \end{equation*}
    Moreover, recalling that $\nu_N=\frac{N\pi}{\kappa h}$, where $N=\min(N_{\mathbf{n}},\lceil M/2\rceil)$, Lemma \ref{lem: V continuity} implies that
    \begin{align*}
        \left(\frac{h}{4\rho\max(1,\kappa h)}\right)^{2}\|\boldsymbol{\alpha}\|^2_{\ell^2}&
        \leq\left(1+\left[2C_{2}\chi_{N} \left(N+C_{3}\sqrt{N^3 \rho \nu_N}\right)  \right]^2\right)\|u\|^2_{H_\kappa^{2N}(\Omega)}\\
        &\leq \left(1+2C_{2}\chi_{\lceil \mu/2 \rceil} \left(\lceil \tfrac{\mu}{2} \rceil+C_{3}\lceil \tfrac{\mu}{2} \rceil^2\sqrt{\frac{\pi \rho }{\kappa h}}\right) \right)^2\|u\|^2_{H_\kappa^{\mu+1}(\Omega)},
    \end{align*}
    where in the last step we exploit the fact that $N = \lceil \mu/2\rceil$ and $2N \leq \mu+1$.

    The bound \eqref{eq: coefficient alpha beta2} follows directly from the coefficient estimate in \eqref{eq:beta_pn0_p}, in fact
    \begin{equation*}
        \|\boldsymbol{\beta}\|_{\ell^2}\leq \sqrt{|\mathcal{N}_h|N}\|\boldsymbol{\beta}\|_{\ell^{\infty}}\leq \frac{3\rho \max(1,\kappa h)^2}{h}\chi_{\lceil \mu/2 \rceil}\sqrt{|\mathcal{N}_h|\lceil \tfrac{\mu}{2} \rceil} \|u\|_{H_\kappa^{\mu+1}(\Omega)}. \qedhere
    \end{equation*}
\end{proof}

\paragraph{Analytic Helmholtz solution.}

If $u$ is analytic on $\overline{\Omega}$, Theorem \ref{th: general coefficient} implies at most \textit{quadratic growth} of edge coefficients $\boldsymbol{\alpha}$ and \textit{square-root growth} of node coefficients $\boldsymbol{\beta}$ with respect to the node parameter $N_{\mathbf{n}}$, while the bounds remain \textit{independent of the edge parameter $N_{\mathbf{e}}$}.

\begin{corollary}[Analytic Helmholtz solution] \label{cor: coefficient Cinfty}
    Let $u$ be a homogeneous Helmholtz solution analytic on $\overline{\Omega}$.
    There exists an explicit constant $C=C(\rho,\kappa h)>0$ such that, for any $N_\mathbf{e},N_\mathbf{n}\in \mathbb{N}^*\!$,
    \begin{equation*}
        \|\boldsymbol{\alpha}\|_{\ell^2}\leq \frac{C}{h} N_{\mathbf{n}}^2\|u\|_{H_\kappa^{2N_{\mathbf{n}}+1}(\Omega)}, \qquad 
        \|\boldsymbol{\beta}\|_{\ell^2}\leq \frac{C}{h} \sqrt{N_{\mathbf{n}}} \|u\|_{H_\kappa^{2N_{\mathbf{n}}+1}(\Omega)}.
    \end{equation*}
\end{corollary}
\begin{proof}
    The result follows from Theorem \ref{th: general coefficient} with $\mu=2N_{\mathbf{n}}$ and the bound in \eqref{eq: bound chi Cinfty}, namely
    \begin{align}
        \frac{\|\boldsymbol{\alpha}\|_{\ell^2}}{\|u\|_{H_\kappa^{2N_{\mathbf{n}}+1}(\Omega)}}&\leq \frac{8\rho \max(1,\kappa h)}{h}\left[\frac{1}{2}+C_{2}\chi_{N_{\mathbf{n}}} N_{\mathbf{n}} \left(1+C_{3}N_{\mathbf{n}}\sqrt{\frac{\pi \rho }{\kappa h}}\right) \right] \nonumber\\
        &\leq \frac{8\rho \max(1,\kappa h)}{h}\left[\frac{1}{2}+\frac{C_{2}\rho e^{\kappa h}}{\kappa h} \left(1+C_{3}\sqrt{\frac{\pi \rho }{\kappa h}}\right) \right]N_{\mathbf{n}}^2, \label{A}
    \end{align}
    and
    \begin{equation} \label{B}
        \frac{\|\boldsymbol{\beta}\|_{\ell^2}}{\|u\|_{H_\kappa^{2N_{\mathbf{n}}+1}(\Omega)}}\leq \frac{3\rho \max(1,\kappa h)^2}{ h}\chi_{N_{\mathbf{n}}}\sqrt{|\mathcal{N}_h|N_{\mathbf{n}}}\leq \frac{3\rho^2\sqrt{|\mathcal{N}_h|} \max(1,\kappa h)^2e^{\kappa h}}{ \kappa h^2}\sqrt{N_{\mathbf{n}}},
    \end{equation}
    where $\chi_{N_{\mathbf{n}}}$ and the constants $C_{2},C_{3}>0$ are defined in \eqref{eq:beta_pn0_p} and \eqref{eq: constants main theorem 2 and 3}.
\end{proof}

Combined with Corollary \ref{cor: corollary Cinfty}, the previous result shows that the family of Trefftz spaces $\{V_{\mathbf{N}}(\mathcal{T}_h)\}_{\mathbf{N}\in (\mathbb{N^*})^2}$ admits the following properties:
\begin{itemize}
\item \textit{Algebraic convergence in $N_\mathbf{e}$ with uniformly bounded coefficient norms}:
For any fixed nodal degree $N_\mathbf{n}$, the coefficient norms are bounded uniformly with respect to $N_\mathbf{e}$.
\item \textit{Geometric convergence in $N_\mathbf{n}$ with at most quadratic growth of the coefficient norm}:  
If $N_\mathbf{e} = \tau_{\mathbf{e}} N_\mathbf{n}$ for $\tau_{\mathbf{e}} > \rho$, $\|\boldsymbol{\alpha}\|_{\ell^2}$ and $\|\boldsymbol{\beta}\|_{\ell^2}$ grow at most as $N_{\mathbf{n}}^2$ and $\sqrt{N_{\mathbf{n}}}$, respectively.
\end{itemize}

\paragraph{Helmholtz solution in $H_\kappa^2(\Omega)$.}

If $u \in H_\kappa^2(\Omega)$, Theorem \ref{th: general coefficient} ensures that the growth of the expansion coefficients of $\boldsymbol{\mathcal{I}}u$ is \textit{independent of both $N_{\mathbf{e}}$ and $N_{\mathbf{n}}$}.

\begin{corollary}[Helmholtz solution in $H_\kappa^2(\Omega)$] \label{cor: coefficient H2}
    Let $u \in H^2_\kappa(\Omega)$ be a homogeneous Helmholtz solution.
    There exists an explicit constant $C=C(\rho,\kappa h)>0$ such that, for any $N_\mathbf{e},N_\mathbf{n}\in \mathbb{N}^*$, 
    \begin{equation*}
        \|\boldsymbol{\alpha}\|_{\ell^2}\leq \frac{C}{h} \|u\|_{H_\kappa^{2}(\Omega)}, \qquad 
        \|\boldsymbol{\beta}\|_{\ell^2}\leq \frac{C}{h} \|u\|_{H_\kappa^{2}(\Omega)}.
    \end{equation*}
\end{corollary}
\begin{proof}
    The result follows from Theorem \ref{th: general coefficient} with $\mu=M=1$ and the simple bound $\chi_{1} = \min(\rho e^{\kappa h}/(\kappa h), e) \leq  e$, with $\chi_{N}$ as introduced in \eqref{eq:beta_pn0_p}, namely
    \begin{align}
        \|\boldsymbol{\alpha}\|_{\ell^2}&\leq \frac{8\rho \max(1,\kappa h)}{h}\left[\frac{1}{2}+C_{2}e \left(1+C_{3}\sqrt{\frac{\pi \rho }{\kappa h}}\right) \right]\|u\|_{H_\kappa^{2}(\Omega)}, \label{eq: bounds rough high freq 1}\\
        \|\boldsymbol{\beta}\|_{\ell^2}&\leq \frac{3\rho e \max(1,\kappa h)^2}{h}\sqrt{|\mathcal{N}_h|} \|u\|_{H_\kappa^{2}(\Omega)}, \label{eq: bounds rough high freq 2}
    \end{align}
    where the constants $C_{2},C_{3}>0$ are defined in \eqref{eq: constants main theorem 2 and 3}.
\end{proof}

Together with Corollary \ref{cor: H2}, this result shows that for every homogeneous Helmholtz solution $u \in H^2_\kappa(\Omega)$ and $N_{\mathbf{n}} \in \mathbb{N}^*$, we achieve \textit{algebraic convergence in $N_{\mathbf{e}}$ with coefficient norms uniformly bounded with respect to both $N_{\mathbf{e}}$ and $N_{\mathbf{n}}$}.

\subsection{Stability in the high-frequency regime}

We now present some stability estimates in the high-frequency regime, corresponding to the results of Section \ref{sec: Accuracy in high-frequency and low-frequency regimes}. Following the same structure as before, we distinguish between analytic and low-regularity homogeneous Helmholtz solutions.

\paragraph{Analytic Helmholtz solution.}

Let $u$ be an analytic Helmholtz solution on $\overline{\Omega}$, and consider the setting of Corollary \ref{cor: high freq Cinfty}, namely $\rho^2 \in \mathbb{Q}$ and $1 \ll \kappa h \in \mathcal{K}$, where $\mathcal{K}$ is an unbounded set satisfying \eqref{lem: d tilde lemma equation}, with the edge and node parameters $N_{\mathbf{e}}, N_{\mathbf{n}} \in \mathbb{N}^*$ chosen to scale linearly with $\kappa h$. We now show that the coefficient norm grows \textit{at most exponentially with $\kappa h$}.

\begin{corollary}[Analytic Helmholtz solution -- High-frequency regime: $\kappa h \to +\infty$] \label{cor: high freq smooth coeff}
Assume $\rho^2 \in \mathbb{Q}$, let $\mathcal{K}$ be an unbounded set satisfying \eqref{lem: d tilde lemma equation}, and let $u$ be a homogeneous Helmholtz solution analytic on $\overline{\Omega}$. Additionally, let $\tau_{\mathbf{e}}, \tau_{\mathbf{n}} \geq 1$ be chosen so that $N_{\mathbf{e}},N_{\mathbf{n}}\in \mathbb{N}^*$ are defined as in \eqref{eq Ne Nn Cinfty}.
Then, there exists an explicit constant $C=C(\rho,\tau_{\mathbf{n}})>0$, such that, for any $1 \ll \kappa h \in \mathcal{K}$,
\begin{equation*}
        \|\boldsymbol{\alpha}\|_{\ell^2}\leq \frac{C}{h}|\mathcal{N}_h|(\kappa h)^5e^{\kappa h}\|u\|_{H_\kappa^{2N_{\mathbf{n}}+1}(\Omega)} , \qquad 
        \|\boldsymbol{\beta}\|_{\ell^2}\leq \frac{C}{h}\sqrt{|\mathcal{N}_h|(\kappa h)^3}e^{\kappa h} \|u\|_{H_\kappa^{2N_{\mathbf{n}}+1}(\Omega)}.
    \end{equation*}
\end{corollary}
\begin{proof}
    The result follows from Corollary \ref{cor: coefficient Cinfty}, with the parameter choices in \eqref{eq Ne Nn Cinfty}.
    In fact, from the estimates \eqref{A} and \eqref{eq: constant bound infty}, there exist explicit constants $C(\rho)>0$ such that
    \begin{align*}
        \frac{\|\boldsymbol{\alpha}\|_{\ell^2}}{\|u\|_{H_\kappa^{2N_{\mathbf{n}}+1}(\Omega)}}&\leq C(\rho) \kappa \left[
        1+\frac{|\mathcal{N}_h|e^{\kappa h}}{\widetilde{\textrm{d}}(\rho,\kappa h)}+\frac{|\mathcal{N}_h|e^{\kappa h}}{\sqrt{\widetilde{\textrm{d}}_0(\rho,\kappa h)}}\right](\tau_{\mathbf{n}}\kappa h)^2\\
        & \leq \frac{C(\rho)\tau_{\mathbf{n}}^2}{h} |\mathcal{N}_h| \widetilde{D}(\rho,\kappa h)(\kappa h)^3e^{\kappa h}
        \leq \frac{C(\rho)\tau_{\mathbf{n}}^2}{h}|\mathcal{N}_h|(\kappa h)^5e^{\kappa h},
    \end{align*}
    where we rely on Lemma \ref{lem: d tilde lemma} to bound $\widetilde{D}(\rho,\kappa h)$ introduced in \eqref{eq: D tilde big}.
    Similarly, from \eqref{B} it follows:
    \begin{equation*}
        \|\boldsymbol{\beta}\|_{\ell^2}\leq \frac{C(\rho)\tau_{\mathbf{n}}^{1/2}}{h} \sqrt{|\mathcal{N}_h| (\kappa h)^3}e^{\kappa h}\|u\|_{H_\kappa^{2N_{\mathbf{n}}+1}(\Omega)}. \qedhere
    \end{equation*}
\end{proof}

\begin{remark}
    As the previous corollary shows, the coefficient growth can be at most exponential in $\kappa h$, potentially affecting the numerical stability of the method. Nevertheless, in practice such behavior is not observed in our numerical experiments, as illustrated in Figure \emph{\ref{fig: plane wave 2}} of Section \emph{\ref{sec: Numerical experiments}}.
\end{remark}

\paragraph{Helmholtz solution in $H_\kappa^2(\Omega)$.}

Let $u \in H^2_\kappa(\Omega)$ be a homogeneous Helmholtz solution, and consider the setting of Corollary \ref{cor: rough high freq}, where $\rho^2 \in \mathbb{Q}$ and $1 \ll \kappa h \in \mathcal{K}$, with $\mathcal{K}$ an unbounded set satisfying \eqref{lem: d tilde lemma equation}. In this setting, the coefficient norms grow \textit{no faster than $\kappa^4 h^3$}, as shown next.

\begin{corollary}[Helmholtz solution in $H_\kappa^2(\Omega)$ -- High-frequency regime: $\kappa h \to +\infty$]
Assume $\rho^2 \in \mathbb{Q}$, $N_\mathbf{e},N_\mathbf{n}\in \mathbb{N}^*$, let $\mathcal{K}$ be an unbounded set satisfying \eqref{lem: d tilde lemma equation}, and $u \in H^2_\kappa(\Omega)$ be a homogeneous Helmholtz solution.
Then, there exists an explicit $C=C(\rho)>0$ such that, for any $1 \ll \kappa h \in \mathcal{K}$,
\begin{equation} \label{eq: bounds}
        \|\boldsymbol{\alpha}\|_{\ell^2}\leq \frac{C}{h}|\mathcal{N}_h|(\kappa h)^4\|u\|_{H_\kappa^{2}(\Omega)} , \qquad 
        \|\boldsymbol{\beta}\|_{\ell^2}\leq \frac{C}{h} \sqrt{|\mathcal{N}_h|} (\kappa h)^2  \|u\|_{H_\kappa^{2}(\Omega)}.
\end{equation}
\end{corollary}
\begin{proof}
    The statement is a consequence of Corollary \ref{cor: coefficient H2}.
    Thanks to the estimate \eqref{eq: bounds rough high freq 1}, together with \eqref{eq: constant bound infty} and Lemma \ref{lem: d tilde lemma}, there exist explicit constants $C(\rho)>0$ such that
    \begin{align*}
        \|\boldsymbol{\alpha}\|_{\ell^2}&\leq  C(\rho) \kappa \left[
        1+\frac{|\mathcal{N}_h|\kappa h}{\widetilde{\textrm{d}}(\rho,\kappa h)}+\frac{|\mathcal{N}_h|\sqrt{\kappa h}}{\sqrt{\widetilde{\textrm{d}}_0(\rho,\kappa h)}}\right]\|u\|_{H_\kappa^{2}(\Omega)}\\
        & \leq \frac{C(\rho)}{h} |\mathcal{N}_h| \widetilde{D}(\rho,\kappa h)(\kappa h)^2 \|u\|_{H_\kappa^{2}(\Omega)}
        \leq \frac{C(\rho)}{h}|\mathcal{N}_h|(\kappa h)^4\|u\|_{H_\kappa^{2}(\Omega)}.
    \end{align*}
     Similarly, from the estimate \eqref{eq: bounds rough high freq 2}, one obtains the upper bound on $\|\boldsymbol{\beta}\|_{\ell^2}$ stated in \eqref{eq: bounds}.
\end{proof}

\section{Discrete scheme and numerical results} \label{sec: Discrete scheme and numerical results}

This section introduces the discrete scheme used to compute numerical solutions within the conforming Trefftz framework developed above.
We describe the construction and assembly of the corresponding linear system, noting that for polygonal domains, the integrals involved in the matrix assembly can be computed exactly in closed form due to the local explicit decomposition of the basis functions in $V_{\mathbf{N}}(\mathcal{T}_h)$ into EPWs.
The section closes with a set of numerical experiments, showing the stability and effectiveness of the proposed approach in practice.

\subsection{Least-square Petrov--Galerkin with regularization} \label{sec: Least-square Petrov--Galerkin with regularization}

\begin{figure}
\centering
\includegraphics[trim=140 250 140 240,clip,width=.27\textwidth]{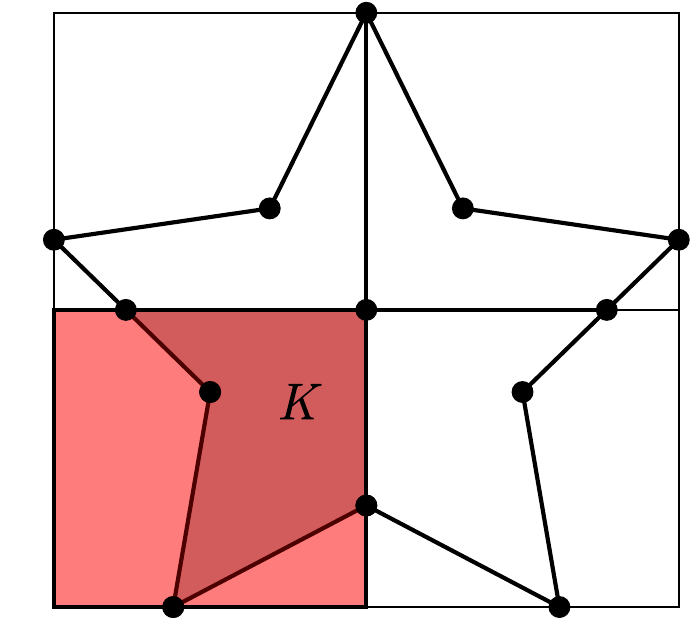}
\caption{A star-shaped domain $\Omega$ covered by a rectangular mesh $\mathcal{T}_h$. The red square highlights a mesh element $K \in \mathcal{T}_h$, while the darker region represents the intersection $K \cap \Omega$.}
\label{fig: cropped star}
\end{figure}

To tackle the redundancy of the edge-and-node approximation space $V_{\mathbf{N}}(\mathcal{T}_h)$ introduced in Definition \ref{def: edge-and-node Trefftz space}, we adopt a \textit{least-squares Petrov–Galerkin} formulation combined with a \textit{regularized Singular Value Decomposition} (SVD), in line with \cite{Parolin2023,Galante2024}.
We use a regularized SVD for its stability and robustness, despite the higher computational cost, while scalable alternatives, such as element-wise SVD or QR regularization \cite{Barucq2021,Barucq2024}, remain a topic for future study.

As a test case, we consider the homogeneous Helmholtz equation \eqref{eq:helmholtz_equation} with impedance boundary conditions $(\partial_{\mathbf{n}}-\imath \kappa)u=g$ on $\partial \Omega$, where $g \in L^2(\partial \Omega)$ is given. Hence, we introduce the following sesquilinear form and the antilinear functional
\begin{equation*}
    \mathcal{A}(v,w):=\int_{\Omega}\left(\nabla v \cdot \nabla \overline{w}-\kappa^2 v\overline{w}\right)-\imath\kappa\int_{\partial \Omega} v  \overline{w}, \qquad \mathcal{F}(w):=\int_{\partial \Omega}g \overline{w}, \qquad v,w \in H_\kappa^1(\Omega).
\end{equation*}
The variational formulation and its conforming least-squares Petrov--Galerkin approximation are
\begin{gather}
    \text{find} \quad u \in H_\kappa^1(\Omega) \quad \text{such that} \quad \mathcal{A}(u,v) = \mathcal{F}(v), \quad \forall v \in H_\kappa^1(\Omega), \label{eq: robin problem}
    \\[1ex]
    \text{find} \quad  u_{\mathbf{N}} \in V_{\mathbf{N}}(\mathcal{T}_h) \quad \text{such that} \quad \mathcal{A}(u_{\mathbf{N}}, v_{\mathbf{M}}) = \mathcal{F}(v_{\mathbf{M}}), \quad \forall v_{\mathbf{M}} \in V_{\mathbf{M}}(\mathcal{T}_h), \label{eq: galerkin problem}
\end{gather}
where $V_{\mathbf{N}}(\mathcal{T}_h)$ with $\mathbf{N}=(N_{\mathbf{e}},N_{\mathbf{n}})$ is taken as the trial space, and $V_{\mathbf{M}}(\mathcal{T}_h)$ with
$\mathbf{M}=2\mathbf{N}$
as the test space. This choice corresponds to oversampling the trial space by a factor of 2.

If $\Omega \subset \mathbb{R}^2$ is a general Lipschitz domain that cannot be exactly discretized by a mesh of rectangular elements, we can still define the Trefftz space $V_{\mathbf{N}}(\mathcal{T}_h)$ generated by basis functions associated with a rectangular mesh $\mathcal{T}_h$ covering $\Omega$; see Figure \ref{fig: cropped star}. In this case, the basis functions are restricted to the intersections $\mathcal{T}_h \cap \Omega$.

The Galerkin problem \eqref{eq: galerkin problem} is implemented by assembling the matrix $\mathbf{A}$ and the right-hand side vector $\mathbf{F}$, defined by evaluating, respectively, 
$\mathcal{A}$ and $\mathcal{F}$ to both the edge and nodal basis functions in $V_{\mathbf{N}}(\mathcal{T}_h)$ and $V_{\mathbf{M}}(\mathcal{T}_h)$.
The approximate solution is then obtained by approximately solving the rectangular linear system
\begin{equation} \label{eq: linear system}
\mathbf{A}\boldsymbol{\xi} = \mathbf{F},
\end{equation}
where $\boldsymbol{\xi}$ collects the coefficients corresponding to the two families of basis functions.
Since the basis functions have compact support and are localized in the physical domain $\Omega$, the matrix $\mathbf{A}$ is generally sparse.

In finite-precision arithmetic, the matrix $\mathbf{A}$ is often ill-conditioned due to redundancy in the basis functions, that leads to numerical non-uniqueness of the coefficient vector $\boldsymbol{\xi}$. Nevertheless, the associated different expansions may still approximate the exact solution with similar accuracy, and if some have small coefficient norms, an accurate approximation can be computed numerically.
We achieve this using a regularized SVD in combination with the oversampling in \eqref{eq: linear system}, following the ideas in \cite{Adcock2019,Adcock2020}.
The regularized solution procedure follows the steps below:
\begin{itemize}
\item Firstly, the reduced SVD of the matrix $\mathbf{A}$ is performed, namely $\mathbf{A}=\mathbf{U}\boldsymbol{\Sigma} \mathbf{V}^*$, where $\boldsymbol{\Sigma} = \operatorname{diag}(\sigma_1,\sigma_2,\dots)$ with $\sigma_1 \geq \sigma_2 \geq \dots$, and $\mathbf{U},\mathbf{V}$ are semi-unitary matrices.
\item Then, the regularization involves discarding the relatively small singular values by setting them to zero.
A threshold parameter $\epsilon \in (0,1]$ is selected, and $\boldsymbol{\Sigma}$ is replaced by  
$\boldsymbol{\Sigma}_{\epsilon}$ by zeroing all $\sigma_n$ such that $\sigma_n < \epsilon \sigma_{1}$.
This results in an approximate factorization of $\mathbf{A}$, that is $\mathbf{A}_{\epsilon}:=\mathbf{U}\boldsymbol{\Sigma}_{\epsilon}\mathbf{V}^*$.
\item Lastly, an approximate solution $u_{\mathbf{N}, \epsilon} \in V_{\mathbf{N}}(\mathcal{T}_h)$ for the linear system in \eqref{eq: linear system} is given by
\begin{equation*}
    u_{\mathbf{N}, \epsilon}:=\sum_{\mathbf{s}\in \Sigma_h}\sum_{n=1}^{N_{\mathbf{e}}}\alpha^{(\epsilon)}_{\mathbf{s},n}\phi_{\mathbf{s},n}+\sum_{\mathbf{p}\in \mathcal{N}_h}\sum_{m=1}^{N_{\mathbf{n}}}\beta^{(\epsilon)}_{\mathbf{p},m}\psi_{\mathbf{p},m},
\end{equation*}
where
\begin{equation*}
\boldsymbol{\xi}^{(\epsilon)}:=\left(\left(\alpha^{(\epsilon)}_{\mathbf{s},n}\right)_{\substack{\mathbf{s}\in \Sigma_h \\ 1\leq n \leq N_{\mathbf{e}}}},\left(\beta^{(\epsilon)}_{\mathbf{p},m}\right)_{\substack{\mathbf{p}\in \mathcal{N}_h \\ 1\leq m \leq N_{\mathbf{n}}}}\right):=\mathbf{A}^{\dagger}_{\epsilon}\mathbf{F}=\mathbf{V}\boldsymbol{\Sigma}_{\epsilon}^{\dagger}\mathbf{U}^*\mathbf{F}.
\end{equation*}
Here $\boldsymbol{\Sigma}_{\epsilon}^{\dagger}$ denotes the pseudo-inverse of the matrix $\boldsymbol{\Sigma}_{\epsilon}$, i.e.\ the diagonal matrix defined by $(\boldsymbol{\Sigma}_{\epsilon}^{\dagger})_{nn}=\boldsymbol{\Sigma}_{nn}^{-1}$ if $\boldsymbol{\Sigma}_{nn} \geq \epsilon \sigma_{1}$ and $(\boldsymbol{\Sigma}_{\epsilon}^{\dagger})_{nn}=0$ otherwise.
\end{itemize}

\subsection{Assembly of the linear system}

All edge and node basis functions, introduced in Definitions \ref{def: edge basis functions} and \ref{def: ndoe basis functions} respectively, can be expanded as a superposition of four plane waves $e^{\imath\kappa\mathbf{d}\cdot \mathbf{x}}$, $\mathbf{d}\in \mathbb{C}^2$, in each cell $K \in \mathcal{T}_h$, by simply using the trigonometric identity $2\imath\sin(z)=e^{\imath z}-e^{-\imath z}$, for $z \in \mathbb{C}$. For instance, if we take a basis function $\phi_{\mathbf{s},n}$ in \eqref{eq: hor edge functions}, associated to the horizontal edge $\mathbf{s}=[x_{\mathbf{s}},x_{\mathbf{s}}+h_1]\times\{y_{\mathbf{s}}\}$, where $(x_{\mathbf{s}},y_{\mathbf{s}})\in \mathcal{N}_h$, the complex-valued direction vectors involved in its expansion are
\begin{equation*}
    \mathbf{d}_{\boldsymbol{\sigma},n}:=\begin{pmatrix}
\sigma_1 \nu_n \\
\sigma_2 \sqrt{1-\nu_n^2}
\end{pmatrix}\in \mathbb{C}^2, \qquad \nu_n=\frac{n\pi}{\kappa h_1}, \qquad \boldsymbol{\sigma}=(\sigma_1,\sigma_2)\in \{\pm1\}^2,\qquad n \in \mathbb{N}^*,
\end{equation*}
and we get
\begin{equation*}
    \phi_{\mathbf{s},n}(\mathbf{x})=\sum_{\boldsymbol{\sigma}\in \{\pm1\}^2}c_{\boldsymbol{\sigma},n}e^{\imath \kappa \mathbf{d}_{\boldsymbol{\sigma},n}\cdot \mathbf{x}}, \qquad \text{where} \qquad c_{\boldsymbol{\sigma},n}:=\pm\frac{\sigma_1\sigma_2 e^{-\imath\kappa\mathbf{d}_{\boldsymbol{\sigma},n}\cdot(x_{\mathbf{s}},y_{\mathbf{s}}\pm h_2)}}{4\sin(h_2\kappa \sqrt{1-\nu_n^2})}, \qquad \mathbf{x} \in K_{\mathbf{s}}^{\pm},
\end{equation*}
and $K_{\mathbf{s}}^{\pm}$ are the cell adjacent to $\mathbf{s}$ defined in \eqref{eq: adj cell hor edge functions}.
A similar expansion holds both for functions $\phi_{\mathbf{s},n}$ associated to vertical edges and node basis function $\psi_{\mathbf{p},n}$.
Numerically, it is crucial to handle the coefficients $c_{\boldsymbol{\sigma},n} \in \mathbb{C}$ with care to avoid introducing potential instabilities.
As for the function $\phi_{\mathbf{s},n}$, the plane waves $e^{\imath \kappa\mathbf{d}_{\boldsymbol{\sigma},n} \cdot \mathbf{x}}$ are propagative when $\nu_n<1$, and evanescent when $\nu_n>1$.

Thanks to this property, all integrals involved in the assembly of the matrix in \eqref{eq: linear system} -- and possibly of the right-hand side -- can be computed exactly in closed form whenever $\Omega$ is a polygonal domain. This follows directly from \cite[Eq.\ (15)]{Hiptmair2016}, to which we refer for details.

\subsection{Numerical experiments} \label{sec: Numerical experiments}

We examine the approximation properties of the proposed method by solving a set of Helmholtz problems using the least-squares Petrov–Galerkin formulation \eqref{eq: galerkin problem} combined with the regularization strategy detailed in Section \ref{sec: Least-square Petrov--Galerkin with regularization}, where we fix the truncation parameter to $\epsilon = 10^{-14}$.

\paragraph{Propagative plane wave.}

\begin{figure}
\centering
\includegraphics[trim=0 0 0 0,clip,height=.298\textwidth]{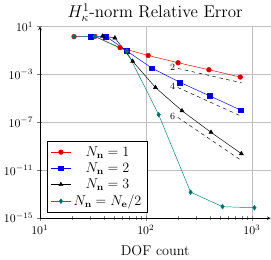}
\hspace{.01\textwidth}
\includegraphics[trim=130 179 40 195,clip,width=.31\textwidth]{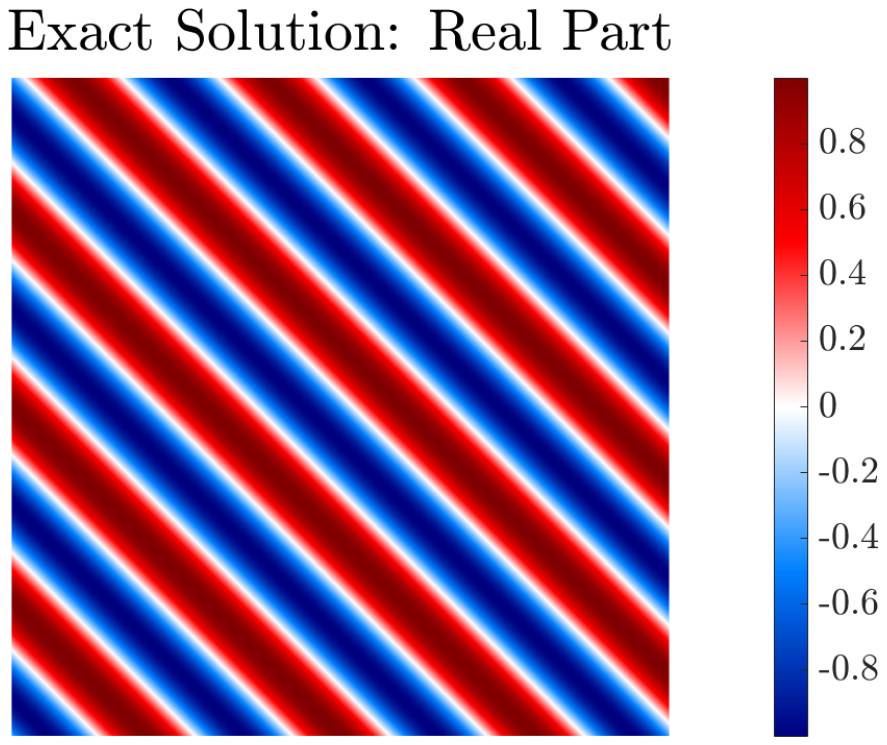}
\includegraphics[trim=130 179 40 195,clip,width=.31\textwidth]{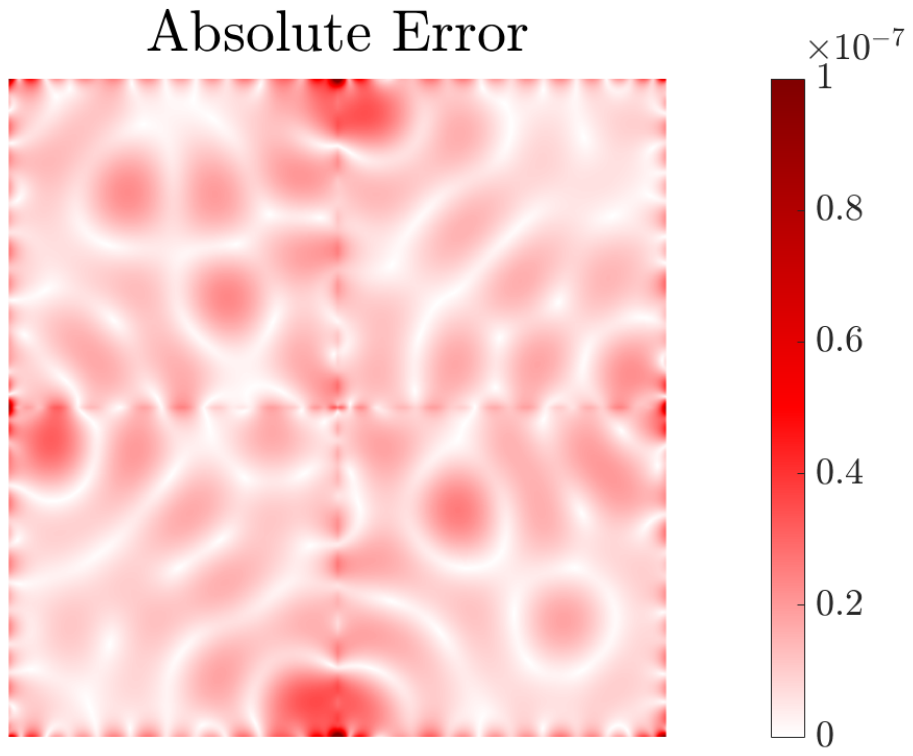}
\vspace{-.2cm}

\caption{Approximation of $u(\mathbf{x}) = e^{\imath \kappa \mathbf{d}\cdot \mathbf{x}}$, with $\mathbf{d} = (1/\sqrt{2}, 1/\sqrt{2})$ and $\kappa=30$, in $\Omega=(0,1)^2$. $\mathcal{T}_h$ has 4 square elements with $h=\frac{1}{2}$.
Left: $H_\kappa^1(\Omega)$-norm relative error vs.\ \#DOFs, varying $N_{\mathbf{e}}$ and for different $N_{\mathbf{n}}$.
Center and right: real part $\Re u$ and absolute error $|u - u_{\mathbf{N}, \epsilon}|$, for $\mathbf{N} = (8,5)$.}
\label{fig: plane wave}
\end{figure}

We consider the domain $\Omega = (0,1)^2$ with a mesh $\mathcal{T}_h$ of four square elements, as in the projection test of Section \ref{sec: approximation limitations}.
Accordingly, $|\Sigma_h| = 12$ and $|\mathcal{N}_h| = 9$. The target function is again the PPW $u(\mathbf{x})=e^{\imath \kappa \mathbf{d}\cdot \mathbf{x}}$ with direction $\mathbf{d} = (1/\sqrt{2}, 1/\sqrt{2})$.

Figure \ref{fig: plane wave} shows the numerical results for a fixed wavenumber $\kappa = 30$.
The central and right sections display the real part $\Re u$ and the absolute error $|u - u_{\mathbf{N}, \epsilon}|$, obtained using $N_{\mathbf{e}} = 8$ basis functions per edge and $N_{\mathbf{n}} = 5$ per node.
On the left, the relative error in the $H_\kappa^1(\Omega)$-norm is plotted against the total number of DOFs, given by $\#\textup{DOFs}=N_{\mathbf{e}} \times |\Sigma_h| + N_{\mathbf{n}} \times |\mathcal{N}_h|$. We vary $N_{\mathbf{e}}$ for different choices of $N_{\mathbf{n}}$. As $u$ is analytic on $\overline{\Omega}$, we observe that the convergence order predicted by the approximation estimates in Corollary \ref{cor: corollary Cinfty} is confirmed, with a $\frac{1}{2}$ gain.

In Figure~\ref{fig: plane wave 2} we report numerical results for different wavenumbers $\kappa$ on the fixed mesh $\mathcal{T}_h$. Both edge and node parameters, $N_{\mathbf{e}}$ and $N_{\mathbf{n}}$, are chosen to scale linearly with $\kappa h$; specifically $N_{\mathbf{n}}=\kappa h$ and $N_{\mathbf{e}}=\tau_{\mathbf{e}}\kappa h$ for several values of the tuning parameter $\tau_{\mathbf{e}}$.
On the left, we display the relative error in the $H_\kappa^1(\Omega)$-norm as a function of $\kappa h$. These results are consistent with Corollary~\ref{cor: high freq Cinfty}, exhibiting geometric convergence as $\kappa h \to +\infty$, with a higher convergence rate for larger values of $\tau_{\mathbf{e}}$.
On the right, we report the normalized coefficient norm $\|\boldsymbol{\xi}^{(\epsilon)}\|_{\ell^2}/\|u\|_{H^1_\kappa(\Omega)}$ as a measure of stability, plotted against $\kappa h$. The coefficients remain small across all tested regimes. The exponential upper bound of Corollary~\ref{cor: high freq smooth coeff} is not realized in practice and largely overestimates the actual behavior, which in fact does not result in any numerical instabilities.

\begin{figure}
\centering
\includegraphics[trim=0 0 0 0,clip,height=.3\textwidth]{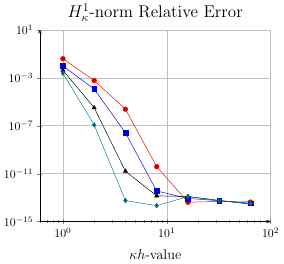}
\hspace{.1\textwidth}
\centering
\includegraphics[trim=0 0 0 0,clip,height=.3\textwidth]{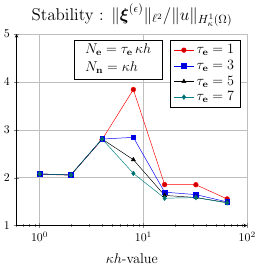}
\vspace{-.1cm}

\caption{Approximation of the plane wave $u(\mathbf{x}) = e^{\imath \kappa \mathbf{d}\cdot \mathbf{x}}$ with $\mathbf{d} = (1/\sqrt{2}, 1/\sqrt{2})$ in $\Omega=(0,1)^2$, where $\mathcal{T}_h$ consists of 4 square elements with $h=\frac{1}{2}$.
$H_\kappa^1(\Omega)$-norm relative error (left) and relative coefficient norm $\|\boldsymbol{\xi}^{(\epsilon)}\|_{\ell^2}/\|u\|_{H^1_\kappa(\Omega)}$ (right) vs.\ $\kappa h$, with $N_{\mathbf{e}}$ and $N_{\mathbf{n}}$ scaling linearly w.r.t.\ $\kappa h$.}
\label{fig: plane wave 2}
\end{figure}

\begin{figure}[t]
\centering
\includegraphics[trim=0 0 0 0,clip,height=.298\textwidth]{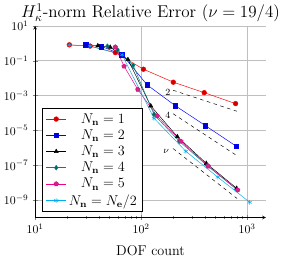}
\hspace{.01\textwidth}
\includegraphics[trim=130 179 40 195,clip,width=.31\textwidth]{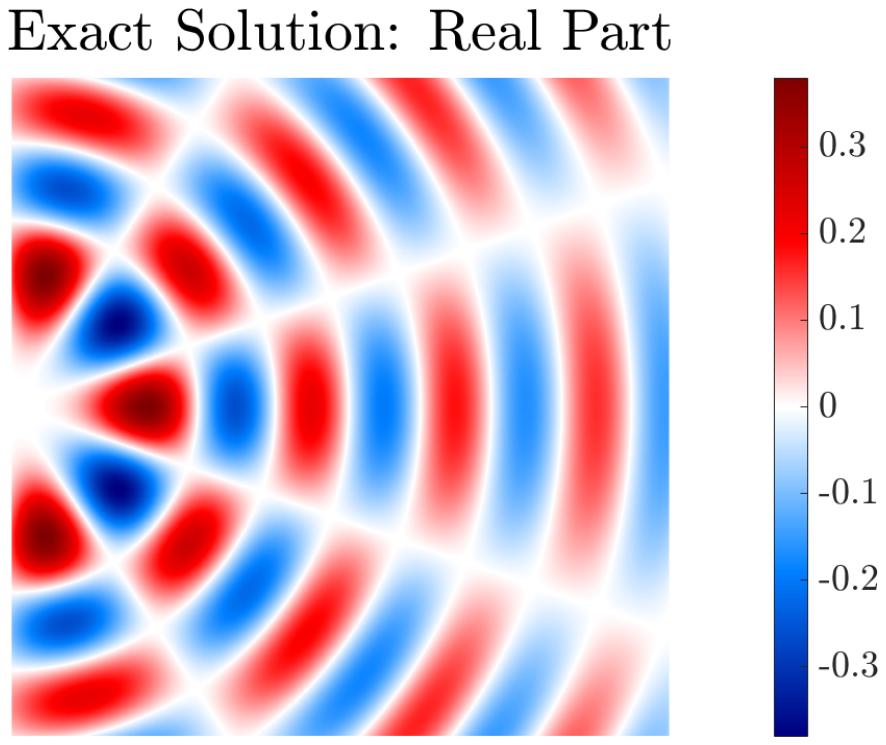}
\includegraphics[trim=130 179 40 195,clip,width=.31\textwidth]{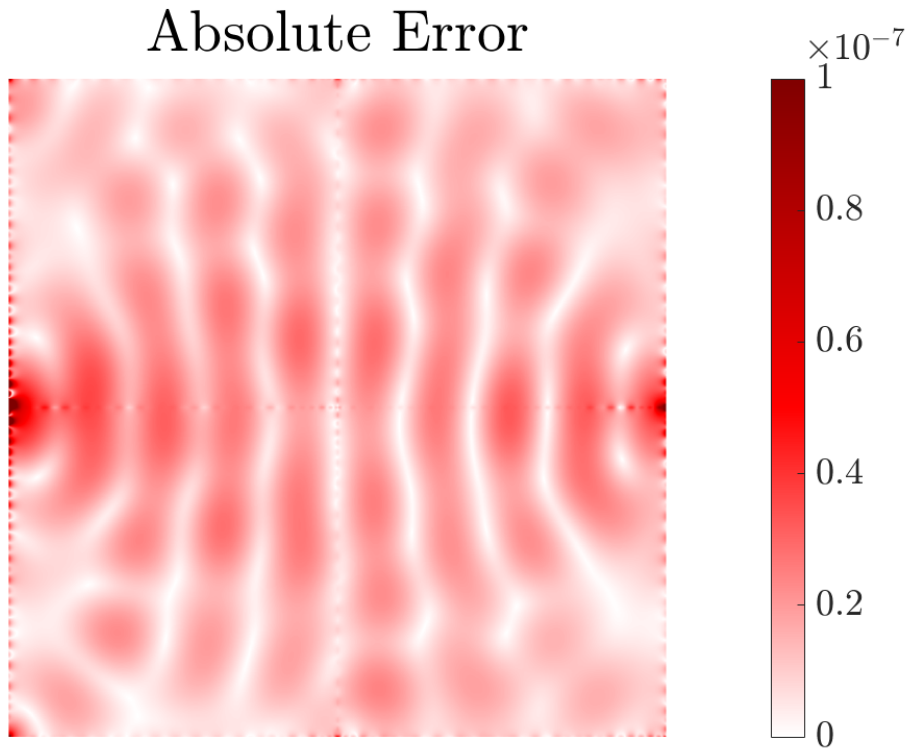}

\includegraphics[trim=0 0 0 0,clip,height=.298\textwidth]{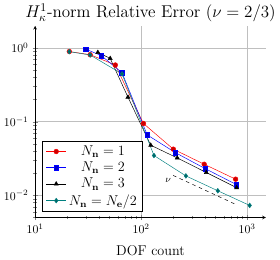}
\hspace{.01\textwidth}
\includegraphics[trim=130 179 40 180,clip,width=.31\textwidth]{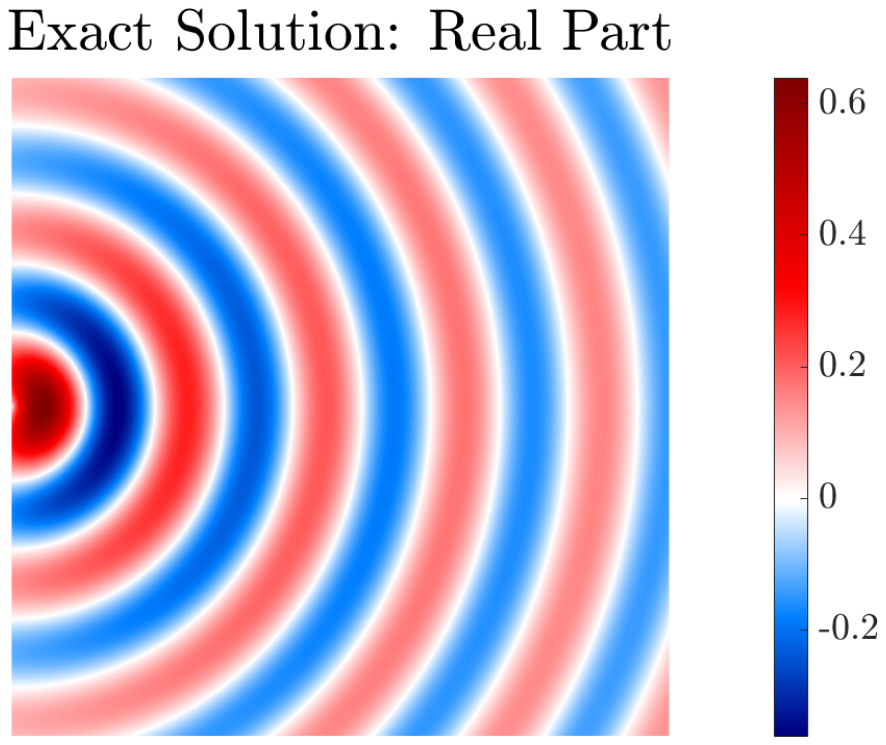}
\includegraphics[trim=130 179 40 180,clip,width=.31\textwidth]{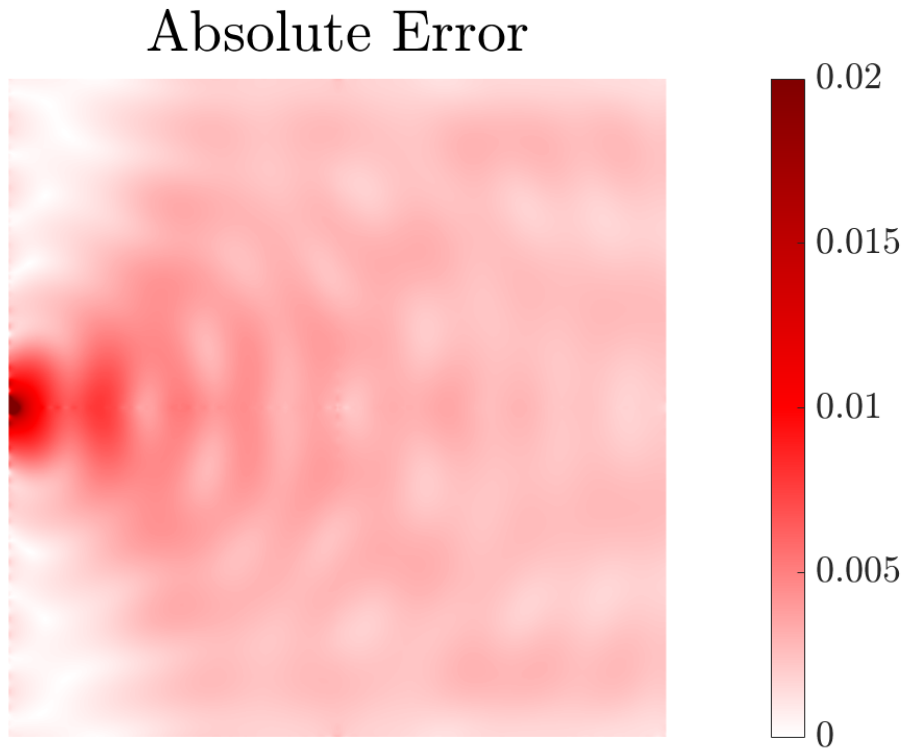}
\vspace{-.2cm}

\caption{Approximation of corner singularities $u_{\nu}$ in \eqref{eq: corner sing}, with $\nu=19/4$ (top) and $\nu=2/3$ (bottom), in $\Omega=(0,1)\times(-\frac{1}{2},\frac{1}{2})$, where $\mathcal{T}_h$ consists of 4 square elements with $h=\frac{1}{2}$.
Left column: $H_\kappa^1(\Omega)$-norm relative error vs.\ \#DOFs, varying $N_{\mathbf{e}}$ and for different $N_{\mathbf{n}}$.
Central and right columns: real part $\Re u_{\nu}$ and absolute error $|u_{\nu} - u_{\mathbf{N}, \epsilon}|$, for $\mathbf{N}=(30,3)$.}
\label{fig: corner sing}
\end{figure}

\paragraph{Corner singularities.}
We consider the domain $\Omega = (0,1) \times (-\tfrac{1}{2}, \tfrac{1}{2})$, discretized using the same computational mesh $\mathcal{T}_h$ of 4 square elements as in the previous example.
We now approximate more challenging solutions, specifically
\begin{equation} \label{eq: corner sing}
u_\nu(\mathbf{x}) := J_{\nu}(\kappa r)e^{\imath \nu \theta}, \qquad \mathbf{x} = (r \cos\theta, r \sin\theta), \qquad \text{where} \quad \nu \in \mathbb{R} \setminus \mathbb{Z},
\end{equation}
and $J_\nu$ is the Bessel function of the first kind \cite[Eq.\ (10.2.2)]{DLMF}. As $\nu \not \in \mathbb{Z}$, these functions exhibit a singularity at the origin on the boundary $\partial \Omega$, which becomes stronger as $\nu$ decreases. In fact, the function \eqref{eq: corner sing} belongs to $H_\kappa^{1+\nu-\lambda}(\Omega)$ for any $\lambda > 0$, but $u_\nu \notin H_\kappa^{1+\nu}(\Omega)$.

In Figure \ref{fig: corner sing}, we report the results for a fixed wavenumber $\kappa = 30$. In the first row, we set $\nu = 19/4$, while in the second row, $\nu = 2/3$.
The central and right columns show the real part $\Re u_\nu$ and the point-wise absolute error $|u_\nu - u_{\mathbf{N}, \epsilon}|$.
In the left column, we plot the relative error in the $H_\kappa^1(\Omega)$-norm as a function of the number of DOFs, varying $N_{\mathbf{e}}$ and for various choices of $N_{\mathbf{n}}$.
We observe that the convergence rate initially grows with $2N_{\mathbf{n}}$ when $2N_{\mathbf{n}} < \nu$, but saturates at the rate $\nu$ when $2N_{\mathbf{n}} > \nu$. This behavior indicates that the convergence rate in Theorem \ref{th: best approximation} may be sharpened by a factor of $\frac{1}{2}$, and that the result may extend to all Helmholtz solutions $u \in H_\kappa^{M+1}(\Omega)$ with a real, non-integer $M>0$.

In Figure~\ref{fig: corner sing 2} we report numerical results for the case $\nu = 2/3$, and various wavenumbers $\kappa$ on the same mesh $\mathcal{T}_h$. The edge parameter $N_{\mathbf{e}}$ is set to scale linearly with $\kappa h$ -- specifically $N_{\mathbf{e}} = \tau_{\mathbf{e}} \kappa h$ for several values of the tuning parameter $\tau_{\mathbf{e}}$ -- while the node parameter is fixed at $N_{\mathbf{n}} = 1$, as suggested by the previous analysis due to the limited Sobolev regularity of the target solution.
On the left, we show the relative error in the $H_\kappa^1(\Omega)$-norm as a function of $\kappa h$. We observe that the error slightly decreases in $\kappa h$ and eventually plateaus. This suggests that the polynomial growth predicted by Corollary~\ref{cor: rough high freq} is not necessary in practice, and that a linear scaling of $N_{\mathbf{e}}$ with $\kappa h$ suffices to ensure $\kappa$-independent accuracy -- even for solutions with regularity below $H^2_\kappa(\Omega)$.
On the right, we report the normalized coefficient norm $\|\boldsymbol{\xi}^{(\epsilon)}\|_{\ell^2}/\|u_{2/3}\|_{H^1_\kappa(\Omega)}$ as a measure of numerical stability, plotted against $\kappa h$. The coefficients remain small across all tested regimes, confirming that stability is preserved when $N_{\mathbf{e}}$ grows linearly with $\kappa h$ without requiring any stronger increase of the edge basis functions.

\begin{figure}
\centering
\includegraphics[trim=0 0 0 0,clip,height=.3\textwidth]{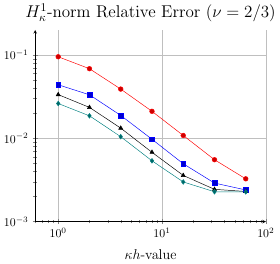}
\hspace{.1\textwidth}
\centering
\includegraphics[trim=0 0 0 0,clip,height=.3\textwidth]{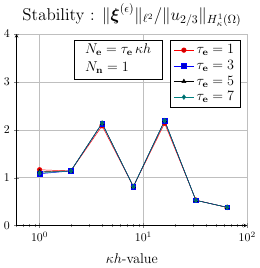}
\vspace{-.1cm}

\caption{Approximation of the corner singularity $u_{\nu}$ in \eqref{eq: corner sing} with $\nu=2/3$ in $\Omega=(0,1)\times(-\frac{1}{2},\frac{1}{2})$, where $\mathcal{T}_h$ consists of 4 square elements with $h=\frac{1}{2}$.
$H_\kappa^1(\Omega)$-norm relative error (left) and relative coefficient norm $\|\boldsymbol{\xi}^{(\epsilon)}\|_{\ell^2}/\|u_{2/3}\|_{H^1_\kappa(\Omega)}$ (right) vs.\ $\kappa h$, with $N_{\mathbf{e}}$ and $N_{\mathbf{n}}$ scaling linearly w.r.t.\ $\kappa h$.}
\label{fig: corner sing 2}
\end{figure}

\paragraph{Star shaped domain.} We now consider a star-shaped domain $\Omega$ inscribed in the unit circle $\{\mathbf{x} \in \mathbb{R}^2 : |\mathbf{x}| = 1\}$. The mesh $\mathcal{T}_h$ consists of 4 rectangular cells, covering $\Omega$ as shown in Figure \ref{fig: cropped star}, with $|\Sigma_h| = 12$ and $|\mathcal{N}_h| = 9$. The wavenumber is fixed at $\kappa = 30$.

First, we approximate the same PPW $u(\mathbf{x})=e^{\imath \kappa \mathbf{d}\cdot \mathbf{x}}$ with $\mathbf{d} = (1/\sqrt{2}, 1/\sqrt{2})$, as in Figure \ref{fig: plane wave}.
The results are shown in Figure \ref{fig: plane wave star}.
The center and right panels show the real part $\Re u$ and the absolute error $|u - u_{\mathbf{N}, \epsilon}|$, obtained with $\mathbf{N}=(8,8)$. The left panel plots the relative $H_\kappa^1$-error against the number of DOFs, varying $N_{\mathbf{e}}$ and for different values of $N_{\mathbf{n}}$.
Despite the generality of the domain, which limits the ability of the $V_{\mathbf{N}}(\mathcal{T}_h)$ basis functions to adapt near the boundary, we still observe good approximation properties, with convergence plateauing at about $10^{-8}$.

\begin{figure}
\centering
\includegraphics[trim=0 0 0 0,clip,height=.298\textwidth]{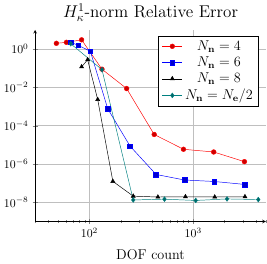}
\hspace{.01\textwidth}
\includegraphics[trim=130 179 40 195,clip,width=.31\textwidth]{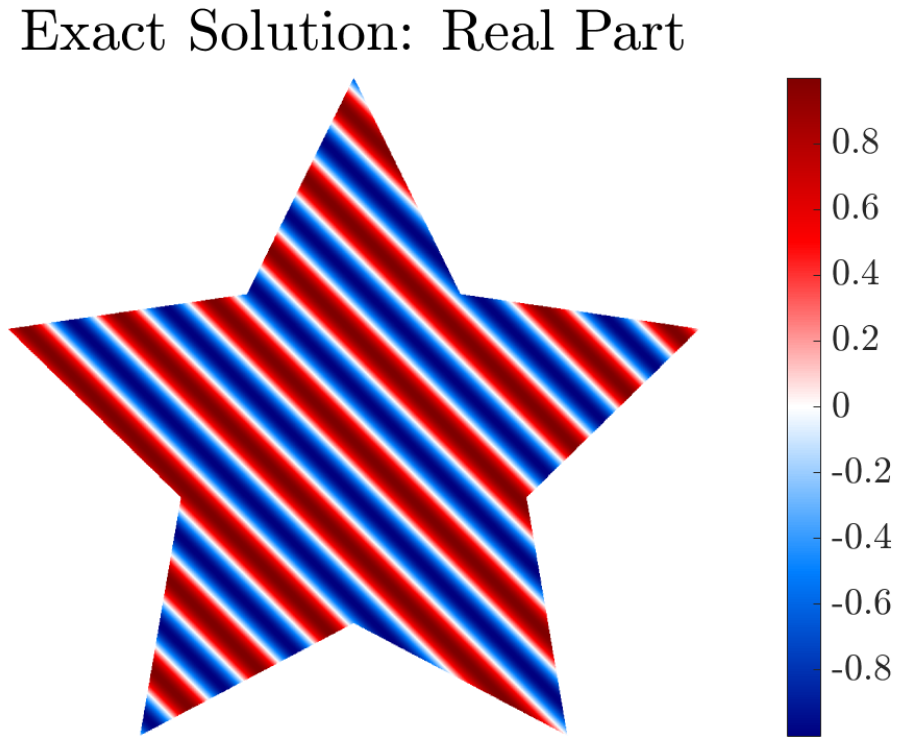}
\includegraphics[trim=130 179 40 195,clip,width=.31\textwidth]{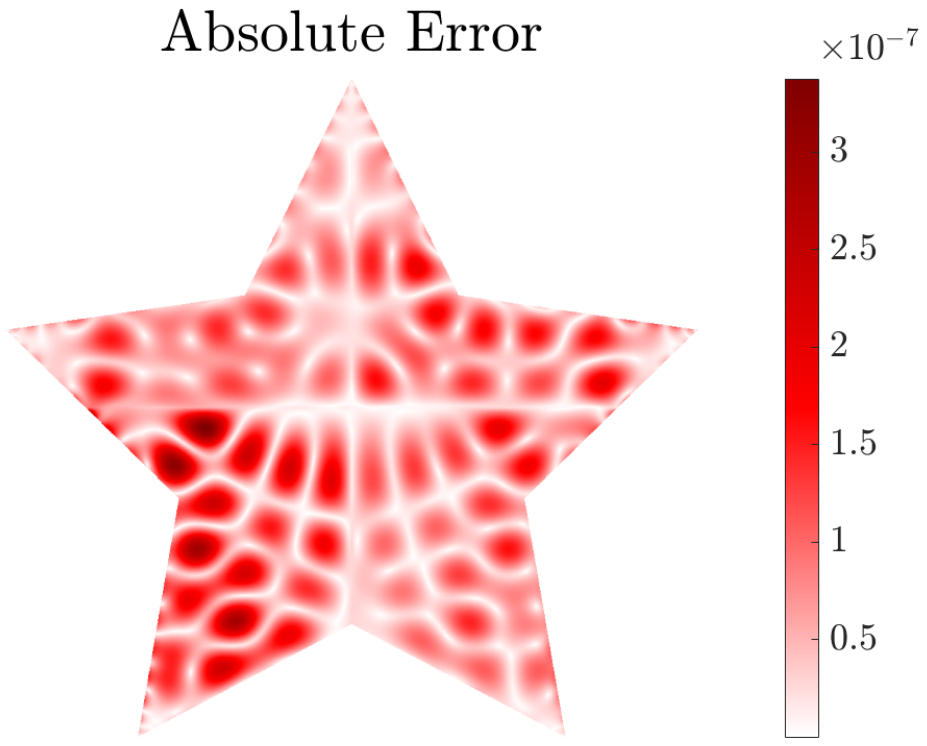}
\vspace{-.2cm}

\caption{Approximation of the plane wave $u(\mathbf{x}) = e^{\imath \kappa \mathbf{d}\cdot \mathbf{x}}$ with $\mathbf{d} = (1/\sqrt{2}, 1/\sqrt{2})$ in a star-shaped domain $\Omega$, where $\mathcal{T}_h$ consists of 4 rectangular elements with $h=\sqrt{2+\varphi}/2$, where $\varphi$ is the golden ratio.
Left: $H_\kappa^1(\Omega)$-norm relative error vs.\ \#DOFs, varying $N_{\mathbf{e}}$ and for different values of $N_{\mathbf{n}}$.
Center and right: real part $\Re u$ and absolute error $|u - u_{\mathbf{N}, \epsilon}|$, for $\mathbf{N} = (8,8)$.
}
\label{fig: plane wave star}
\vspace{.6cm}

\includegraphics[trim=0 0 0 0,clip,height=.298\textwidth]{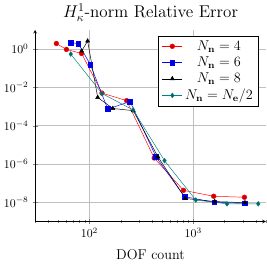}
\hspace{.01\textwidth}
\includegraphics[trim=130 179 40 190,clip,width=.31\textwidth]{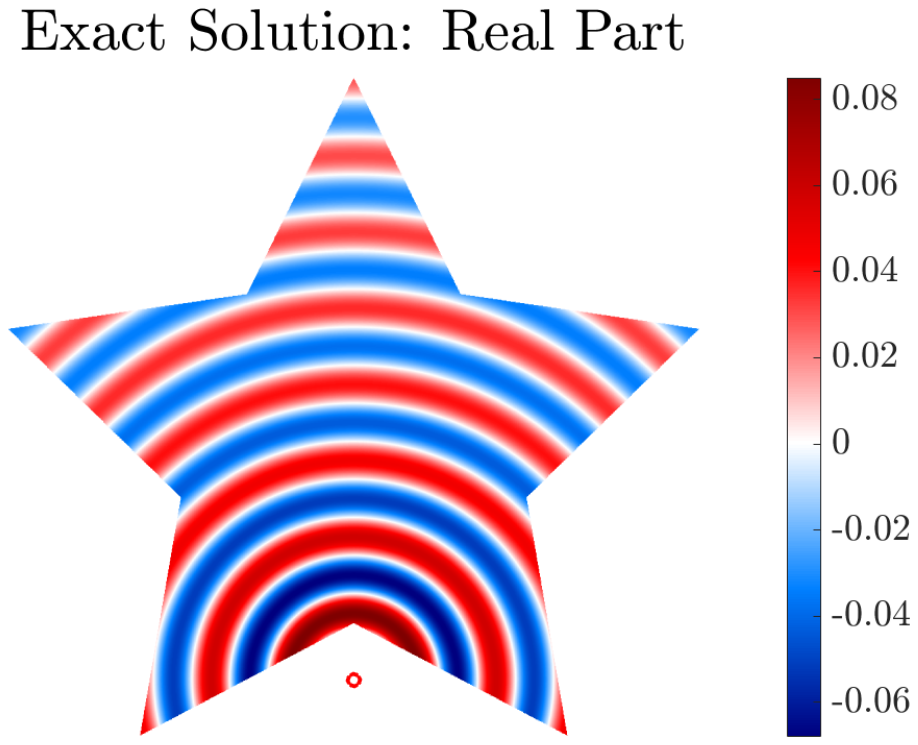}
\includegraphics[trim=130 179 40 190,clip,width=.31\textwidth]{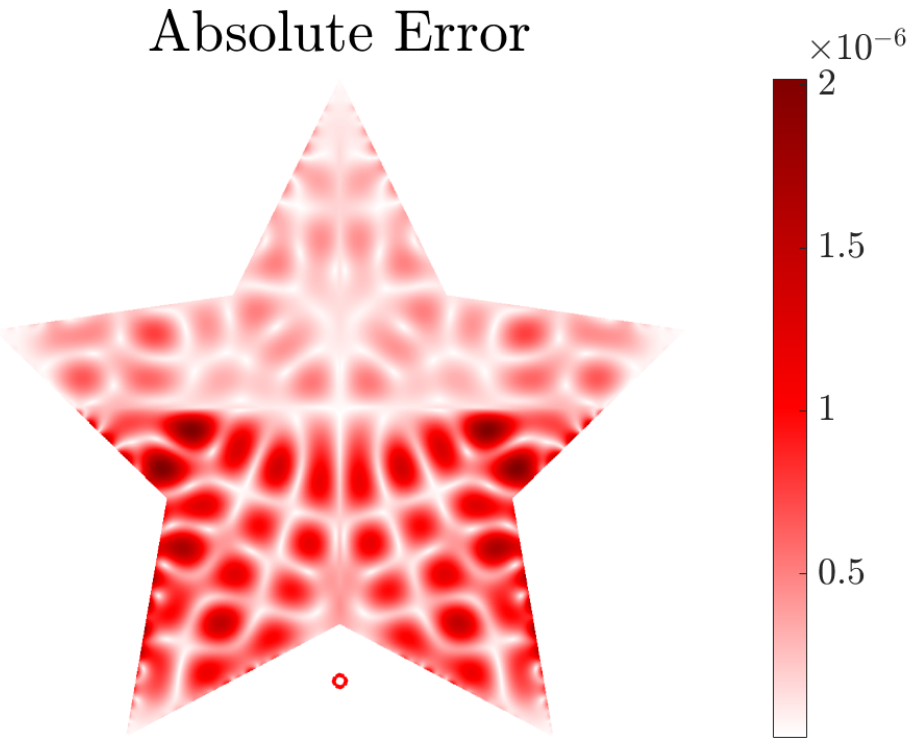}
\vspace{-.2cm}

\caption{Approximation of the fundamental solution $\Phi_{\mathbf{x}_0}$ in \eqref{eq: fundamental solution}  with $\kappa=30$ in a star-shaped domain $\Omega$, where $\mathcal{T}_h$ consists of 4 rectangular elements with $h=\sqrt{2+\varphi}/2$, where $\varphi$ is the golden ratio. The singularity $\mathbf{x}_0$, marked by a red circle, lies on one edge of $\Sigma_h$ at a distance $3\lambda/4$ from the domain boundary.
Left: $H_\kappa^1(\Omega)$-norm relative error vs.\ \#DOFs, varying $N_{\mathbf{e}}$ and for different $N_{\mathbf{n}}$.
Center and right: real part $\Re\Phi_{\mathbf{x}_0}$ and absolute error $|\Phi_{\mathbf{x}_0} - u_{\mathbf{N}, \epsilon}|$, for $\mathbf{N} = (30,5)$.
}
\label{fig: fund star}
\vspace{.6cm}

\centering
\includegraphics[trim=0 0 0 0,clip,height=.298\textwidth]{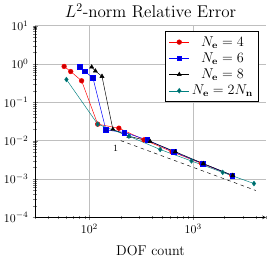}
\hspace{.01\textwidth}
\includegraphics[trim=130 179 40 195,clip,width=.31\textwidth]{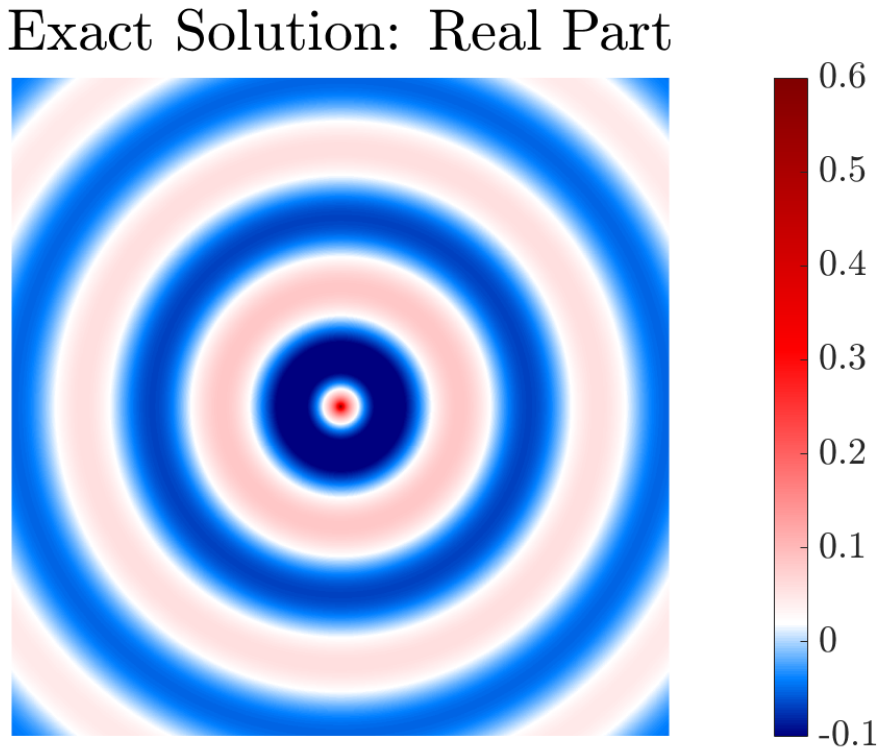}
\includegraphics[trim=130 179 40 195,clip,width=.31\textwidth]{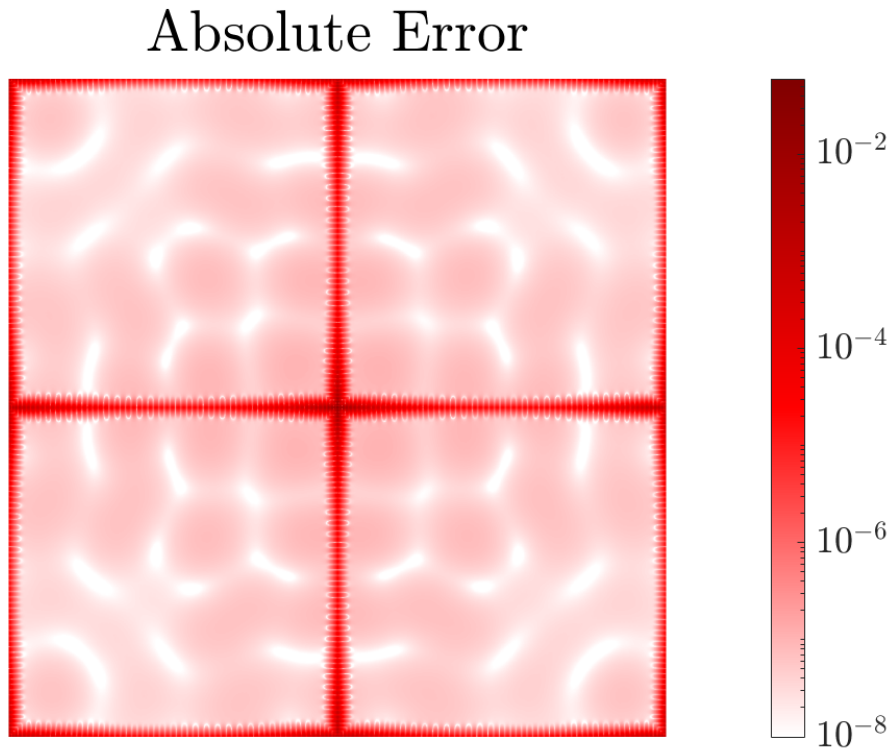}
\vspace{-.2cm}

\caption{Approximation of the fundamental solution $\Phi_{\mathbf{x}_0}$ in \eqref{eq: fundamental solution},  with $\kappa=30$ and $\mathbf{x}_0\in \mathcal{N}_h$ located at the domain center, in $\Omega=(-\frac{1}{2},\frac{1}{2})^2$, where $\mathcal{T}_h$ consists of 4 square elements with $h=\frac{1}{2}$.
Left: $L^2(\Omega)$-norm relative error vs.\ \#DOFs, varying $N_{\mathbf{n}}$ and for different $N_{\mathbf{e}}$.
Center and right: real part $\Re\Phi_{\mathbf{x}_0}$ and absolute error $|\Phi_{\mathbf{x}_0} - u_{\mathbf{N}, \epsilon}|$, for $\mathbf{N} = (30,50)$.
}
\label{fig: middle delta}
\end{figure}

Then, we consider another numerical test, where we approximate the fundamental solution to the Helmholtz equation, given by
\begin{equation} \label{eq: fundamental solution}
    \Phi_{\mathbf{x}_0}(\mathbf{x}) := \frac{\imath}{4} H_0^{(1)}(\kappa|\mathbf{x}-\mathbf{x}_0|), \qquad \mathbf{x}_0 \notin \overline{\Omega},
\end{equation}
where $H_0^{(1)}$ is the Hankel function of the first kind \cite[Eq.\ (10.2.5)]{DLMF}. 
We place the singularity $\mathbf{x}_0$ on an edge of the skeleton $\Sigma_h$ at distance $3\lambda/4$ from the domain boundary, where $\lambda = 2\pi/\kappa$ denotes the wavelength.
The results are shown in Figure \ref{fig: fund star}.
The center and right panels show the real part $\Re\Phi_{\mathbf{x}_0}$ and the absolute error $|\Phi_{\mathbf{x}_0} - u_{\mathbf{N}, \epsilon}|$, obtained with $\mathbf{N}=(30,5)$; the singularity $\mathbf{x}_0$ is marked by a red circle. The left panel plots the relative $H_\kappa^1(\Omega)$-error against the number of DOFs, for different values of $N_{\mathbf{n}}$ while varying $N_{\mathbf{e}}$.
Despite the domain’s complexity, which hinders the ability of the $V_{\mathbf{N}}(\mathcal{T}_h)$ basis to capture boundary behavior, we observe good approximation. Nevertheless convergence stagnates around $10^{-8}$. The test is challenging, as the singularity lies outside $\Omega$ but inside the surrounding rectangular mesh.

\paragraph{Internal nodal point source.}
We return to the domain $\Omega = (0,1)^2$ with wavenumber $\kappa = 30$, using the same computational mesh $\mathcal{T}_h$ of four square elements as in the previous examples.
We approximate again $\Phi_{\mathbf{x}_0}$ in \eqref{eq: fundamental solution}, but this time with a singularity $\mathbf{x}_0$ located at the center of $\Omega$, so that $\mathbf{x}_0 \in \mathcal{N}_h$.
In this setting, $\Phi_{\mathbf{x}_0}$ is no longer a solution to the Helmholtz equation \eqref{eq:helmholtz_equation}, and hence not a solution to the variational problem \eqref{eq: robin problem}. Instead, it weakly satisfies the inhomogeneous equation $-\Delta u - \kappa^2 u = \delta_{\mathbf{x}_0}$, where $\delta_{\mathbf{x}_0}$ denotes the Dirac distribution centered at $\mathbf{x}_0$.
We approximate $\Phi_{\mathbf{x}_0}$ by solving the Galerkin problem \eqref{eq: galerkin problem} with a modified right-hand side $\mathcal{F} + \delta_{\mathbf{x}_0}$, which is well defined on the discrete space $V_{\mathbf{N}}(\mathcal{T}_h) \subset C^0(\overline{\Omega})$. However, as $\delta_{\mathbf{x}_0} \notin (H^1_\kappa(\Omega))^*$, it follows that $\Phi_{\mathbf{x}_0} \notin H^1_\kappa(\Omega)$, and no convergence in the $H^1_\kappa(\Omega)$-norm can be expected.

For this reason, the left panel of Figure \ref{fig: middle delta} reports the relative $L^2(\Omega)$-error versus the total number of degrees of freedom. Unlike the previous tests, where the number of DOFs was increased by raising $N_{\mathbf{e}}$ while keeping $N_{\mathbf{n}}$ fixed, here we keep $N_{\mathbf{e}}$ fixed and increase $N_{\mathbf{n}}$, as no convergence is observed numerically when increasing $N_{\mathbf{e}}$ alone.
Note that the error analysis of Section \ref{sec: Combined edge-and-node Trefftz space} does not cover this case.
Besides, the center and right panels of Figure \ref{fig: middle delta} show the real part $\Re\Phi_{\mathbf{x}_0}$ and the absolute error $|\Phi_{\mathbf{x}_0} - u_{\mathbf{N}, \epsilon}|$, respectively, computed with $\mathbf{N} = (30,50)$. The error is clearly concentrated along the mesh skeleton $\Sigma_h$.

\paragraph{Square-tiled domain.} Finally, we consider a conforming mesh $\mathcal{T}_h$ for a polygonal domain $\Omega$, tessellated with square cells; see Figure \ref{fig: space invader plot}.
$\Omega$ is contained within the rectangle $(0,9) \times (0,8)$, and the mesh is composed of $|\mathcal{T}_h|=56$ square cells, each of unit side length, it has $|\Sigma_h| = 141$ edges and $|\mathcal{N}_h| = 84$ nodes, and the domain is contained within the rectangle $(0,9) \times (0,8)$.
We focus on a high-frequency regime by setting $\kappa = 100$, aiming to approximate the fundamental solution $\Phi_{\mathbf{x}_0}$ given in \eqref{eq: fundamental solution}. The characteristic diameter of the domain is $\textup{diam}(\Omega) \approx 11$, which together with the high wavenumber makes this a numerically challenging problem.
Besides, the singularity $\mathbf{x}_0$ is placed at a distance $\textup{dist}(\mathbf{x}_0, \Omega) = \lambda/4$ from the domain, with $\lambda = 2\pi/\kappa$ denoting the wavelength; it is marked by a red circle in Figure \ref{fig: space invader plot}. The left panel of the figure shows the real part $\Re\Phi_{\mathbf{x}_0}$, while the right one displays the absolute error obtained with $\mathbf{N} = (40,40)$, corresponding to a total of $\#\textup{DOFs} = N_{\mathbf{e}} |\Sigma_h| + N_{\mathbf{n}} |\mathcal{N}_h| = 9\times 10^3$.
In this setting, the number of DOFs per wavelength in each direction can be estimated as
\begin{equation*}
\lambda \frac{\sqrt{\#\textup{DOFs}}}{\textup{diam}(\Omega)} = \frac{2\pi}{\kappa} \frac{\sqrt{N_{\mathbf{e}}|\Sigma_h| + N_{\mathbf{n}}|\mathcal{N}_h|}}{\textup{diam}(\Omega)} \approx 0.5.
\end{equation*}
Low-order methods usually demand 6-10 DOFs per wavelength for 1-2 digits of accuracy; in contrast, the presented Trefftz method attains an $L^{\infty}$-error of order $10^{-8}$ with only a small fraction of that resolution.

\begin{figure}
\centering
\includegraphics[trim=120 240 40 195,clip,width=.49\textwidth]{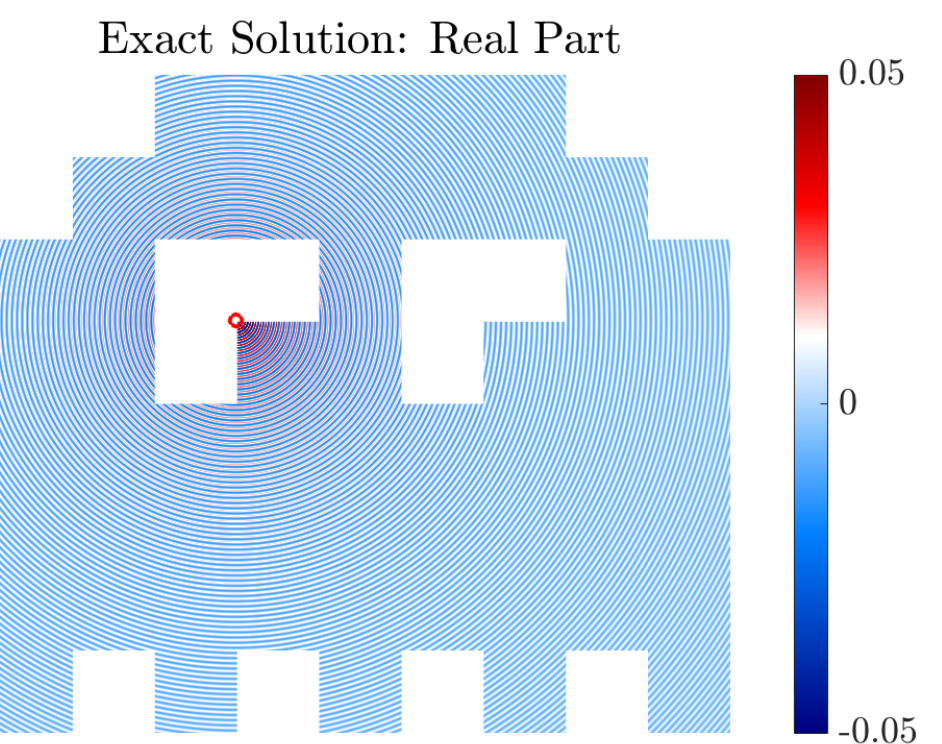}
\hfill
\includegraphics[trim=120 240 40 195,clip,width=.49\textwidth]{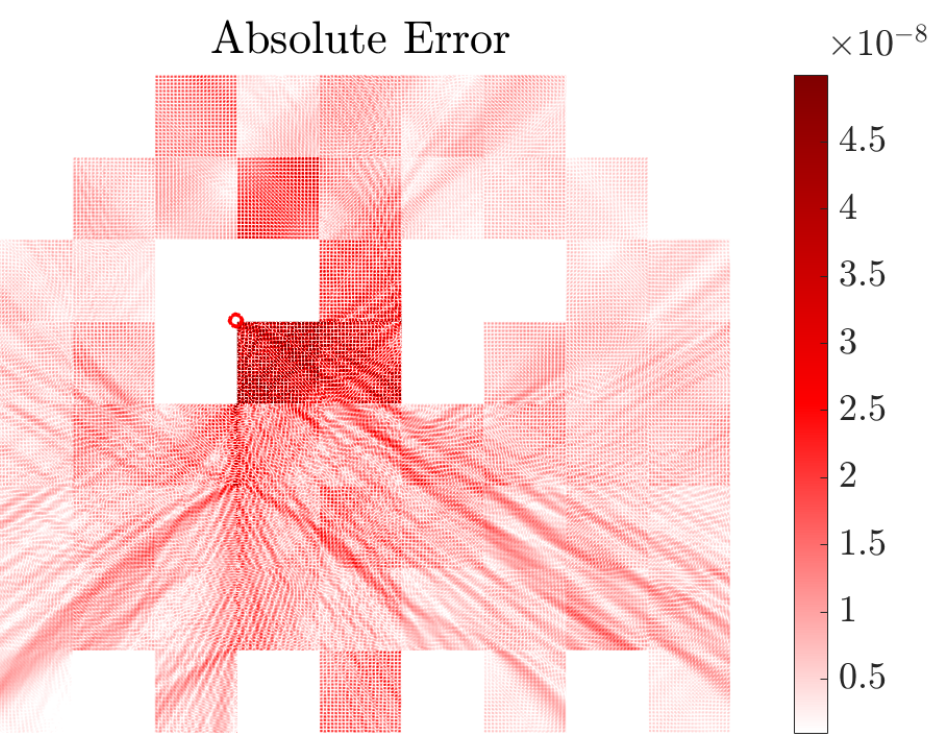}
\caption{Approximation of the fundamental solution $\Phi_{\mathbf{x}_0}$ in \eqref{eq: fundamental solution} with $\kappa=100$ in the space invader domain $\Omega$, where $\mathcal{T}_h$ consists of 56 square elements with $h=1$. The singularity $\mathbf{x}_0$ is located at a distance $\lambda/4$ from the domain boundary and marked by a red circle. Left: real part $\Re\Phi_{\mathbf{x}_0}$. Right: absolute error $|\Phi_{\mathbf{x}_0} - u_{\mathbf{N},\epsilon}|$, computed with $\mathbf{N} = (40,40)$, and hence $9 \times 10^3$ DOFs.
}
\label{fig: space invader plot}
\end{figure}

\section{Conclusions}

This paper presents a novel Trefftz Continuous Galerkin method. The discrete space, built on a Cartesian grid, is $H^1_\kappa(\Omega)$-conforming, and spanned by compactly supported basis functions given locally as simple linear combinations of EPWs. For this space, we establish wavenumber-explicit stability and error bounds. Numerical results for the full Galerkin discretization confirm the theoretical bounds and even exhibit better performance, pointing to possible improvements.

Building on these results, several natural directions for further development can be considered.
In particular, a discontinuous Trefftz space can be constructed from the local approximation space on each cell, offering greater flexibility for problems with piecewise-constant coefficients.
This approach could also allow for the definition of local Trefftz spaces on elements with more general geometries, potentially enabling the method to handle more complex meshes.
Similarly, while the focus of Section \ref{sec: Local edge-based Trefftz space} is on Dirichlet boundary conditions, an analogous local approach can be applied to Neumann problems. By defining edge-based reference modes through a Hilbert basis and solving local Helmholtz--Neumann problems within each cell, new families of basis functions can be constructed.
This enables the construction of globally conforming spaces in $H(\mathrm{div};\Omega)$, suitable, for instance, for formulations of the Helmholtz equation involving both scalar and vector fields.
Other interesting directions include deriving full Galerkin error estimates, generalizing the numerical scheme and analysis to the 3D settings, and further extending the approach to time-harmonic Maxwell and elastic wave equations.

\paragraph{Acknowledgements.} The authors are grateful to Bernardo Cockburn for asking, at the 2024 Trefftz Workshop in Oaxaca, whether Trefftz Continuous Methods were possible, a question that led to this paper.

\appendix
\section{Appendix}

\paragraph{A trigonometric inequality.}
We begin with a basic trigonometric inequality, which is a key tool in the proof of Lemma~\ref{lem:orthogonality lemma}.

\begin{lemma} \label{lem: coth inequality}
    The following inequalities hold:
    \begin{equation} \label{eq: coth inequality}
        \left|\frac{\cot(x)}{x}\right|\leq \frac{1}{\sin^2(x)}\leq \frac{\pi^2}{4\min_{m \in \mathbb{N}}|x-m\pi |^2}, \qquad  0<x\notin \pi\mathbb{Z}.
    \end{equation}
\end{lemma}
\begin{proof}
    The first inequality in \eqref{eq: coth inequality} follows directly from the bound $|\sin(x)\cos(x)| \leq |x|$ for $x>0$. For the second inequality, note that by periodicity one can write $x = n\pi \pm d$, with $n \in \mathbb{N}$ and
    \begin{equation*}
    d := \min_{m \in \mathbb{N}} |x - m\pi | \in (0,\pi/2],
    \end{equation*}
    since $x \notin \pi\mathbb{Z}$.
    Using the elementary estimate $\sin(t) \geq \tfrac{2t}{\pi}$ for $t \in [0,\pi/2]$, we obtain $|\sin(x)| = |\sin(d)| \geq \frac{2d}{\pi}$, which yields the second inequality in \eqref{eq: coth inequality}.
\end{proof}

\paragraph{Trace inequalities on the reference cell.}

We next recall a couple of trace inequalities on the reference cell $\widehat{K}$ and its edge $\hat{\mathbf{s}}$ defined in \eqref{eq: reference cell and edge}. These results, though stated on $\widehat{K}$, extend to any mesh cell $K$ and edge $\mathbf{s}\subset \partial K$ up to a rigid transformation.

\begin{lemma} \label{lem: trace inequality}
    Let $M\in \mathbb{N}$, and $u \in H_\kappa^{M+1}(\widehat{K})$. Then,
    \begin{equation*}
        | u|_{H^{M}_{\kappa}(\hat{\mathbf{s}})}\leq \sqrt{\frac{2}{\widehat{h}_2}}\max(1,\kappa \widehat{h}_2)\|u\|_{H^{M+1}_{\kappa}(\widehat{K})}.
    \end{equation*}
\end{lemma}
\begin{proof}
    Let us first assume $u\in C^{\infty}(\overline{\widehat{K}})$ and $M=0$. From the mean value theorem, there exists $y_0\in[0,\widehat{h}_2]$ such that
    \begin{equation*}
         u(x,y_0)=\frac{1}{\widehat{h}_2}\int_0^{\widehat{h}_2} u(x,t)\,\textup{d}t, \qquad x\in[0,\widehat{h}_1],
    \end{equation*}
    and, from the fundamental theorem of calculus,
    \begin{equation*}
         u(x,y)= u(x,y_0)+\int_{y_0}^{y}\partial_2 u (x,t)\,\textup{d}t, \qquad (x,y)\in \partial \widehat{K}.
    \end{equation*}
    Hence, it follows that
    \begin{align}
        \frac{| u(x,y)|^2}{2}&\leq | u(x,y_0)|^2+\left|\int_{y_0}^{y}\partial_2 u (x,t)\,\textup{d}t\right|^2
        \leq \frac{1}{\widehat{h}_2^2}\left(\int_0^{\widehat{h}_2}| u(x,t)|\,\textup{d}t\right)^2+\left(\int_{0}^{\widehat{h}_2}|\partial_2 u (x,t)|\,\textup{d}t\right)^2 \nonumber\\
        &\leq \frac{1}{\widehat{h}_2}\int_0^{\widehat{h}_2}| u(x,t)|^2\,\textup{d}t+\widehat{h}_2\int_{0}^{\widehat{h}_2}|\partial_2 u (x,t)|^2\,\textup{d}t, \label{eq: lemma trace useful}
    \end{align}
    where we used the Cauchy--Schwarz inequality in the last step. Choosing $y=\widehat{h}_2$ and integrating in $x$ over $(0,\widehat{h}_1)$, we conclude
    \begin{equation} \label{eq: L2 case}
        \| u\|^2_{L^{2}(\hat{\mathbf{s}})}\leq\frac{2}{\widehat{h}_2}\left(\|u\|^2_{L^{2}(\widehat{K})}+\widehat{h}_2^2\|\partial_2u\|^2_{L^{2}(\widehat{K})}\right)\leq \frac{2\max(1,\kappa \widehat{h}_2)^2}{\widehat{h}_2}\|u\|^2_{H^{1}_{\kappa}(\widehat{K})}.
    \end{equation}

    The formula in \eqref{eq: L2 case} readily extends to $u \in H_\kappa^1(\widehat{K})$ by a density argument.
    Moreover, if we now assume $u \in H_\kappa^{M+1}(\widehat{K})$, it follows
    \begin{equation*}
        | u|^2_{H^{M}_{\kappa}(\hat{\mathbf{s}})}=\frac{1}{\kappa^{2M}}\|\partial_1^M u\|^2_{L^{2}(\hat{\mathbf{s}})}\leq\frac{2\max(1,\kappa \widehat{h}_2)^2}{\widehat{h}_2\kappa^{2M}}\|\partial^M_1u\|^2_{H^{1}_{\kappa}(\widehat{K})} \leq \frac{2\max(1,\kappa \widehat{h}_2)^2}{\widehat{h}_2}\|u\|^2_{H^{M+1}_{\kappa}(\widehat{K})}. \qedhere
    \end{equation*}
\end{proof}
\begin{lemma} \label{lem: Linfty trace inequality}
    Let $u \in H_\kappa^{2}(\widehat{K})$. Then,
    \begin{equation*}
        \| u\|_{L^{\infty}(\partial\widehat{K})}\leq \frac{2\max(1,\kappa \widehat{h}_1,\kappa \widehat{h}_2)^2}{\min(\widehat{h}_1,\widehat{h}_2)}\|u\|_{H^{2}_{\kappa}(\widehat{K})}.
    \end{equation*}
\end{lemma}
\begin{proof}
    By symmetry, it is sufficient to prove the estimate on $\hat{\mathbf{s}}$. Reasoning as in \eqref{eq: lemma trace useful}, one obtains
    \begin{equation*}
        \frac{| u(x,y)|^2}{2}\leq \frac{1}{\widehat{h}_1}\int_0^{\widehat{h}_1}| u(t,y)|^2\,\textup{d}t+\widehat{h}_1\int_{0}^{\widehat{h}_1}|\partial_1 u (t,y)|^2\,\textup{d}t, \qquad (x,y)\in \partial \widehat{K},
    \end{equation*}
    and hence, setting $y=\widehat{h}_2$, it follows
    \begin{equation*}
        \| u\|^2_{L^{\infty}(\hat{\mathbf{s}})}\leq \frac{2}{\widehat{h}_1}\left(\| u\|^2_{L^{2}(\hat{\mathbf{s}})}+\widehat{h}_1^2\|\partial_1 u\|^2_{L^{2}(\hat{\mathbf{s}})}\right).
    \end{equation*}
    Combining this formula with \eqref{eq: L2 case} yields
    \begin{align*}
        \| u\|^2_{L^{\infty}(\hat{\mathbf{s}})}&\leq \frac{4}{\widehat{h}_1\widehat{h}_2}\|u\|^2_{L^2(\widehat{K})}+\frac{4\widehat{h}_1}{\widehat{h}_2}\|\partial_1u\|^2_{L^2(\widehat{K})}+\frac{4\widehat{h}_2}{\widehat{h}_1}\|\partial_2u\|^2_{L^2(\widehat{K})}+4\widehat{h}_1\widehat{h}_2\|\partial_1\partial_2u\|^2_{L^2(\widehat{K})}\\
        &\leq \frac{4}{\min(\widehat{h}_1,\widehat{h}_2)^2}\left(\|u\|^2_{L^2(\widehat{K})}+\max(\kappa\widehat{h}_1,\kappa\widehat{h}_2)^2 |u|^2_{H_\kappa^1(\widehat{K})}+\max(\kappa\widehat{h}_1,\kappa\widehat{h}_2)^4 |u|^2_{H_\kappa^2(\widehat{K})}\right)
        \\
        &\leq \frac{4\max(1,\kappa\widehat{h}_1,\kappa\widehat{h}_2)^4}{\min(\widehat{h}_1,\widehat{h}_2)^2} \|u\|^2_{H_\kappa^2(\widehat{K})},
    \end{align*}
    so that
    \begin{equation*}
        \| u\|_{L^{\infty}(\partial\widehat{K})}^2=\max_{\mathbf{s}\subset \partial \widehat{K}}\| u\|^2_{L^{\infty}(\mathbf{s})}\leq \frac{4\max(1,\kappa\widehat{h}_1,\kappa\widehat{h}_2)^4}{\min(\widehat{h}_1,\widehat{h}_2)^2} \|u\|^2_{H_\kappa^2(\widehat{K})}. \qedhere
    \end{equation*}
\end{proof}

\paragraph{Proofs of high-frequency regime lemma.} \label{proof: lem: d tilde lemma}

Let us now turn to the proof of Lemma~\ref{lem: d tilde lemma}.
We recall from \eqref{eq: shape regularity} that the mesh size is defined as $h=\max(h_1,h_2)$, while the shape parameter is given by $\rho=h/\min(h_1,h_2)\geq 1$.
In what follows, without loss of generality, we assume $h_1\geq h_2$, so that $h_1=h$ and $h_2=h/\rho$. The arguments remain unchanged if $h_1$ and $h_2$ are swapped.

\begin{proof}[Proof of Lemma \textup{\ref{lem: d tilde lemma}}]
Let $\rho\geq1$, and assume that there exist $p,q \in \mathbb{N}^*$ such that $\gcd(p,q)=1$ and $\rho^2=p/q \in \mathbb{Q}$.
In order to prove that \eqref{lem: d tilde lemma equation} holds on an unbounded set $\mathcal{K}\subset (0,+\infty)$, we will exhibit a positive sequence $t_j\nearrow\infty$, for $j\in\mathbb{N}^*$, and an explicit constant $C(q)>0$, such that
\begin{equation*}
    \widetilde{D}(\rho,t_j) =1+\frac{1}{\widetilde{\textrm{d}}(\rho,t_j)}+\frac{1}{\sqrt{\widetilde{\textrm{d}}_0(\rho,t_j)}}\leq C(q) t_j^2,
\end{equation*}
where we recall that, for $t>0$,
\begin{equation*}
    \widetilde{\mathrm{d}}(\rho,t)=\!\!\!\inf_{\substack{m \in \mathbb{N},n \in \mathbb{N}^*:\\\widetilde{\nu}_n(t)<1}}\left|\sqrt{1-\widetilde{\nu}_n^2(t)}-\frac{m\pi\rho}{t}\right|, \qquad 
    \widetilde{\mathrm{d}}_0(\rho,t)=\!\!\!\inf_{\substack{n \in \mathbb{N}^*:\\\widetilde{\nu}_n(t)>1}}\!\!\sqrt{\widetilde{\nu}_n^2(t)-1},
    \qquad \widetilde{\nu}_n(t)=\frac{(2n-1)\pi}{2t}.
\end{equation*}
Let us consider the sorted set
\begin{equation*}
    \Big\{r_1<r_2<\dots\Big\}:=\Bigg\{r\in(0,+\infty):r^2=r^2(n,m):=\left(\frac{(2n-1)\pi}{2}\right)^2+(m\pi\rho)^2, (n,m) \in \mathbb{N}^* \times \mathbb{N}\Bigg\}.
\end{equation*}
For any $j\in\mathbb{N}^*$, define
\begin{equation*}
    t_j:=\frac{r_j+r_{j+1}}{2},
\end{equation*}
and the two sets
\begin{equation*}
    \mathcal{R}_j^-:=\Big\{(n,m)\in \mathbb{N}^* \times \mathbb{N}: r(n,m)\leq r_j\Big\}, \qquad
    \mathcal{R}_j^+:=\Big\{(n,m)\in \mathbb{N}^* \times \mathbb{N}: r(n,m)\geq r_{j+1}\Big\}.
\end{equation*}
For any $j\in\mathbb{N}^*$, observe that $r_j<t_j<r_{j+1}$, and that $\mathbb{N}^*\times \mathbb{N}=\mathcal{R}_j^-\cup \mathcal{R}_j^+$.

Let us first focus on $\widetilde{\mathrm{d}}(\rho,t_j)$. For any $(n,m)\in \mathcal{R}_j^-$, set $r=r(n,m)$. Since $(2n-1)\pi/2\leq r\leq r_j<t_j$,
one has that $\widetilde{\nu}_n(t_j)<1$. Besides,
\begin{equation*}
    m\pi\rho<r\leq r_j<t_j, \qquad \text{and} \qquad t_j^2-r^2> t_j(t_j-r_j)=t_j(r_{j+1}-r_j)/2.
\end{equation*}
Hence, it follows that
\begin{equation} \label{eq: App1}
    \left|\sqrt{t_j^2-\left(\frac{(2n-1)\pi}{2}\right)^2}-m\pi\rho\right|=\frac{t_j^2-r^2}{\sqrt{t_j^2-((2n-1)\pi/2)^2}+m\pi\rho}\geq \frac{t_j^2-r^2}{2t_j}\geq\frac{r_{j+1}-r_j}{4}.
\end{equation}
Similarly, for any $(n,m)\in \mathcal{R}_j^+$ such that $\widetilde{\nu}_n(t_j)<1$, if $r=r(n,m)$,
\begin{equation*}
    \sqrt{t_j^2-((2n-1)\pi/2)^2}<t_j<r_{j+1}\leq r, \qquad \text{and} \qquad r^2-t_j^2> r(r_{j+1}-t_j)=r(r_{j+1}-r_j)/2,
\end{equation*}
thus, one obtains
\begin{equation} \label{eq: App2}
    \left|\sqrt{t_j^2-\left(\frac{(2n-1)\pi}{2}\right)^2}-m\pi\rho\right|=\frac{r^2-t_j^2}{\sqrt{t_j^2-((2n-1)\pi/2)^2}+m\pi\rho}\geq \frac{r^2-t_j^2}{2r}\geq\frac{r_{j+1}-r_j}{4}.
\end{equation}
Assume now that $r_j = r(n_j,m_j)$ for some $(n_j,m_j)\in \mathbb{N}^*\times \mathbb{N}$. Then, by direct computation,
\begin{align*} 
    \frac{r_{j+1}^2-r_j^2}{\pi^2}&=\frac{1}{4}\left[(2n_{j+1}-1)^2-(2n_{j}-1)^2\right]+\rho^2\left[m_{j+1}^2-m_j^2\right]\\
    & = \frac{q(n_{j+1}-n_j)(n_{j+1}+n_j-1)+p\left(m_{j+1}^2-m_j^2\right)}{q}\geq \frac{1}{q}.
\end{align*}
In particular, this implies:
\begin{equation} \label{eq: App3}
r_{j+1}-r_{j}=\frac{r_{j+1}^2-r_j^2}{2t_j} \geq \frac{\pi^2}{2qt_j}, \qquad j\in\mathbb{N}^*.
\end{equation}
Therefore, thanks to the bounds in \eqref{eq: App1}, \eqref{eq: App2}, and \eqref{eq: App3}, one obtains:
\begin{equation} \label{eq: Fin1}
    t_j\widetilde{\mathrm{d}}(\rho,t_j)=\inf_{\substack{(n,m)\in \mathcal{R}_j^-\cup\mathcal{R}_j^+: \\ \widetilde{\nu}_n(t_j)<1}}\left|\sqrt{t_j^2-\left(\frac{(2n-1)\pi}{2}\right)^2}-m\pi\rho\right|\geq\frac{r_{j+1}-r_j}{4}\geq \frac{\pi^2}{8qt_j}.
\end{equation}

Concerning $\widetilde{\mathrm{d}}_0(\rho,t_j)$, observe that the set over which the infimum is taken can be equivalently expressed as
\begin{equation*}
    \Big\{n \in \mathbb{N}^*: \widetilde{\nu}_n(t_j)>1\Big\}=\Big\{n \in \mathbb{N}^*: (n,0)\in \mathcal{R}_j^+\Big\},
\end{equation*}
as, by definition, $r_{j+1}=\inf\{r(n,m)>t_j:(n,m)\in \mathbb{N}^*\times \mathbb{N}\}$.
In addition, for any $n \in \mathbb{N}^*$ such that $(n,0)\in \mathcal{R}_j^+$, one has that
\begin{equation*}
    \left((2n-1)\pi/2\right)^2-t_j^2\geq r_{j+1}^2-t_j^2\geq t_j(r_{j+1}-t_j)=t_j(r_{j+1}-r_j)/2,
\end{equation*}
and so from \eqref{eq: App3} it follows:
\begin{equation} \label{eq: Fin2}
    t_j\widetilde{\mathrm{d}}_0(\rho,t_j)=\inf_{\substack{n \in \mathbb{N}^*:(n,0)\in \mathcal{R}_j^+}}\sqrt{\left(\frac{(2n-1)\pi}{2}\right)^2-t_j^2}\geq \sqrt{t_j\left(\frac{r_{j+1}-r_j}{2}\right)}\geq \frac{\pi}{2\sqrt{q}}.
\end{equation}

Since $t_j>1$ for any $j \in \mathbb{N}^*$, from \eqref{eq: Fin1} and \eqref{eq: Fin2}, one finally obtains:
\begin{equation*}
    \widetilde{D}(\rho,t_j) =1+\frac{1}{\widetilde{\textrm{d}}(\rho,t_j)}+\frac{1}{\sqrt{\widetilde{\textrm{d}}_0(\rho,t_j)}}\leq \left(1+\frac{8q}{\pi^2}+\sqrt{\frac{2\sqrt{q}}{\pi}}\right)t_j^2, \qquad j \in \mathbb{N}^*. \qedhere
\end{equation*}

\end{proof}

\bibliographystyle{plain}
\bibliography{biblio}

\end{document}